\documentclass[11pt,reqno]{amsart}    
\def\MR#1{}
\usepackage[margin=37mm]{geometry}
\usepackage{amsfonts,amsmath,amssymb,enumerate}    
\usepackage{amsthm, bm}    
\usepackage{mathrsfs}
\usepackage{setspace}
\usepackage{comment}
\usepackage{tikz}    
\usetikzlibrary{arrows,matrix}
\usetikzlibrary{decorations.markings,shapes.geometric,shapes.misc,cd} 

\usepackage{mathtools}    
\usepackage{hyperref}

\usepackage[capitalise]{cleveref}
\usepackage{upgreek}
\usepackage{simpler-wick}

\usepackage[hyperref,doi,url=false,
backref,
giveninits=true,
backend=biber,
style=alphabetic,
maxbibnames=99]{biblatex}
\bibliography{bibliography.bib}

\AtEveryBibitem{
 \clearlist{address}
 \clearfield{date}
 \clearfield{isbn}
 \clearfield{issn}
 \clearlist{location}
 \clearfield{month}
 \clearfield{series}
 
 \ifentrytype{book}{}{
  \clearlist{publisher}
  \clearname{editor}
 }
 \ifentrytype{misc}{}{
  \clearfield{eprint}
 }
}

\crefname{equation}{}{}
\Crefname{equation}{}{}
\DeclareFontEncoding{LS1}{}{}
\DeclareFontSubstitution{LS1}{stix}{m}{n}
\DeclareSymbolFont{symbols2}{LS1}{stixfrak}{m}{n}
\DeclareMathSymbol{\typecolon}{\mathbin}{symbols2}{"25}

\newcommand{\nord}[1]{\mathopen: #1 \mathclose:}

\newcommand{\bilin}[2]{\hspace{.5mm}\kappa\hspace{-.5mm}\bigl( #1|#2\bigr)}

\theoremstyle{plain}
\newtheorem{thm}{Theorem}
\newtheorem{lem}[thm]{Lemma}
\newtheorem{prop}[thm]{Proposition}
\newtheorem{cor}[thm]{Corollary}

\newtheorem*{prop*}{Proposition}
\theoremstyle{definition}

\newtheorem{exmp}[thm]{Example}
\theoremstyle{remark}
\newtheorem{rem}[thm]{Remark}

\newtheorem*{rem*}{Remark}

\newcommand{\be}{\begin{equation}}    
\newcommand{\ee}{\end{equation}}    
\newcommand{\beu}{\begin{equation*}}    
\newcommand{\eeu}{\end{equation*}}    
\newcommand{\bea}{\begin{eqnarray}}    
\newcommand{\eea}{\end{eqnarray}}    
\newcommand{\beaa}{\begin{eqnarray*}}    
\newcommand{\eeaa}{\end{eqnarray*}}    
\newcommand{\bmx}{\begin{pmatrix}}    
\newcommand{\emx}{\end{pmatrix}}

\newcommand{\ol}{\overline}    
    
\newcommand{\del}{\partial}    
\newcommand{\g}{{\mathfrak g}}
\renewcommand{\a}{{\mathfrak a}}
\renewcommand{\b}{{\mathfrak b}}

\newcommand{\gh}{{\widehat \g}}

\newcommand{\f}{{\mathfrak f}}
\renewcommand{\c}{{\mathfrak c}}

\newcommand{\gl}{{\mathfrak{gl}}}
\newcommand{\n}{{\mathfrak n}}    
\newcommand{\np}{{\mathfrak n}}    
\newcommand{\nm}{{\mathfrak n_-}}    
\newcommand{\npm}{{\mathfrak n_\pm}}    
\newcommand{\nmp}{{\mathfrak n_\mp}}    
\newcommand{\h}{{\mathfrak h}}

\newcommand{\p}{{\mathfrak p}}

\newcommand{\lpg}[2]{\jota\left(#1\right)\,\vap 1\,\jota\left(#2\right)}

\newcommand{\B}{{\mathsf B}}

\newcommand{\Heis}{\mathsf H}
\newcommand{\Hg}{\Heis}
\newcommand{\Cl}{\mathsf{Cl}}

\newcommand{\mf}{\mathfrak}
\newcommand{\mc}{\mathcal}

\newcommand{\npc}{{\widetilde\np}}

\newcommand{\bc}{{\widetilde \b}}
\newcommand{\gc}{{\widetilde \g}}
\newcommand{\nc}{{\widetilde \n}}

\newcommand{\J}{\mathsf J}
\newcommand{\Jp}{{}^+\!\mathsf J}
\newcommand{\Rp}{{}^+\!R}

\newcommand{\JJ}{\mathbb J}

\newcommand{\dd}[2]{\mathsf d\!\left(#1,#2 \right)}     

\newcommand{\D}{\mathcal{D}}    
\newcommand{\Dg}{\mathcal{D}}    
    
\newcommand{\Dl}{\mathcal{D}_1}    
    
\newcommand{\Dw }{\ol{\Der}\,\Onp}
\newcommand{\Dwm}{\ol{\Der}\,\Onm}
\newcommand{\Dwg}{\ol{\Der}\,\Og}

\newcommand{\Mh}{\mathsf M}
\newcommand{\Wh}{\Lambda}
\newcommand{\Mf}{\bm{\mathsf M}} 
\newcommand{\Mfw}{\ol{\bm{\mathsf M}}}

\newcommand{\Mw}{\ol{\mathsf M}}

\newcommand{\half}{\frac{1}{2}}

\newcommand{\nn}{\nonumber}    
\newcommand{\sign}{{\rm sign}}    
    
\newcommand{\8}{{\infty}}

\newcommand{\eps}{\epsilon}

\newcommand{\ZZ}{{\mathbb Z}}

\newcommand{\CC}{{\mathbb C}}

\newcommand{\Q}{{\mathcal Q}}

\newcommand{\ket}[1]{{\,\left|#1\right>}\,}

\newcommand{\id}{{\textup{id}}}    
\newcommand{\bl}{{\bullet}}    
\newcommand{\wh}{\widehat}

\renewcommand{\gl}{\mf{gl}}

\renewcommand{\sl}{\mf{sl}}

\newcommand{\A}{\mathsf A}
\newcommand{\Am}{-\mathsf A}
\newcommand{\Apm}{\pm\mathsf A}

\newcommand{\Ag}{\I \times \ZZ}

\newcommand{\goi}[2]{=}    
\newcommand{\Hom}{\mathrm{Hom}}
\newcommand{\Homres}{\mathrm{Hom}^{\mathrm{res}}}
\newcommand{\Homcont}{\mathrm{Hom}^{\mathrm{cont}}}

\newcommand{\on}{.}
\newcommand{\npo}{.}

\renewcommand{\binom}[2]{{#1 \brack #2}}

\newcommand{\btp}{\begin{tikzpicture}[baseline=0pt,scale=0.9,line width=0.25pt]}    
\newcommand{\etp}{\end{tikzpicture}}

\renewcommand{\L}{\mathcal{L}}

\newcommand{\scr}{\mathscr}

\newcommand{\wt}{\widetilde}

\DeclareMathOperator{\res}{res}

\DeclareMathOperator{\ad}{ad}    
\DeclareMathOperator{\lad}{\mathit L\! \ad}

\newcommand{\la}{\left<}
\newcommand{\ra}{\right>}

\renewcommand{\O}{\mc O}
\newcommand{\Og}{{\O}}
\newcommand{\Onp}{{\O(\np)}}
\newcommand{\Onm}{{\O(\nm)}}
\newcommand{\Onpm}{{\O(\n_\pm)}}

\newcommand{\Ogp}{\Og \oplus \Omega_\Og}

\newcommand{\VV}{{\mathbb V}}

\newcommand{\vac}{\!\ket{0}\!}

\newcommand{\tox}{\,\mathbin{\widetilde\otimes}\,}

\DeclareMathOperator{\Coind}{Coind}
\DeclareMathOperator{\End}{End}

\DeclareMathOperator{\Der}{Der}

\DeclareMathOperator{\Derc}{\widetilde{Der}}

\newcommand{\ox}{\otimes}

\newcommand{\invlim}{\varprojlim}

\DeclareMathOperator*{\prodr}{\overrightarrow\prod}

\newcommand{\into}{\hookrightarrow}

\newcommand{\onto}{\twoheadrightarrow}
\newcommand{\va}{\mathscr V}
\newcommand{\vla}{\mathscr{L}}

\newcommand{\cent}{\mathsf k}
\newcommand{\cocent}{\mathsf d}
\newcommand{\Cocent}{\mathsf D}

\newcommand{\fm}[1]{[#1]} 

\def\lh{{{}^L\!\h}}

\def\hn{\hat\np}

\def\lh{{}^L\h}

\def\B{\mc B}
\def\H{H}

\newcommand{\Exp}[1]{\mathrm e^{#1}}

\newcommand{\F}{\mc F}

\DeclareMathOperator{\im}{Im}

\newcommand{\cai}{\sigma}

\newcommand{\isom}{\xrightarrow\sim}

\def\short{true}

\DeclareMathOperator{\height}{ht}
\DeclareMathOperator{\wgt}{wgt}

\newcommand{\vap}[1]{{\scriptstyle{\left( #1 \right)}}}

\newcommand{\oc}{\mathring}

\newcommand{\df}[1]{{[\leq\!\!#1]}}
\newcommand{\I}{\mathcal I}

\newcommand{\HH}{\mc H}

\newcommand{\HHc}{\widetilde \HH}

\newcommand{\exxp}[1]{\exp\!\left(#1\right)}
\newcommand{\vf}{\rho}
\newcommand{\vfd}{\uprho}
\newcommand{\vphi}{\upphi}
\newcommand{\bet}{\upbeta}
\newcommand{\gam}{\upgamma}
\newcommand{\bee}{\mathsf b}
\newcommand{\cee}{\mathsf c}

\newcommand{\NN}{\ZZ_{\geq 1}}
\newcommand{\NNm}{\ZZ_{\leq -1}}
\newcommand{\ii}[2]{(#1,#2)}
\newcommand{\ia}[2]{#1,#2}

\newcommand{\imb}{i}

\renewcommand{\SS}{\mathsf S}
\newcommand{\jota}{\jmath}

\newcommand{\Cloc}{C_{\mathrm{loc}}}
\newcommand{\Hloc}{H_{\mathrm{loc}}}
\newcommand{\Lglog}{\gl(\oc\g)[t,t^{-1}]}

\newcommand{\Log}{\oc\g[t,t^{-1}]}
\newcommand{\LLog}{L\!L\oc\g}
\newcommand{\LpLog}{L_{\!+}L\oc\g}
\makeatletter
\newcommand{\extp}{\@ifnextchar^\@extp{\@extp^{\,}}}
\def\@extp^#1{\mathop{\bigwedge\nolimits^{\!#1}}}
\makeatother
\makeatletter
\newcommand{\hextp}{\@ifnextchar^\@hextp{\@hextp^{\,}}}
\newcommand{\nff}{\vartheta}

\def\@hextp^#1{\mathop{\wh{\bigwedge\nolimits^{\!#1}}}}
\makeatother

\author{Charles Young}
\address{
Department of Physics, Astronomy and Mathematics, University of Hertfordshire, College Lane, Hatfield AL10 9AB, UK.}  \email{c.a.s.young@gmail.com}
\date{\today}

\begin{document} 
\title[Analog of Feigin-Frenkel homomorphism for double loop algebras]{An analog of the Feigin-Frenkel homomorphism\\ for double loop algebras}

\begin{abstract}
We prove the existence of a homomorphism of vertex algebras, from the vacuum Verma module over the loop algebra of an untwisted affine algebra, whose construction is analogous to that of the Feigin-Frenkel homomorphism from the vacuum Verma module at critical level over an affine algebra.
\end{abstract}


\maketitle
\setcounter{tocdepth}{1}
\tableofcontents


\section{Introduction and overview}
The goal of this paper is to give an analog of the Feigin-Frenkel homomorphism $\VV_0^{\g,-h^\vee}\to \Mh(\np)$ in the case in which $\g$ is of untwisted affine type. To set the scene, we should first recall the situation in finite types.


\subsection{}\label{sec: finitetype}
The Lie algebra $\sl_2$ (over $\CC$) has a realization in terms of first order differential operators:
\begin{align}
E\mapsto D\qquad
H\mapsto -2XD\qquad 
F\mapsto -XXD.\qquad \label{bigcellsl2}
\end{align}
Here $E,F,H$ are the Chevalley-Serre generators and $X,D$ are generators of a Weyl algebra with commutation relations $[D,X] =1$. 
At the heart of the Wakimoto construction \cite{Wakimoto} is the observation that this homomorphism of Lie algebras can be promoted to a homomorphism of vertex algebras, given 
by
\begin{alignat}{2}
E\fm{-1}\vac &\mapsto &\phantom+\bet\fm{-1}&\vac\nn\\
H\fm{-1}\vac &\mapsto &-2 \gam\fm 0\bet\fm{-1}&\vac\nn\\
F\fm{-1}\vac &\mapsto &- \gam\fm 0\gam\fm 0 \bet\fm{-1}&\vac -2 \gam\fm{-1}\vac, \label{sl2wak}
\end{alignat}
from the vacuum Verma module over $\wh\sl_2$ at the critical level, to the vacuum Fock module for a $\bet\gam$-system of free fields. Note the new feature, the term $-2\gam\fm{-1}\vac$. 

This is a special case of a construction which works for any finite-dimensional simple Lie algebra $\g =_\CC \nm\oplus\h\oplus \np$. 
The realization \cref{bigcellsl2} generalizes to a homomorphism 
\be \vf: \g \to \Der\Onp;\qquad A\mapsto \sum_{\alpha\in \Delta_+} P^\alpha_A(X) D_\alpha \label{bigcell}\ee 
from $\g$ 
to the Lie algebra $\Der\Onp$ of derivations of the algebra 
$\Onp = \CC[X^{\alpha}]_{\alpha\in \Delta_+}$ 
of polynomial functions on the unipotent group $U = \exp(\np) \cong \np$. 
This realization arises from the infinitesimal action of $\g$ on a flag manifold, $B_-\backslash G$, whose big cell is diffeomorphic to $U$.

Recall (from e.g. \cite{KacVertexAlgBook,FrenkelBenZvi}) that the vacuum Verma module $\VV_0^{\g,k}$ over $\gh$ at level $k\in \CC$ is generated as a vertex algebra by states $\{A\fm{-1}\vac: A\in \g\}$, whose non-zero non-negative products 
(i.e., whose OPEs) 
are given by 
\begin{alignat}{2}  A\fm{-1}\vac\,\, \vap 0\,\, B\fm{-1}\vac &= &[A,B]\fm{-1}&\vac, \nn\\
 A\fm{-1}\vac \,\,\vap 1\,\, B\fm{-1}\vac &=\,\,& k \bilin{A}{B} &\vac. \label{Vkdef}
\end{alignat}
Here $\bilin\cdot\cdot$ is the invariant symmetric bilinear form on $\g$ normalized as in \cite{KacBook}. With this normalization, the critical level is equal to $-h^\vee$, where $h^\vee$ is  the dual Coxeter number of $\g$. 

Let $\Mh(\np)$ be the vacuum Fock module for the $\bet\gam$-system on $\np \cong U$. It is generated as a vertex algebra by states $\bet_\alpha\fm{-1}\vac$ and $\gam^\alpha\fm 0 \vac$, $\alpha\in \Delta_+$, obeying
\begin{align}
\bet_\alpha\fm{-1}\vac\,\,\vap 0\,\, \gam^\beta\fm0\vac = \delta_\alpha^\beta \vac .\nn
\end{align} 
(See \cref{sec: va} for the details.) 

Both $\VV_0^{\g,k}$ and $\Mh(\np)$ have natural $\ZZ_{\geq 0}$-gradations (by \emph{depth}) and the first two graded subspaces are 
\begin{alignat}{2} 
\VV_0^{\g,k}[0] &\cong \CC\qquad &\Mh(\np)[0] &\cong \Onp\nn\\
\VV_0^{\g,k}[1] &\cong \g \qquad &\Mh(\np)[1] &\cong \Der\Onp \oplus \Omega_\Onp,\label{firsttwo}
\end{alignat}
where $\Omega_\Onp = \Hom_\Onp(\Der\Onp,\Onp)$ is the space of one-forms.
One identifies $\vac \simeq 1$ and $A\fm{-1}\vac \simeq A\in \g$; and $\gam^\alpha\fm 0 \simeq X^\alpha$, $\gam^\alpha\fm{-1} \simeq dX^\alpha$ and $D_\alpha \simeq \bet_\alpha\fm{-1}$.  
Given these identifications, the homomorphism $\vf:\g \to \Der\Onp$ gives rise to a graded linear map 
\begin{subequations}
\be \VV_0^{\g,k}\df 1 \to \Mh(\np)\df 1,\nn\ee 
sending $\vac \to \vac$ and 
\be A\fm{-1}\vac \mapsto \sum_{\alpha\in \Delta_+} P^\alpha_A(\gam\fm0) \bet_\alpha\fm{-1}\vac.\nn\ee
\end{subequations}
This map does not preserve the non-negative products, but the result of Feigin and Frenkel \cite{FF1990}, \cite{FF1988,FF1990b} \cite{Frenkel_2005,Fre07} is that,  at the critical level $k= -h^\vee$, it may be lifted to one which does. Namely, 
there exists a linear map 
\be \phi:\g \to \Omega_\Onp; \qquad A\mapsto \sum_{\alpha\in \Delta_+} Q_{\alpha,A}(X) dX^\alpha \nn\ee
such that the graded linear map 
\begin{subequations}\label{fftoptwo}
\be \VV_0^{\g,-h^\vee}\df 1 \to \Mh(\np)\df 1\ee 
associated to $\vf+\phi:\g \to \Der\Onp \oplus \Omega_\Onp$, i.e. the one sending $\vac \to \vac$ and 
\be A\fm{-1}\vac \mapsto \sum_{\alpha\in \Delta_+} P^\alpha_A(\gam\fm0) \bet_\alpha\fm{-1}\vac
+ \sum_{\alpha\in \Delta_+} Q_{\alpha,A}(\gam\fm0) \gam^\alpha\fm{-1}\vac,\ee
\end{subequations}
does preserve the non-negative products. 
This latter map \cref{fftoptwo} is the restriction of a homomorphism of graded vertex algebras, 
\be \VV_0^{\g,-h^\vee} \to \Mh(\np).\nn\ee

The map $\phi:\g \to \Omega_\Onp$ respects the weight gradation. In particular $\phi(\h\oplus\np) = 0$ on grading grounds.

For example in the case of $\sl_2$, $\phi(f) = -2 dX$, $\phi(e) = \phi(h) = 0$, as in \cref{sl2wak}.

Various perspectives on this important result have subsequently appeared in the literature \cite{deBF}, \cite{FF1999},\cite{FG1}, \cite{ACM} In particular, see \cite{GMShomog} for an interpretation the language of vertex algebroids and chiral algebras  \cite{MSV,GMSii}, \cite{
GMSi,
AG, 
GMSiii,
BDChiralAlgebras,Malikov_2017}

\subsection{}
Now, and for the rest of this paper, let us suppose instead that $\g$ is of untwisted affine type, i.e. that
\be \g \quad\cong_\CC\quad \oc \g[t,t^{-1}] \oplus \CC\cent \oplus \CC\cocent \nn\ee
for some finite-dimensional simple Lie algebra $\oc \g$. 
(Here $\cent$ is central and $\cocent = t\del_t$ is the derivation element corresponding to the homogeneous gradation.)

We still have the Cartan decomposition $\g =_\CC \nm \oplus \h \oplus \np$. The Lie subalgebra $\np = \bigoplus_{\alpha \in \Delta_+} \np_\alpha$ is now of countably infinite dimension, and no longer nilpotent. But its completion $\npc = \prod_{\alpha\in \Delta_+} \np_\alpha$ is a pro-nilpotent pro-Lie algebra (i.e. a certain inverse limit of nilpotent Lie algebras -- see  \cite{Kumar}, and \cref{sec: npc} below) and there is still a bijective exponential map 
\be \exp: \npc \xrightarrow{\,\sim\,} U\nn\ee 
to a group $U$, which is now a pro-unipotent pro-group. We shall fix (in \cref{sec: coords}) a convenient choice of coordinates on $U$,
\be X^{a,n}: U \to \CC.\nn\ee 
Here $\ii an$ runs over a countable index set, $\A$, which also indexes a topological basis $J_{a,n}$ of $\npc$. 
(Recall $\dim(\np_\alpha)$ can be greater than 1 in affine types other than $\wh\sl_2$, so we cannot simply index by the positive roots $\Delta_+$.) 

We set $\Onp := \CC[X^{a,n}]_{\ii a n\in \A}$ and define $\Der\Onp$ to be the Lie algebra of derivations of $\Onp$ consisting of sums of the form
\be \sum_{\ii a n \in \A} P^{\ia a n}(X) D_{\ia a n},\qquad P^{a,n}(X) \in \Onp,\nn\ee
subject to the constraint that only finitely summands are nonzero. It has a completion, $\Derc\Onp\supset \Der\Onp$, consisting of sums of the same form but without the constraint.
(Here the new generators $D_{\ia an}$ obey $[D_{\ia an},X^{\ia b m}] = \delta_{\ia an}^{\ia bm}$.)
As before, we define the space of one-forms $\Omega_\Onp = \Hom_\Onp(\Der\Onp,\Onp)$.

The group $U$ can still be seen as copy of the big cell of a flag manifold $B_-\backslash G$, in a sense made precise in \cite{KashiwaraFlag}.\footnote{It is perhaps worth stressing that $B_-\backslash G$ is not the \emph{affine Grassmannian} or the \emph{affine flag variety} in the usual sense of e.g. \cite[\S2.2]{Gortz}. For example when $\g=\wh\sl_2$, the (set of $\CC$-points of the) affine Grassmannian is $SL_2(\CC[[t]]) \big\backslash SL_2(\CC((t)))$, whereas here $B_-= SL_2(t^{-1}\CC[t^{-1}]) \oc B_-$ with $\oc B_-$ the usual lower-triangular Borel subgroup of $SL_2(\CC)$. See also the discussion in  \cite[\S5]{Fopersontheprojectiveline} (in which one should swap $t\leftrightarrow t^{-1}$ to match the present conventions).}
For our purpose the important point is that there is a homomorphism of Lie algebras,
\be \vf:\g \to \Derc\Onp; \qquad A\mapsto \sum_{\ii a n\in \A} P^{\ia a n}_A(X) D_{\ia a n} \label{affinebigcell}\ee
as we show in a concrete fashion in \cref{sec: inft}. This is the analog of the homomorphism \cref{bigcell}. 

Some examples in the case $\g = \wh\sl_2$ are shown in \cref{sl2diffs}.

(The centre of $\g$ lies in the kernel, $\vf(\cent) = 0$, so the homomorphism $\vf$ actually factors through $\g/\CC\cent \cong \oc \g[t,t^{-1}] \rtimes \CC\cocent$.)

\begin{figure}
\begin{align}
\vf(J_{E,0}) &= D_{E,{0}} +2X^{H,{1}} D_{E,{1}} -X^{F,{1}}D_{H,{1}}\nn\\ 
&\quad +\left(2X^{H,{2}}-2\left(X^{H,{1}}\right)^{2}\right)D_{E,{2}}
-X^{F,{2}}D_{H,{2}}+\left(X^{F,{1}}\right)^{2}D_{F,{2}}\nn\\
&\quad +\left(2X^{H,{3}}+\tfrac{4}{3}\left(X^{H,{1}}\right)^{3}\right)D_{E,{3}}
+\left(-X^{F,{3}}-X^{E,{1}}\left(X^{F,{1}}\right)^{2}\right)D_{H,{3}}
+2\left(X^{F,{1}}\right)^{2}X^{H,{1}}D_{F,{3}}\nn\\
& \quad +\dots\nn\\
\nn\\
\vf(J_{F,1}) &= 
D_{F,{1}}
+X^{E,{1}}D_{H,{2}}
-2X^{H,{1}}D_{F,{2}}\nn\\
&\quad+\left(X^{E,{1}}\right)^{2}D_{E,{3}}
+\left(X^{E,{2}}+2X^{E,{1}}X^{H,{1}}\right)D_{H,{3}}
+\left(-2X^{H,{2}}-2\left(X^{H,{1}}\right)^{2}\right)D_{F,{3}}\nn\\
&\quad + \dots\nn\\
\nn\\
\vf(J_{E,-1}) &=
\left(2X^{H,{1}}+2X^{E,{0}}X^{F,{1}}\right)D_{E,{0}}\nn\\
&\quad
+\left(2X^{H,{2}}+2\left(X^{H,{1}}\right)^{2}\right)D_{E,{1}}
+\left(-X^{F,{2}}-2X^{F,{1}}X^{H,{1}}\right)D_{H,{1}}
-\left(X^{F,{1}}\right)^{2}D_{F,{1}}\nn\\
&\quad
+\left(2X^{H,{3}}+2X^{E,{1}}X^{F,{2}}-\tfrac{8}{3}\left(X^{H,{1}}\right)^{3}\right)D_{E,{2}}
-X^{F,{3}}D_{H,{2}}
+2\left(X^{F,{1}}\right)^{2}X^{H,{1}}D_{F,{2}}\nn\\
&\quad+ \dots\nn\\
\nn\\
\vf(J_{F,0}) &=
-\left(X^{E,{0}}\right)^{2}D_{E,{0}}
+X^{E,{1}}D_{H,{1}}
-2X^{H,{1}}D_{F,{1}}\nn\\
&\quad-\left(X^{E,{1}}\right)^{2}D_{E,{2}}
+X^{E,{2}}D_{H,{2}}
+\left(-2X^{H,{2}}+2\left(X^{H,{1}}\right)^{2}\right)D_{F,{2}}\nn\\
&\quad+\left(X^{E,{3}}-2X^{E,{1}}\left(X^{H,{1}}\right)^{2}\right)D_{H,{3}}
+\left(-2X^{H,{3}}+\tfrac{8}{3}\left(X^{H,{1}}\right)^{3}\right)D_{F,{3}}\nn\\
&\quad + \dots\nn
\end{align}
\caption{\label{sl2diffs}
In type $\g = \wh\sl_2$, the first few terms of the images of the Chevalley-Serre generators $e_1 = J_{E,0}$, $e_0 = J_{F,1}$, $f_1 = J_{F,0}$,  $f_0 = J_{E,-1}$ under the homomorphism $\vf: \g \to \Derc\Onp$.
}
\end{figure}

\subsection{}
One may define the vacuum Verma module $\VV_0^{\g,k}$ over $\gh$ at level $k\in \CC$ when $\g$ is affine, such that \cref{Vkdef} still holds, where $\bilin\cdot\cdot$ is the standard non-degenerate symmetric invariant bilinear form from \cite{KacBook} (with $\bilin\cent\cocent =1$ and so on). It is still a vertex algebra. 

The main result of the present paper (\cref{homcor}) is that the homomorphism $\vf$ can be promoted to a homomorphism of vertex algebras 
$\VV_0^{\g,0} \to \Mfw$.
Of course, we have yet to explain what $\Mfw$ is. To motivate its definition, it is instructive to consider what happens when one attempts to generalize the construction above in the most direct fashion.

Thus, let $\Mh(\np)$ be, again, the vacuum Fock module for the $\bet\gam$-system on $\np\cong U$. It is a vertex algebra, generated by (now, countably infinitely many) states $\bet_{\ia a n}\fm{-1}\vac$ and $\gam^{\ia a n}\fm 0 \vac$, $\ii a n \in \A$, obeying
\begin{align}
\bet_{\ia a n}\fm{-1}\vac\,\,\vap 0\,\, \gam^{\ia b m}\fm0\vac = \delta_{\ia a n}^{\ia b m} \vac .\nn
\end{align} 
Both $\VV_0^{\g,k}$ and $\Mh(\np)$ are once more $\ZZ_{\geq 0}$-graded by depth, and the identifications in \cref{firsttwo} continue to hold. 
(One now identifies $\gam^{\ia a n}\fm 0 \simeq X^{\ia a n}$, $\gam^{\ia a n}\fm{-1} \simeq dX^{\ia a n}$ and $D_{\ia a n} \simeq \bet_{\ia a n}\fm{-1}$.)  
In particular, 
\be \Mh(\np)[1] \cong \Der\Onp \oplus \Omega_{\Onp} .\nn\ee
Importantly, it is $\Der\Onp$ and not its completion $\Derc\Onp$ which appears here. Indeed, by definition, $\Mh(\np)$ consists of \emph{finite} linear combinations of states of the form 
$\gam^{\ia{a_1}{n_1}}[-N_1]\dots \gam^{\ia{a_r}{n_r}}[-N_r] 
                \bet_{\ia{b_1}{m_1}}[-M_1]\dots \bet_{\ia{b_s}{m_s}}[-M_s] \vac.$ 
Thus, in contrast to \cref{sec: finitetype}, the image $\vf(\g)\subset \Derc\O$ does not naturally embed in $\Mh(\np)[1]$. 
We need some larger space. 
We shall introduce a completion $\wt\Mh(\np)$ of $\Mh(\np)$ as a vector space, in which certain infinite linear combinations are allowed provided they truncate to finite linear combinations when $\bet_{\ia a n}[N]$ is set to zero \emph{for large $n$}. (See \cref{sec: wtMh}.
) The definition is chosen to ensure that
\be \wt\Mh(\np)[1] \cong \Derc\Onp \oplus \Omega_{\Onp} .\nn\ee
We get the graded linear map 
\be \VV_0^{\g,k}\df 1 \to \wt\Mh(\np)\df 1\nn\ee
sending $\vac \to \vac$ and 
\be A\fm{-1}\vac \mapsto \sum_{\ii a n\in \A} P^{\ia a n}_A(\gam\fm0) \bet_{\ia a n}\fm{-1}\vac,\nn\ee
with $P^{a,n}_A(X)$ the polynomials from \cref{affinebigcell}. 

\subsection{}At this point, naively, one would like to ask whether this map gives rise to a homomorphism of vertex algebras in the same way as in \cref{fftoptwo} above. In that direction, we shall establish the following statement. (It will actually be a corollary, \cref{zeromodethm}, of our main result.) There exists a linear map 
\be \phi:\g \to \Omega_\Onp; \qquad A\mapsto \sum_{(a,n)\in \A} Q_{a,n;A}(X) dX^{a,n} \label{affphi}\ee
such that the graded linear map 
\be \VV_0^{\g,k}\df 1 \to \Mh(\np)\df 1\nn\ee 
associated to $\vf+\phi:\g \to \Derc\Onp \oplus \Omega_\Onp$, i.e. the one sending $\vac \to \vac$ and 
\be A\fm{-1}\vac \mapsto \sum_{(a,n) \in \A} P^{a,n}_A(\gam\fm0) \bet_{a,n}\fm{-1}\vac
+ \sum_{(a,n)\in \A} Q_{{a,n};A}(\gam\fm0) \gam^{a,n}\fm{-1}\vac,\nn\ee
does preserve \emph{at least the $0$th} vertex algebra product (for any $k\in \CC$). That is, if we call this latter map $\nff$, we have
\be \nff(A) \vap 0 \nff(B) = \nff(A\vap 0 B) \label{zeroeq}\ee
for all $A\simeq A\fm{-1}\vac$ and $B\simeq B\fm{-1}\vac$ in $\g\cong \VV_0^{\g,k}[1]$. 

The map $\phi$ again respects the weight gradation, so that $\phi(\h\oplus \np) = 0$.

For example, in type $\g = \wh\sl_2$, when using the same choice of coordinates on $U$ as in \cref{sl2diffs} one finds that
\begin{subequations}\label{hls}
\be \phi(f_1=J_{F,0}) = -2 dX^{E,0},\qquad \phi(f_0 = J_{E,-1}) = -4 dX^{F,1},\nn\ee
and then  
\begin{align} 
    \phi(J_{H,-1}) &= -12 dX^{H,1} - 4 X^{F,1} dX^{E,0}\nn\\ 
    \phi(J_{F,-1}) &= -8 dX^{E,1} + 4 X^{H,1} dX^{E,0} \nn\\
    \phi(J_{E,-2}) &= -10 dX^{F,2} -8 X^{H,1} dX^{F,1} +2 (X^{F,1})^2 dX^{E,0}, \nn
\end{align}
and so on.
\end{subequations}

While encouraging, this statement skirts around a serious caveat, which is the reason the question above was naive: the completion $\wt\Mh(\np)$ is not a vertex algebra. Or, more precisely, the vertex algebra structure on $\Mh(\np)$ does not extend to $\wt\Mh(\np)$. In particular, when one tries to extend the definition of the vertex algebra $n$th products $\vap n: \Mh(\np) \times \Mh(\np) \to \Mh(\np)$ by bilinearity to $\wt\Mh(\np)\times \wt\Mh(\np)$, the results are not in general finite.

Thus, it was a non-trivial fact about the image of $\nff$ that the expression $\nff(A) \vap 0 \nff(B)$ in \cref{zeroeq} was even well-defined. And one finds the would-be $1$st products $\nff(A) \vap 1 \nff(B)$ are not in general finite. 

\subsection{}\label{sec: zeta} 
As an aside, let us examine some examples of these divergences in type $\g=\wh\sl_2$. 
\emph{For this subsection only}, we modify the OPEs by introducing a formal variable $z$ which will serve as a regulator:
\begin{align}
\bet_{\ia a n}\fm{-1}\vac\,\,\vap 0\,\, \gam^{\ia b m}\fm0\vac = z^n \delta_{\ia a n}^{\ia b m} \vac.\nn
\end{align} 

One has 
\be \nff(J_{H,0}) = - 2 \sum_{n\geq 0} \gam^{E,n}[0] \bet_{E,n}[-1] \vac 
 + 2 \sum_{n\geq 1} \gam^{F,n}[0] \bet_{F,n}[-1] \vac .\nn\ee 
It follows that the regulated $1$st product $\nff(J_{H,0})\vap 1 \nff(J_{H,0})$ is the formal series
\be \nff(J_{H,0})\vap 1 \nff(J_{H,0}) = \Bigl(-4 - 4 \sum_{n\geq 1} z^{2n}\Bigr)\vac.\nn\ee
(We recall the standard details of computing such products in \cref{sec: va}.)
As more intricate examples, one finds
\begin{align} 
\nff(J_{E, 1})\vap 1 \nff(J_{F,-1}) &=  \Bigl(-8z - 4 \sum_{n\geq 1} z^{2n+3}\Bigr)\vac \nn\\
\nff(J_{H, 1})\vap 1 \nff(J_{H,-1}) &=  \Bigl(-12z - 4z^3 - 8 \sum_{n\geq 1} z^{2n+3}\Bigr)\vac \nn
\\ \nff(J_{E,-2})\vap 1 \nff(J_{F,2}) &=  \Bigl(-10z^2  - 4 \sum_{n\geq 1} z^{2n+4}\Bigr)\vac \nn
.\end{align} 
These series are divergent when one attempts to remove the regulator by setting $z=1$.
One might be tempted to treat these divergences by $\zeta$-function regularization. For an introduction to the formal-variable approach to $\zeta$-function regularization, see \cite{LepowskyZeta} (and cf. also \cite{Bloch,LepowskyZeta2,DLM1}). It amounts to the following prescription. First, one notes that each series above is the small-$z$ expansion of some rational expression in $z$. One substitutes $z=\Exp y$ in that rational expression, to obtain a rational expression in $\Exp y$; then one expands $\Exp y$ as a formal series in $y$. The result is a quotient of formal series in $y$, and hence a well-defined formal Laurent series in $y$. Finally, one extracts the constant term in that series.

Very suggestively, when one does that, the result is zero in each example above. For instance for $\nff(J_{E,-2})\vap 1 \nff(J_{F,2})$ one obtains 
$ -10 - 4 \left(-\frac{5}{2}\right) = 0$.
In what follows we shall use a different approach,
\ifdefined\short
\else
\footnote{
The reason is that these divergences are not the only obstacle to making  $\wt\Mh(\np)$ into a vertex algebra. Here is another: the vertex algebra axioms include
\be A\vap 1\left( B \vap 0C\right) = \left( A \vap 1 B\right)\vap 0 C + B\vap 0 \left( A \vap 1 C\right) + \left( A \vap 0 B\right) \vap 1 C,
\label{rel10}\ee
but consider the following elements of $\wt\Mh(\np)$, in the case $\g=\wh\sl_2$: 
\be A = \sum_{n\geq 0} \gam^{\ia E n}\fm 0 \bet_{\ia E n}\fm{-1}\vac\quad
    B = \sum_{n\geq 1} \gam^{\ia E {n-1}}\fm 0 \bet_{\ia E n}\fm{-1}\vac ,\quad 
    C = \sum_{n\geq 0} \gam^{\ia E {n+1}}\fm 0 \bet_{\ia E n}\fm{-1}\vac .\nn\ee
Then every term on the right in \cref{rel10} is zero, yet on the left we have $A\vap 1 \left( B\vap 0 C\right) = A\vap 1 \left(\gam^{\ia E 0}\fm 0 D_{\ia E 0}\fm{-1}\vac\right) = - 1$. The underlying problem here is that the commutator ($B\vap 0 C$) of two traceless semi--infinite matrices ($B$ and $C$) need not be traceless. (When there is an upper cut-off $N$ on the index $n$, one has instead $B\vap 0 C =   \gam^{\ia E 0}\fm 0 \bet_{\ia E 0}\fm{-1}\vac - \gam^{\ia E N}\fm 0 \bet_{\ia E N}\fm{-1}\vac$ and then $A\vap 1(B\vap 0 C) = -1 +1 = 0$, and \cref{rel10} holds.)
} 
\fi
but it will indeed be the case that the homomorphism we construct is from the vacuum Verma module at level zero.

\subsection{} To proceed, we need more information about the image of the homomorphism $\vf:\g \to \Derc\Onp$ from \cref{affinebigcell}.
We illustrate the idea with an example in type $\g=\wh\sl_2$. Let us consider a term in the infinite sum $\vf(e_1) = \sum_{(a,n) \in \A} P^{a,n}_{e_1}(X) D_{a,n}$ for the generator $e_1 = J_{E,0}$: say, the term $P^{E,8}_{e_1}(X) D_{E,8}$. One finds
\begin{align}
P_{e_1}^{E,8}(X) &= 2 X^{H,8} \nn\\
&      - 2 \left(X^{H,4}\right)^2
       + 4 \left(X^{H,2}\right)^2 X^{H,4}
       - \tfrac 23 \left(X^{H,2}\right)^4
       - 8 \left(X^{H,1}\right)^2 X^{H,2} X^{H,4}\nn\\
&       + \tfrac 83 \left(X^{H,1}\right)^2 \left(X^{H,2}\right)^3
       + \tfrac{16}{3} \left(X^{H,1}\right)^3 X^{H,2} X^{H,3}
       + \tfrac{4}{3} \left(X^{H,1}\right)^4 X^{H,4}
       - \tfrac{4}{3} \left(X^{H,1}\right)^4 \left(X^{H,2}\right)^2\nn\\
&       - \tfrac{8}{15} \left(X^{H,1}\right)^5 X^{H,3}
       + \tfrac{8}{45} \left(X^{H,1}\right)^6 X^{H,2}
       - \tfrac{2}{315} \left(X^{H,1}\right)^8
       + \tfrac{8}{3} X^{E,3} X^{F,2} \left(X^{H,1}\right)^3\nn\\
&       - \tfrac{4}{3} X^{E,2} X^{F,2} \left(X^{H,1}\right)^4
       - 4 X^{E,2} X^{E,3} \left(X^{F,1}\right)^2 X^{H,1}
       - \left(X^{E,2}\right)^2 \left(X^{F,2}\right)^2\nn\\
&       - 2 \left(X^{E,2}\right)^2 \left(X^{F,1}\right)^2 \left(X^{H,1}\right)^2
       - 4 X^{E,1} X^{E,3} \left(X^{F,1}\right)^2 \left(X^{H,1}\right)^2
       + 4 X^{E,1} X^{E,2} \left(X^{F,1}\right)^2 X^{H,3}\nn\\
&       - \tfrac 83 X^{E,1} X^{E,2} \left(X^{F,1}\right)^2 \left(X^{H,1}\right)^3
       + 2 \left(X^{E,1}\right)^2 \left(X^{F,1}\right)^2 X^{H,4}
       - 2 \left(X^{E,1}\right)^2 \left(X^{F,1}\right)^2 \left(X^{H,2}\right)^2\nn\\
&       + 4 \left(X^{E,1}\right)^2 \left(X^{F,1}\right)^2 X^{H,1} X^{H,3}
       + 4 \left(X^{E,1}\right)^2 \left(X^{F,1}\right)^2 \left(X^{H,1}\right)^2 X^{H,2}\nn\\
&       - \tfrac 23 \left(X^{E,1}\right)^2 \left(X^{F,1}\right)^2 \left(X^{H,1}\right)^4
 - 2 \left(X^{E,1}\right)^2 X^{E,2} \left(X^{F,1}\right)^2 X^{F,2}.\nn
\end{align}
Observe that only the first monomial has any factor $X^{a,n}$ with $n> 4$. 
This is an example of a general pattern: for any fixed  $A\in\g$, in the coefficient polynomials $P^{a,n}_A(X)$ there is, for large $n$, always at most one leading monomial $X^{b,m}$ with $m\sim n$; the remaining monomials have only factors $X^{c,p}$ with $p \lesssim n/2$. 

We shall make this idea precise with the notion of \emph{widening gap} in \cref{sec: Dw}. See \cref{linthm}, which will show that the difference  
\be \vf(J_{a,n}) - \sum_{b,c\in \I} f_{ba}{}^c \sum_{m > \max(1,n) } X^{b,m-n} D_{c,m}, \label{wgi}\ee
has widening gap. (Here $f_{ba}{}^c$ are structure constants of $\oc\g$.) Elements of widening gap form a Lie algebra, $\Dw$, with $\Der\Onp\subset \Dw \subset\Derc\Onp$.

The notion of widening gap goes over to the vertex algebra $\Mh(\np)$: one can define a subspace $\Mw(\np)$ of the completion $\wt\Mh(\np)$ in which infinite sums are allowed but only if they have widening gap. The vertex algebra structure on $\Mh(\np)$ does extend to $\Mw(\np)$. (See \cref{lem: Mw}.)

\subsection{}\label{sec: vfintro}
The question therefore becomes: what to do with the leading terms in $\vf(J_{a,n})$? 
Our approach, in \cref{sec: cai}, will be to glue together two copies of the realization \cref{affinebigcell} back-to-back. Let $ \Og := \O(\g) := \CC[X^{\ia a n}]_{\ii a n \in \Ag} $ denote the algebra of polynomial functions on all of $\g$. We shall define $\Derc\Og$ and its subalgebra, $\Dwg$, of elements of widening gap. The Cartan involution $\cai:\g\to\g$, which exchanges $\np$ and $\nm$, gives rise to an involution $\tau: \Derc\Og\to \Derc\Og$, which exchanges $\Derc\Onp$ and $\Derc\Onm$.
On twisting the homomorphism $\vf : \g \into \Derc\Onp \subset$ from \cref{affinebigcell} by these involutions, we get a homomorphism $\tau\circ\vf\circ\cai: \g \into \Derc\Onm$. Adding the two, we obtain a homomorphism
\begin{align} \vfd := (\vf + \tau\circ\vf\circ\cai) : \g 
&\to \Derc\Onp \oplus \Derc\Onm\nn\\&\into \Derc\Og.\nn
\end{align}
Of course, having helped oneself to a copy of $\O=\O(\g)$, there is an obvious homomorphism $\g\to \Derc\Og$ coming from the coadjoint representation, which sends $J_{a,n}\mapsto\sum_{b,c\in \I} f_{ba}{}^c \sum_{m\in \ZZ}  X^{b,m-n} D_{c,m}$. So one should keep in mind that what is special about the homomorphism $\vfd$ is that, by construction, the resulting action of $\g$ on $\Og$ stabilizes $\Onp$ and $\Onm$.

We then check that the difference
\be \vfd(J_{a,n}) - \sum_{b,c\in \I} f_{ba}{}^c \sum_{m\in \ZZ}  X^{b,m-n} D_{c,m} \nn\ee
has widening gap, i.e. belongs to $\Dwg$. The advantage of this statement, compared to \cref{wgi}, is that we can replace the sums, $\sum_{m\in \ZZ}  X^{b,m-n} D_{c,m}$, by an abstract set of generators $\SS^b_{c,n}$ of the loop algebra $\gl(\oc\g)[t,t^{-1}]$. In this way we can, and shall, regard $\vfd$ as a homomorphism from $\g$ to the Lie algebra 
\be \Dg := \Dwg \rtimes \left(\Lglog \rtimes \CC\Cocent\right) \label{Dgintro}\ee
(here $\Cocent$ is a derivation element, and $\g\ni\cocent\mapsto\Cocent$).
See \cref{lem: Jdif} and the discussion following.

\subsection{} 
We define a vertex algebra $\Mfw$ in light of this definition of the Lie algebra $\Dg$. Namely, we have $\Mh := \Mh(\g)$, the vacuum Fock module for the $\bet\gam$-system on $\g$, and we introduce its completion as a vector space, $\wt\Mh$, and the subspace generated by elements of widening gap, $\Mw\subset \wt\Mh$. 
Then $\Mw$ is a vertex algebra, and we can take a ``semi-direct product of vertex algebras'' with the level zero vacuum Verma module $\VV_0^{\Lglog\rtimes \CC\Cocent,0}$ to define $\Mfw$; so, as a vector space,
\be \Mfw \cong_\CC \Mw \ox \VV_0^{\Lglog\rtimes\CC\Cocent,0}. \nn\ee
See \cref{sec: Mfw}. By construction we have
\begin{align} \Mfw[0] &\cong \Og \nn\\
              \Mfw[1] &\cong \Dg \oplus  \Omega_\Og .\nn\end{align}
At that point we shall be in a position to state our main result: see \cref{mainthm} and \cref{homcor}. 
It says the following: let $\vphi = \phi + \tau \circ\phi \circ\sigma : \g \to \Omega_\Og$
where $\phi:\g\to \Omega_\Onp\subset \Omega_{\Og}$ is the map from \cref{affphi}. Then the graded linear map 
\be \VV_0^{\g,0}\df 1 \to \Mfw\df 1\nn\ee 
associated to $\vfd+\vphi:\g \to \Dg \oplus \Omega_\Og$ (in the same fashion as above) preserves the non-negative products, and is the restriction of a homomorphism of graded vertex algebras
\be \theta: \VV_0^{\g,0} \to \Mfw.\label{thetadef}\ee

Associated to this homomorphism of vertex algebras is a homomorphism of Lie algebras
\be L\g \to \L(\Mfw\df 1)\label{Lghom}\ee 
from the loop algebra $L\g := \g \ox \CC((s))$ of the affine algebra $\g$ to the Lie algebra of formal modes of states in $\Mfw\df1$. In fact, the central element $\cent\in \g$ is in the kernel of $\vf$, so we actually get a homomorphism
\be \LLog \to \L(\Mfw\df1) \nn\ee
from the double-loop algebra $\LLog := \oc\g[t,t^{-1}]\ox \CC((s))$.

The homomorphism $\theta$ has the property that the non-negative modes of states in $\VV_0^{\g,0}[1]\cong \g$ stabilize $\Mh(\np)$ and $\Mh(\nm)$ inside $\Mh= \Mh(\g)$. Thus, we get an action of $L_+\g:= \g\ox \CC[[s]]$, and in fact of $\LpLog := \oc\g[t,t^{-1}] \ox \CC[[s]]$, on $\Mh(\np)$ and $\Mh(\nm)$. See \cref{Lpgstab}. 
(This is in contrast to the obvious vertex-algebra homomorphism $\VV_0^{\g,0}\to \Mfw$, which sends $J_{a,n}\fm{-1}\vac \to \sum_{b,c\in \I} f_{ba}{}^c \SS^b_{c,n}\fm{-1}\vac $; cf. \cref{sec: vfintro}.)

Finally, in \cref{bthm}, we shall lift the homomorphism $\theta$ to a homomorphism
\be  \VV_0^{\g,0} \to \Mfw \ox \pi_0 \nn\ee
where $\pi_0$ is the vacuum Fock module for a system of $\dim\h$ free bosons (see \cref{sec: pi0}). (For this homomorphism, the state $\cent\fm{-1}\vac$ is no longer in the kernel.)

\subsection{} \label{sec: discussion}

Let us conclude this introduction with some comments about these results.

\medskip

The homomorphism $\theta$ is not a \emph{free-field} realization: we adjoined the copy of the vacuum Verma module $\VV_0^{\Lglog\rtimes \CC\Cocent,0}$ in the definition of $\Mfw$, and so the vertex algebra $\Mfw$ does not have mutually commuting creation operators and mutually commuting annihilation operators.
This is apparent already at the level of the Lie algebra homomorphism $\vfd : \g \to \Dg$: the Lie algebra $\Dg$ defined in \cref{Dgintro} had generators $\SS^a_{b,n}$ in addition to the 
mutually commuting coordinates $X^{a,n}$ and mutually commuting derivatives $D_{a,n}$. 
There is some rough intuition that says that is to be expected. In this paper, the vertex algebra structure is always associated to the second coordinate, $s$, appearing in the double loop algebra $\LLog := \oc\g[t,t^{-1}] \ox \CC((s))$.
One would like to be able to say at the same time that $X^{a,n}$ and $D_{b,n}$ are modes in the $t$-coordinate of states ``$X^a_{0}\vac$'' and ``$D_{b,-1}\vac$'', and then that the $\SS^a_{b,n}$ are merely the modes in the $t$-coordinate of a composite state, ``$X^a_{0} D_{b,-1} \vac$'', roughly speaking. 
To make sense of such statements, one would need a theory of vertex algebras on polydiscs (of complex dimension two, in our case), perhaps following \cite{CG,GW,SWW}. Since in the present paper we confine ourselves to the standard definition of vertex algebras, it is perhaps unsurprising that we need to include these $\SS^a_{b,n}$ as generators in their own right.

Relatedly, whereas in finite types $\Mh(\np)$ gets the structure of a module at the critical level over the central extension $\gh$ of the full loop algebra $L\g$, here the subspace $\Mh(\np)\subset \Mfw$ is stabilized by $L_+\g$, as in \cref{Lpgstab}, but certainly not by all of $L\g$. It might be interesting to study the $L\g$ module through $\Mh(\np)$ inside $\Mfw$.
 
\medskip

To illustrate the structure of the image of $\theta$, let us examine an example in type $\g = \wh\sl_2$. One finds
\begin{align}\label{exmp}
 \theta(J_{F,1}\fm{-1} \vac) = 
& \phantom+  \SS^E_{H,1}\fm{-1} \vac - 2 \SS^H_{F,1}\fm{-1}\vac\\
&       - \gam^{E,0}\fm0 \bet_{H,1}\fm{-1} \vac
       - \gam^{E,-1}\fm0 \bet_{H,0}\fm{-1} \vac
       + 2\gam^{H,0}\fm0 \bet_{F,1}\fm{-1} \vac\nn\\
&       - 4\gam^{E,-1}\fm{-1} \vac\nn\\
&       + \bet_{F,1}\fm{-1} \vac\nn\\
&       + 2\gam^{E,-1}\fm0 \gam^{F,0}\fm0 \bet_{F,0}\fm{-1} \vac       - \gam^{E,-1}\fm0 \gam^{E,-1}\fm0 \bet_{E,-1}\fm{-1} \vac\nn\\
&       - 2\gam^{H,-1}\fm0 \gam^{E,-1}\fm0 \bet_{H,-1}\fm{-1} \vac       + 2\gam^{H,-1}\fm0 \gam^{H,-1}\fm0 \bet_{F,-1}\fm{-1} \vac\nn\\
&       + 2\gam^{E,-2}\fm0 \gam^{F,-1}\fm0 \bet_{F,-2}\fm{-1} \vac       - 2\gam^{H,-1}\fm0 \gam^{E,-1}\fm0 \gam^{E,-1}\fm0 \bet_{E,-2}\fm{-1} \vac\nn\\
&       + \tfrac 83\gam^{H,-1}\fm0 \gam^{H,-1}\fm0 \gam^{H,-1}\fm0 \bet_{F,-2}\fm{-1} \vac\nn\\
&       + \gam^{E,1}\fm0 \gam^{E,1}\fm0 \bet_{E,3}\fm{-1} \vac       + 2\gam^{H,1}\fm0 \gam^{E,1}\fm0 \bet_{H,3}\fm{-1} \vac\nn\\
&       - 2\gam^{H,1}\fm0 \gam^{H,1}\fm0 \bet_{F,3}\fm{-1} \vac
+\dots.\nn
\end{align}
In the first line there are terms belonging to $\VV_0^{\Lglog\rtimes \CC\Cocent,0}[1]$; in the second,  a finite sum of compensating quadratic terms. Then in the remaining lines is the sum of other terms, which is infinite but with widening gap. 
The divergences in the $1$st products, cf. \cref{sec: zeta}, are removed because we set the $1$st products of the states $\SS^a_{b,n}\fm{-1} \vac$ to zero. Note that zero, rather than some other finite level, was not a choice: since $\theta(\cent\fm{-1}\vac) =0$ and $\bilin\cent\cocent = 1$, $\theta$ could not be a homomorphism from $\VV_0^{\g,k}$ at any nonzero level $k$. 
Correspondingly, the homomorphism in \cref{Lghom} is from the loop algebra $L\g$, rather than any central extension thereof. It is tempting to say that the critical level is zero for untwisted affine algebras. We do not consider deforming to other levels in the present paper.

\medskip

One motivation for the present paper comes from Gaudin models. In the case of $\g$ of finite type there is a deep connection  \cite{FFR,Fopers} between the centre of the vacuum Verma module at the critical level, $\VV_0^{\g,-h^\vee}$, and the (large, commutative) algebra of Gaudin Hamiltonians, sometimes called the Bethe algebra \cite{MTV1,MTVschubert,RybnikovProof}. In the approach to the Bethe ansatz for Gaudin models described in \cite{FFR}, the Feigin-Frenkel homomorphism $\VV_0^{\g,-h^\vee}\to \Mh(\np)\ox \pi_0$ plays a key role. 
Gaudin models of \emph{affine} type should provide a means of describing the spectra of integrals of motion of certain integrable quantum field theories: an idea pioneered in \cite{FFsolitons}, and with further progress in \cite{V17,FH, LVY,LVY2,LacroixThesis,FJM,DKLMV,DLMV2,Vic2,Y,GLVW}.

\subsection{} The structure of this paper is as follows. 

In \cref{sec: diffop} we construct the homomorphism $\vf: \g \to \Derc\Onp$. Then in \cref{sec: cai} we introduce the homomorphism $\vfd: \g \to \Dg$. In \cref{sec: va} we recall basic facts about $\bet\gam$-systems and vertex algebras, before going on to state the main results starting in \cref{sec: mr}. The proof of the main theorem, \cref{mainthm}, is given in \cref{sec: proof}. It follows the strategy due to Feigin and Frenkel and discussed in detail in \cite[\S5]{Fre07}. In particular, we introduce the $\bee\cee$-ghost system and use it to define a subcomplex, the \emph{local complex}, of the Chevalley-Eilenberg complex for (in our case) the double loop algebra $\LLog$. In \cref{sec: proofb} we give the proof of \cref{bthm}. Finally, in \cref{sec: cis} we compute explicitly the values of coefficients appearing in the images of the Chevalley-Serre generators of $\g$ under the homomorphism $\theta$: see \cref{CScor}.


\section{Realization of $\g$ by differential operators on the big cell}\label{sec: diffop}

\subsection{Loop realization} \label{sec: loop rel}
We work over the complex numbers $\CC$. Let $\oc\g$ be a finite-dimensional simple Lie algebra, and  
$\Log$ the Lie algebra of Laurent polynomials, in a formal variable $t$, with coefficients in $\oc\g$.
Let $\bilin{\cdot}{\cdot}: \oc\g\times\oc\g \to \CC$ denote the non-degenerate symmetric $\oc\g$-invariant bilinear form on $\oc\g$, with the standard normalization from \cite{KacBook}. 
Let $\g'$ denote the central extension of $\Log$ by a one dimensional centre $\CC\cent$,
\be 0 \to \CC \cent \to \g' \to \Log \to 0, \nn\ee
whose commutation relations are given by $[\cent,\cdot] = 0$ and 
\be [a \ox f(t), b \ox g(t) ] :=[a,b] \ox f(t) g(t) - (\res_tfdg) \bilin a b \cent .\nn\ee
If we write $a_n := a\ox t^n$ for $a\in \oc\g$ and $n\in \ZZ$, the commutation relations take the form
\be [a_m, b_n ] = [a,b]_{n+m} + m \delta_{n+m,0} \bilin a b \cent .\nn\ee

Define the Lie algebra 
\be \g := \g' \rtimes  \CC\cocent,\nn\ee
by declaring that $\cocent$ obeys $[\cocent,\cent]=0$ and $[\cocent, a\otimes f(t)] = a\otimes t\del_t f(t)$ for all $a \in \oc\g$ and $f(t) \in \CC[t,t^{-1}]$. 

The form $\bilin{\cdot}{\cdot}$ extends uniquely to a non-degenerate invariant symmetric bilinear form on $\g$, which we also write as $\bilin{\cdot}{\cdot}$, whose nonzero entries are given by
\be \bilin{a_n}{b_m} = \bilin a b \delta_{n+m,0} ,\qquad \bilin \cent \cocent = \bilin \cocent \cent = 1 \nn\ee 
for $a,b\in \oc\g$, $n,m\in \ZZ$.
\subsection{Kac-Moody data}\label{sec: uaa}
Recall that the Lie algebra $\g$ is isomorphic to a Kac-Moody algebra $\g(A)$ with indecomposable Cartan matrix $A = (A_{ij})_{i,j\in I}$ of untwisted affine type. Here  $I=\{0,1,\dots,\ell\}$ is the set of the labels of the nodes of the Dynkin diagram. 
Let $\h\subset \g(A)$ be the Cartan subalgebra and 
\be \g = \nm \oplus \h \oplus \np \nn\ee
the Cartan decomposition, where $\np$ (resp. $\nm$) is generated by $e_i$ (resp. $f_i$), $i\in I$. These $e_i,f_i$ are the Chevalley-Serre generators of the derived subalgebra $\g' := [\g,\g]$. By definition they obey
\begin{subequations} \label{KMrels}
\begin{alignat}{2}
\label{KM rel a} [h,  e_i] &= \langle \alpha_i,h\rangle   e_i, &\qquad
[h,  f_i] &= - \langle \alpha_i,h\rangle  f_i, \\
\label{KM rel b} [h, h'] &= 0, &\qquad
[ e_i,  f_j] &= \check \alpha_i \delta_{ij},
\end{alignat}
for any $h, h' \in \h$, together with the Serre relations
\be \label{KM rel c}
(\text{ad}\,  e_i)^{1- A_{ij}}  e_j = 0, \qquad (\text{ad}\,  f_i)^{1- A_{ij}}  f_j = 0.
\ee
\end{subequations}
Here  $\la\cdot,\cdot\ra : \h^* \times \h \to \CC$ is the canonical pairing, $\alpha_i$ (resp. $\check\alpha_i$), $i\in I$, are the simple roots (resp. simple coroots) of $\g$. We have $\la \alpha_i,\check\alpha_j\ra = A_{ji}$.

Let $\b := \h \oplus \np$ and $\b_- := \h \oplus \nm$. 

There exist unique collections of relatively prime positive integers $\{\check a_i\}_{i\in I}$ and $\{a_i\}_{i\in I}$ such that $\sum_{j\in I} A_{ij} a_j  =0$ and $\sum_{j\in I} \check a_i A_{ij} = 0$. Whenever $\g$ is untwisted (in fact more generally whenever $\g$ is not of type ${}^2\!A_{2k}$) one has $\check a_0 = 1$ and $a_0 = 1$.  
The dual Coxeter and Coxeter numbers of $\g$ are given respectively by
\be h^\vee = \sum_{i\in I} \check a_i,\qquad h = \sum_{i\in I} a_i = 1+ \sum_{i\in \oc I} a_i. \nn\ee
The central element $\cent\in \h$ and imaginary root $\delta\in \h^*$ are given by
\be \cent = \sum_{i\in I} \check a_i \check \alpha_i\in \h, \qquad \delta = \sum_{i\in I} a_i \alpha_i \in \lh\label{kddef}\ee

Let $\oc I = I \setminus \{0\} = \{1,\dots,\ell\}$. The matrix $(A_{ij})_{i,j\in \oc I}$ obtained by removing the zeroth row and column of $A$ is the Cartan matrix of $\oc\g$. 
Let $\oc\h := \textup{span}_\CC \{\check\alpha_i\}_{i\in \oc I}\subset \g$ denote its Cartan subalgebra and $\oc\h^* = \textup{span}_\CC \{\alpha_i\}_{i\in \oc I}$ its dual.

\subsection{Root lattice and basis of root vectors}\label{Bpdef}
Let $Q := \bigoplus_{i\in I} \ZZ \alpha_i$ denote the root lattice of $\g$. We have the decomposition of $\g$ into root spaces,
\be \g = \bigoplus_{\alpha\in Q} \g_\alpha,\qquad \g_\alpha := \{ x\in \g: [h,x] = x \la \alpha,h\ra \text{ for all }  h\in \h\},\nn\ee
and this is a $Q$-gradation of the the Lie algebra $\g$.  
Let $\Delta := \{ \alpha\in Q: \dim \g_\alpha \neq 0 \}$ be the set of roots of $\g$. 
We write $Q_{\geq 0} := \bigoplus_{i \in I} \ZZ_{\geq 0} \alpha_i$ and $Q_{>0} := Q_{\geq 0} \setminus \{0\}$. The positive roots of $\g$ are $\Delta_+ := Q_{>0} \cap \Delta$. 

Let $\oc Q := \bigoplus_{i\in \oc I} \ZZ\alpha_i$ denote the root lattice of $\oc\g$ and $\oc \Delta \subset \oc Q$ its set of roots. Let $\oc \Delta_+ := \oc \Delta \cap \Delta_+$ be the positive roots of $\oc\g$. 

Recall that the roots of $\g$ are given by
\be \Delta = \{ \alpha + n \delta : \alpha \in \oc\Delta, n \in \ZZ\} \cong \oc \Delta \times \ZZ\nn\ee
and the positive roots are given by
\be \Delta_+ = \oc\Delta_+ \sqcup \{ \alpha + n\delta: \alpha \in \oc \Delta, n\in \NN\} \cong \oc \Delta_+ \sqcup \oc \Delta \times \NN \nn\ee

Let $\{H_i\}_{i\in \oc I}\subset \oc\h$ be a basis of $\oc\h$. 
Choose root vectors $E_{\pm \alpha}\in \oc\g_{\pm \alpha}$ for each $\alpha\in \oc\Delta_+$, normalized such that $\la \alpha , [E_\alpha,E_{-\alpha}] \ra = 2$. 
Let\footnote{Strictly speaking, in the examples in type $\wh\sl_2$ in the introduction, we wrote for example $J_{E,n}$ rather than $J_{\alpha_1,n}$ and $J_{H,n}$ rather than $J_{1,n}$.}
\be J_\alpha := E_\alpha, \qquad J_i := H_i,\nn\ee
and also let $\I := \left(\oc\Delta\setminus\{0\}\right) \cup \oc I$,
so that
\be \{ J_a \}_{a\in \I} = \{E_{\pm\alpha}\}_{\alpha\in\oc\Delta_+}\cup \{H_i\}_{i\in \oc I} \nn\ee
is a Cartan-Weyl basis of $\oc\g$. 

Let $\B$ denote the basis of $\g$ given by
\be \B = \{\cent,\cocent\} \cup \{ J_{a,n} \}_{a\in \I, n\in \ZZ}\nn\ee 
where $J_{a,n} := J_a\ox t^n$.
We fix a total ordering $\prec$ on $\B$ as follows. 
Pick any total ordering $\prec$ of the positive roots $\oc\Delta_+$ of $\oc\g$ such that $\alpha\prec\beta$ whenever $\alpha-\beta \in Q_{>0}$.  
and any total ordering $\prec$ of the Cartan generators $\{H_i\}$ of $\oc\g$. Then declare that for every $n\in \ZZ$,
\begin{subequations}\label{Bpdef}
\be E_{-\beta,n} \prec E_{-\alpha,n} \prec H_{i,n} \prec H_{j,n} \prec E_{\alpha,n} \prec E_{\beta,n} \prec E_{-\beta,n+1} \dots \ee
whenever $i\prec j$ and $\alpha\prec \beta$, and finally that 
\be E_{-\alpha,0} \prec \cocent \prec \cent \prec H_{i,0}. \nn\ee
\end{subequations}

We can identify the Chevalley-Serre generators of $\g'$ as follows: $e_i = E_{\alpha_i,0}$ and $f_i = E_{-\alpha_i,0}$, for $i\in \oc I$, while $e_0 = E_{-\delta+\alpha_0,1}$ and $f_0 = E_{\delta-\alpha_0,-1}$, 
where
\be \delta-\alpha_0 = \sum_{i\in \oc I} a_i \alpha_i \in \oc\Delta_+\label{highestrootdef}\ee 
is the highest root of $\oc\g$. 

Let 
\be \wgt: \Ag \to Q; \quad 
\begin{cases} \wgt{\ii \alpha n} = \alpha + n \delta & \alpha \in \oc\Delta \setminus\{0\} \\
              \wgt{\ii i n} = n \delta  & i \in \oc I 
            \end{cases}
            \nn\ee
\subsection{The pro-nilpotent pro-Lie algebra $\npc$}\label{sec: npc}
Define 
\be \npc := \prod_{\alpha\in \Delta_+} \g_\alpha. \nn\ee
An element of $\npc$ is a (possibly infinite) sum of the form $\sum_{\alpha\in \Delta_+} x^\alpha$ with $x^\alpha\in \g_\alpha$ for each $\alpha\in \Delta_+$. (It lies in $\np\subset \npc$ if and only if all but finitely many of the $x^\alpha$ are zero.) The Lie bracket is well-defined on $\npc$ because, for a given positive root $\alpha\in \Delta_+$, there are only finitely many positive roots $\beta,\gamma\in \Delta_+$ such that $\beta+\gamma=\alpha$. Similarly, the Lie bracket is well-defined on 
\be \bc := \h\oplus \npc, \qquad \gc := \nm \oplus \h \oplus \npc .\nn\ee 

Let $\height(\alpha)$ denote the grade of a root $\alpha\in Q$ in the homogeneous $\ZZ$-gradation of $\g$,
i.e. $\height(n\delta + \alpha) = n$ for $n\in \ZZ$ and $\alpha\in \oc\Delta$. 
Define
\be \n_{\geq k} := \bigoplus_{\substack{\alpha\in \Delta_+\\ \height(\alpha)\geq k}} \np_\alpha, \qquad k\geq 1. \nn\ee
These are Lie ideals in $\np$ and we have embeddings $\n_{\geq k+1} \into \n_{\geq k}$ for each $k\in \ZZ_{\geq 1}$.
For each $k$, the quotient $\np\big/\n_{\geq k}$ is a nilpotent Lie algebra, and these nilpotent Lie algebras form an inverse system
\be \dots \onto \np\big/\n_{\geq k+1} \onto \np\big/\n_{\geq k} \onto \dots . \nn\ee 
The inverse limit of this inverse system is isomorphic to $\npc$:
\be \npc \cong \invlim_{k} \np\big/\n_{\geq k}.\label{civ}\ee
Indeed, elements of this inverse limit are by definition (possibly infinite) sums of the form $\sum_{\alpha\in \Delta_+} x^\alpha$, $x^\alpha\in \g_\alpha$, which truncate to finite sums modulo $\n_{\geq k}$ for any $k$; in other words, they are nothing but elements of $\npc$.

In this way, $\npc$ becomes a topological Lie algebra, with the linear topology in which a base of the open neighbourhoods of $0$ is given by the ideals
\be \nc_{\geq k} := \prod_{\substack{\alpha\in \Delta_+\\ \height(\alpha)\geq k}} \np_\alpha ,\qquad k\geq 1. \nn\ee  
(This topology is Hausdorff since $\bigcap_{k=1}^\8 \nc_{\geq k} = \{0\}$.)

In fact, the Lie algebra $\npc$ gets the structure of a \emph{pro-nilpotent pro-Lie algebra}; see e.g. \cite[\S4.4 and \S6.1]{Kumar}. The pro-Lie algebra structure on $\npc$ consists of the family $\F$ consisting of \emph{all} Lie ideals $\a \subset \npc$ such that $\a\supset \nc_{\geq k}$ for some $k$, and the \emph{pro-topology} is the linear topology in which these ideals form a base of the open neighbourhoods of $0$. 

\subsection{The group $U$}\label{sec: udef}
For each $k\geq 1$, let $\exp \bigl( \np\big/ \n_{\geq k} \bigr)$ denote a copy of the vector space $\np/ \n_{\geq k}$ and let $m\mapsto \exp(m)\equiv \Exp m$ be the map into this copy. We may endow this copy with a group structure, given by the Baker-Campbell-Hausdorff formula,
\be  \exxp{x}\exxp{y} := \exxp{x+y+ \frac 12 [x,y] + \dots} ,\nn\ee 
from which only finitely many terms contribute since $\np / \n_{\geq k}$ is nilpotent. For all $m,k$ with $m\geq k$, there is a commutative diagram
\be\begin{tikzcd}
\np / \n_{\geq m}\rar[two heads]\dar{\sim} & \np / \n_{\geq k} \dar{\sim}\\
\exxp{\np / \n_{\geq m} } \rar[two heads] & \exxp{  \np / \n_{\geq k} }
\end{tikzcd}
\nn\ee
where the horizontal maps are the canonical projections. 
The inverse limit
\be  U := \invlim_k \exxp{ \np / \n_{\geq k} } \label{udef}\ee
is then a group, and the diagram above defines an exponential map $\exp : \hn_+ \isom U$.
Recall that, by definition of the inverse limit, we have the commutative diagram
\be
\begin{tikzcd} & &  U \arrow[dashed,dll,shorten = 2mm,two heads]
 \dlar[two heads]{\pi_3} \dar[two heads]{\pi_2} \drar[two heads]{\pi_1} \\
 \dots \phantom{O}\rar[dashed,two heads]& 
\exxp{\np/\n_{\geq 3}} \rar[two heads]& 
\exxp{\np/\n_{\geq 2}} \rar[two heads]& 
\exxp{\np/\n_{\geq 1}}
\end{tikzcd}
\nn\ee
in which the maps are surjective group homomorphisms, and the group element $g$ in \cref{gel} is equivalent to the sequence $(\pi_i(g))_{i=1}^\8$ of its truncations.

In fact, to any pro-Lie algebra one can canonically associate a \emph{pro-algebraic group} (or \emph{pro-group}, for short). If the pro-Lie algebra is pro-nilpotent, then this pro-group will be \emph{pro-unipotent}. For the definitions, see e.g. \cite[\S4]{Kumar}.
The group $U$ is the pro-unipotent pro-group algebra associated to the pro-nilpotent pro-Lie algebra $\npc$,  just as the quotients $\exxp{\np/\n_{\geq k}}$ are the unipotent affine algebraic groups associated to the nilpotent Lie algebras $\np/\n_{\geq k}$.


\subsection{Polynomial functions on $U$}\label{sec: coords}
Now let us choose coordinates on $U$. 
Recall our ordered basis $(\B,\prec)$ of $\g$ from \cref{Bpdef}. We get an ordered basis of $\np$, 
\begin{align} \B_+ &:= \B \cap \np \nn\\
&=  \left\{ J_{\alpha,0}\right\}_{\alpha\in \oc\Delta_+} 
 \cup \left\{J_{a,n} \right\}_{a\in \I, n\in\NN} = \{J_{\ia a n}\}_{{\ii a n}\in \A},\label{Bplus}
\end{align}
where we introduced a notion for the index set,
\be \A := \{(\alpha,0)\}_{\alpha\in \oc\Delta_+} \,\,\cup\,\, \Ag_{\geq 1}. \label{Adef}\ee
It will also be useful to introduce 
the set
\be \Am := \{(\alpha,0)\}_{\alpha\in \oc\Delta_-} \,\,\cup\,\, \I \times \NNm. \label{Amdef}\ee

Any element of $g\in U$ can be uniquely written in the form
\be g = \prodr_{{\ii a n} \in \A} \exxp{x^{\ia a n} J_{\ia a n}} \label{gel}\ee
for some $x^{\ia a n}\in \CC$, where we use $\prodr$ to denote the product with factors ordered so that $\exxp{x^{\ia b m} J_{\ia b m}}$ stands to the left of $ \exxp{x^{\ia c p} J_{\ia c p}}$ if $J_{\ia b m}\prec J_{\ia c p}$ in our basis.


For each ${\ii a n}\in \A$ let $X^{\ia a n} : U \to \CC$ be the function on $U$ such that $X^{{\ia a n}}(g) = x^{\ia a n}$. Then these are good coordinates on $U$. 
We get the $\CC$-algebra
\begin{align} \Onp &:= \CC\!\left[X^{\ia a n}\right]_{{\ii a n}\in \A}\nn
\end{align}
of polynomial functions on $U$. 
Inside $\Onp$ we have for each $k\in \ZZ_{\geq 1}$ the subalgebra 
\be \Onp_{< k} := \CC[X^{\ia a n}]_{\substack{{\ii a n}\in \A\\ n< k}} \label{Okdef}\ee
of polynomial functions on the quotient group $\exp(\np\big/ \n_{\geq k})$, and $\Onp= \bigcup_{k\geq 1} \Onp_{< k}$ is their union. 
In other words, $\Onp$ is the direct limit of these $\Onp_{< k}$ with respect to the inclusions 
\be \dots \into \Onp_{<k} \into \Onp_{<k+1} \into \dots.\label{oinc}\ee

\subsection{Differential operators on $U$}\label{Hdef}
Let $\HH$ denote the Weyl algebra on generators $X^{\ia a n}$, $D_{\ia a n}$, ${\ii a n}\in \A$, i.e. the associative unital $\CC$-algebra obtained by quotienting the free associative unital $\CC$-algebra in these generators by the two-sided ideal generated by the relations 
\be [X^{\ia a n},X^{\ia b m} ] = 0, \qquad [D_{\ia a n}, X^{\ia b m}] = \delta_a^b \delta_{n+m,0},\qquad [D_{\ia a n}, D_{\ia b m} ] =0 ,\nn\ee
where $[A,B] := AB-BA$. 
As a vector space, one has
\be \HH \cong_\CC \CC[X^{\ia a n}]_{{\ii a n}\in \A} \ox \CC[D_{\ia a n}]_{{\ii a n}\in\A}.\nn\ee
We think of $\HH$ as the algebra of polynomial differential operators on $U$. The left action of $\HH$ on $\Onp$ is given by
\be X^{\ia a n} \on f = X^{\ia a n} f,\qquad D_{\ia a n} \on f = [D_{\ia a n},f] \nn\ee
for ${\ii a n}\in \A$ and $f\in \Onp$. 

Let $\HH_{\geq k}$ denote the left ideal in $\HH$ generated by $D_{{\ii a n}}$ with $n\geq k$. It consists of those polynomial differential operators that annihilate $\Onp_{<k}$. The quotients $\HH/\HH_{\geq k}$ form an inverse system of associative algebras,
\be \dots \onto \HH \big/ \HH_{\geq k+1} \onto \HH \big/ \HH_{\geq k} \onto. \nn\ee
The inverse limit 
\be \HHc := \invlim_k \HH \big/ \HH_{\geq k}\nn\ee 
is a topological associative algebra, whose elements are by definition (possibly infinite) sums of the form
\be \sum_{{\ii a n}\in \A} P^{\ia a n}(X;D) D_{\ia a n}, \label{PSDsum}\ee
where, for each ${\ii a n}\in \A$, $P^{\ia a n}(X;D)\in \CC[X^{\ia b m}]_{{\ii b m}\in\A}\ox \CC[D_{\ia b m}]_{{\ii b m} \in \A}$.
Since any $f\in \Onp$ lies in $\Onp_{<k}$ for some $k$, this completion $\HHc$ also has a well-defined action on $\Onp$.

Note that the degree of the polynomials $P^{\ia a n}(X;D)$ in \cref{PSDsum} is allowed to grow without bound as $n$ increases. 

\subsection{The $Q$-gradation}
The algebra $\HH$ has a gradation by the root lattice $Q$, 
\be \HH = \bigoplus_{\alpha \in Q} \HH_\alpha, \nn\ee
defined by the demand that $\D_{\ia a n} \in \HH_{\wgt\ii a n}$ and $X^{\ia a n} \in \HH_{-\wgt\ii a n}$ for all $\ii a n \in \A$. That is, more explicitly,
\be D_{\ia \alpha n}\in \HH_{\alpha +n\delta}\qquad\text{and}\qquad X^{\ia \alpha n} \in \HH_{-\alpha - n\delta}, \nn\ee
for each $\ii \alpha n \in  \oc \Delta_+\times\{0\} \cup  (\oc\Delta\setminus\{0\}) \times \NN \subset \A$, and 
\be D_{\ia i n}\in \HH_{n\delta}\qquad\text{and}\qquad X^{\ia i n} \in \HH_{- n\delta}, \nn\ee
for each $\ii i n \in \oc I \times \NN \subset \A$. 
In particular we get a gradation of $\Onp\subset \HH$ by the negative root lattice:
\be  \Onp = \bigoplus_{\alpha \in -Q_{\geq 0}} \Onp_\alpha .\nn\ee

\subsection{Polynomial vector fields on $U$}\label{sec: derO}
The commutator bracket $[A,B] := AB-BA$ makes $\HH$ into a $Q$-graded Lie algebra. It has the Lie subalgebra $\Der\Onp$ consisting the (finite) sums of the form 
\be \sum_{{\ii a n}\in \A} P^{\ia a n}(X) D_{\ia a n}, \label{derogen}\ee
where the coefficients $P^{\ia a n}(X)\in \CC[X^{\ia b m}]_{{\ii b m}\in \A}$ are nonzero for only finitely ${\ii a n}\in \A$, i.e. it the free $\Onp$-module with $\Onp$-basis consisting of the derivative operators $D_{\ia a n}$, ${\ii a n}\in \A$.
It inherits the $Q$-grading of $\HH$. 

\begin{lem}
The commutator bracket makes the completion $\HHc$ into a topological Lie algebra.
\end{lem}
\begin{proof}
The associative algebra $\HHc$ is certainly a Lie algebra with respect to the commutator.
Let us check the commutator $[\cdot,\cdot]: \HHc\times \HHc \to \HHc$ is continuous, where $\HHc\times \HHc$ has the product topology. The open neighbourhoods of zero in the product topology include the subspaces $\HHc_{\geq k}\times \HHc_{\geq \ell}$ for all $k,\ell \in \ZZ_{\geq 1}$. Suppose $(A_n,B_n)$, $n=1,2,\dots$, is a sequence in $\HHc\times \HHc$ which converges to zero. Then, for any $k$, the terms in this sequence are eventually in $\HHc_{\geq k}\times \HHc_{\geq k}$. It follows that, for any $k$, the terms of the sequence $[A_n,B_n]$, $n=1,2,\dots$, of commutators eventually lie in $\HHc_{\geq k}$. But the latter are a base of open neighbourhoods of zero in $\HHc$, so the sequence of commutators converges to zero. Since these are all linear topologies (i.e. vector addition is continuous) this is enough to show that the commutator is a continuous map, as required. 
\end{proof}
Define $\Derc\Onp\subset \HHc$ to be the vector subspace of $\HHc$ topologically generated by the monomials of the form $P^{\ia a n}(X) D_{\ia a n}$. In other words, $\Derc\Onp$ consists of (now possibily infinite) sums of the form \cref{derogen}. 

\begin{lem} $\Derc\Onp$ is a topological Lie subalgebra of $\HHc$.
\end{lem}
\begin{proof}
This is really immediate, given the previous lemma: all that has to be checked is that terms linear in $D$'s close under the commutator, which is obvious. But let us write out the full argument nonetheless.
Let $\sum_{{\ii a n}\in \A} P^{\ia a n}(X) D_{\ia a n}$ and $\sum_{{\ii a n}\in \A} Q_{\ia a n}(X) D_{\ia a n}$ be two elements of $\Derc\Onp$. Their commutator is equal to
\be \sum_{{\ii b m}\in \A} \left(\sum_{{\ii a n}\in \A} P^{\ia a n}(X) \left[D_{\ia a n}, Q_{\ia b m}(X)\right] - \sum_{{\ii a n}\in \A} Q_{\ia a n}(X) \left[D_{\ia a n}, P^{\ia b m}(X)\right]\right) D_{\ia b m}  \label{comsum}\ee
For each ${\ii b m}$, the polynomial $Q_{\ia b m}(X)\in \Onp$ lies in $\Onp_{< k}$ for some $k$, so the sum 
\be \sum_{{\ii a n}\in \A} P^{\ia a n}(X) [ D_{\ia a n},Q_{\ia b m}(X)] \nn\ee
is well-defined: only the (finitely many) terms with $n<k$ contribute. The other sum in \cref{comsum} works the same way. We see that \cref{comsum} is a well-defined sum belonging to $\Derc\Onp$. 
\end{proof}

\subsection{The topological Lie algebra $\gc$}
For each $k\in \ZZ$, define vector subspaces of $\g$ and $\gc$ respectively as follows:
\be \g_{\geq k} = \bigoplus_{\substack{\alpha\in \Delta\\\height(\alpha) \geq k}} \g_\alpha, \qquad
 \gc_{\geq k} = \prod_{\substack{\alpha\in \Delta\\\height(\alpha) \geq k}} \g_\alpha .\nn\ee
(So when $k\geq 1$, $\g_{\geq k} = \n_{\geq k}$ and $\gc_{\geq k} = \nc_{\geq k}$.) Just as in \cref{civ}, we have 
\be \gc \cong \invlim_k \g\big/\g_{\geq k} \quad\left(\cong \invlim_k \gc\big/\gc_{\geq k}\right),\nn\ee
where this is now an inverse limit merely of vector spaces. This endows the vector space $\gc$ with a  
linear topology in which $\{\gc_{\geq k}\}_{k\in \ZZ}$ form a base of the open neighbourhoods of $0$. Moreover $\gc = \bigcup_{k\in \ZZ} \gc_{\geq k}$ and $\dim\left(\gc_{\geq k}/\gc_{\geq m}\right)<\8$ for all $k,m \in \ZZ$ with $k\leq m$. 

(With this topology, $\gc$ becomes a \emph{c.l.c. $\CC$-vector space}, in the language of \cite{KashiwaraFlag}.) 

The Lie bracket on $\gc$ is continuous in this topology, i.e. 
\be \gc \times \gc \to \gc; \qquad (A,B) \mapsto [A,B] \nn\ee
is continuous, where $\gc \times \gc$ gets the product topology. 

The adjoint action makes the Lie algebra $\gc$, regarded as a vector space, into a module over itself. In particular, it makes $\gc$ into a module over the subalgebra $\npc$.
\begin{lem}\label{Uactg} The action of $\npc$ on $\gc$ exponentiates to yield a well-defined action of $U$ on $\gc$ from the right, 
\be \gc \times U \to \gc,\nn\ee
given by
\be B \npo e^A = \sum_{k=0}^\8 \frac{(-1)^k}{k!} \ad_A^k B = B + [B,A] + \frac 1 2 [[B,A],A] + \dots \nn\ee
for $B\in \gc$, $A\in \npc$.
\qed\end{lem}

\subsection{Infinitesimal transformations}\label{sec: inft}
Unlike the action of $\npc\subset \gc$ on $\gc$, the action of the full Lie algebra $\gc$ on itself does not exponentiate to an action of a group. Indeed, for $A,B\in \gc$, the sum $\sum_{k=1}^\8 \frac 1 {k!} \ad^k_{A}B$ may not be an element of $\gc_{\geq m}$ for any $m\in \ZZ$, which is to say it  may not be a well-defined element of $\gc$. (Consider for example $A = H_{i,-1}$ and $B = E_{\alpha,0}$.) 

Nonetheless, it is convenient for computations in what follows to be able to treat infinitesimal transformations (by general elements of $\gc$) on the same footing as finite group transformations. That is, we would like to make sense of  group elements of the form $\exp(\eps A) = 1+ \eps A$, working to first order in an (``infinitesimal'') parameter $\eps$.

To that end, let $\eps$ be a formal variable and let $\CC_\eps$ denote the ring
\be \CC_\eps := \CC[\eps]\big/\eps^2\CC[\eps].\nn\ee 
For any Lie algebra $\p$ (over $\CC$), we have the Lie algebra $\p(\CC_\eps) := \CC_\eps \ox_\CC \p$.   
Consider the Lie algebra 
\be \gc(\CC_\eps) := \CC_\eps \ox_\CC \gc. \nn\ee 
We shall write its elements as $A+ \eps B$ with $A,B\in \gc$. It has a nilpotent Lie ideal  $\eps \gc$. The vector subspace
\be \f_\eps := \eps\gc \oplus \npc
, \label{esa}\ee 
forms a Lie subalgebra
\be \f_\eps \subset \gc(\CC_\eps).\nn\ee
It is another pro-nilpotent pro-Lie algebra.
We get the corresponding pro-unipotent pro-group $\exxp{\f_\eps}$.
Let us write $\ad_A(B) := [A,B]$.  
\begin{lem}\label{BCHlem} 
For all $A\in \npc$ and $B\in \gc$, we have the following relations in $\exxp{\f_\eps}$,
\be \Exp{ A } \Exp{ \eps B } = \Exp{ \eps B} \left(\prod_{k=1}^\8 \Exp{ \eps \frac{1}{k!}\ad_A^k(B)}\right) \Exp A, \nn\ee 
\be \Exp{ \eps B } \Exp{ A } = \Exp{ A} \left(\prod_{k=1}^\8 \Exp{ \eps \frac{(-1)^k}{k!}\ad_A^k(B)} \right)\Exp{\eps B}, \nn\ee
together with $\Exp{ \eps A} \Exp{\eps B} = \Exp{\eps B} \Exp{\eps A}$, and $\Exp{ xA } \Exp{\eps A} = \Exp{(x+\eps) A} = \Exp{ \eps A} \Exp{xA}$ for all $x\in \CC$.
\end{lem}
\begin{proof} These follow from the Baker-Campbell-Hausdorff formula. (Note that the order of the terms in the products $\prod_{k=1}^\8$ is unimportant by virtue of the equality $\Exp{ \eps A} \Exp{\eps B} = \Exp{\eps B} \Exp{\eps A}$.)
\end{proof}


There is another vector-space decomposition of $\f_\eps$, namely ${\f_\eps} =_\CC \eps \b_- \oplus \npc(\CC_\eps)$, and this gives a useful way to factorize elements of the group $\exxp{\f_\eps}$. Let $U(\CC_\eps) := \exxp{ \npc(\CC_\eps)}$. Then the multiplication map
\be \exxp{\eps\b_-} \times U(\CC_\eps) \isom \exxp{\f_\eps} \label{mb}\ee
is a bijection. 

Let us introduce the right coset space
\be U_0(\CC_\eps) := \exxp{\eps\b_-} \!\Big\backslash\! \exxp{\f_\eps}.\nn\ee
The action from the right of the subgroup $U(\CC_\eps)\subset \exxp{\f_\eps}$ is both transitive and free, in view of \cref{mb}, and we get a bijection
\be U(\CC_\eps) \isom U_0(\CC_\eps);\qquad g \mapsto  \exxp{\eps \b_-} g .\nn\ee
Let 
\be U_0 := \exxp{\eps \b_-} U \nn\ee
denote the orbit of $\exxp{\eps \b_-}$ under the right action of the subgroup $U\subset U(\CC_\eps)$. The bijection above restricts to a bijection
\be U \isom U_0;\qquad g \mapsto  \exxp{\eps \b_-} g .\nn\ee
By means of this bijection, we can regard $X^{\ia a n}$, ${\ii a n}\in \A$ as coordinates on $U_0$, and $\Onp$ as the ring of polynomial functions on $U_0$.

For us, the motivation for these constructions is contained in the following lemma.

\begin{lem}\label{Plem} Let $A\in \gc$. Then there exist polynomials $\left\{P_A^{\ia b m}(X)\in \Onp\right\}_{{\ii b m}\in \A}$ (depending linearly on $A$) such that
\be \Exp{\eps \b_-}\left( \prodr_{{\ii b m}\in\A} \Exp{x^{\ia b m} J_{\ia b m}}\right) \Exp{\eps A} = \Exp{\eps \b_-}
 \prodr_{{\ii b m}\in\A} \Exp{\left(x^{\ia b m} + \eps P_A^{\ia b m}(x)\right)J_{\ia b m}}\nn\ee
in $U_0(\CC_\eps)$, for every element $\prodr_{{\ii b m}\in\A} \Exp{x^{\ia b m} J_{\ia b m}}$ of the group $U \subset U(\CC_\eps)$. 

If $A\in \g_{\alpha}$ and $J_{\ia b m}\in \g_\beta$ then $P_A^{\ia b m}(X) \in \Onp_{\alpha-\beta}$.
\end{lem}
\begin{proof}
This is true by construction, but let us go through the details to explain how in practice one computes these $P_A^{\ia b m}$. 
Recall that the group element 
\be g = \prodr_{{\ii b m}\in\A} \Exp{x^{\ia b m} J_{\ia b m}}\in U \nn\ee
is defined by the sequence $(g_i)_{i=1}^\8$ of its truncations, 
\be
g_i := \pi_i(g) = \prodr_{\substack{{\ii b m}\in\A\\ m< i }} \Exp{x^{\ia b m} J_{\ia b m}}\in \Exp{\np/\n_{\geq i}}, 
\label{geli}\ee
to the subgroups $\Exp{\np / \n_{\geq i}}$. The product in \cref{geli} is finite, for each $i$.

Let $A\in \gc$. Then $A\in \gc_{\geq k}$ for some integer $k$. Consider the right action on $g$ by the group element $\Exp{\eps A}\in \exxp{\eps\gc}\subset \exxp{\f_\eps}$. The group element $g':= g\Exp{\eps A}\in \exxp{\f_\eps}$ is defined by the sequence of \emph{its} truncations, $\left(g'_i\right)_{i=1}^\8$. On grading grounds, we have 
\be g'_i  = \pi_i\left( g_{\max(i,i-k)} \Exp{\eps A}\right), \nn\ee
i.e. to compute $\pi_i(g \Exp{\eps A})$ we need at most the first $\max(i,i-k)$ terms in the sequence \cref{geli} of truncations of $g$. Also, we need only the truncation $\pi_i(\Exp{\eps A})$ of $\Exp{\eps A}$; this truncation can be expressed as a finite product of monomials of the form $\Exp{\eps a}$ with $a$ a multiple of one of our basis vectors of $\g$. 
The relations in \cref{BCHlem} hold in particular when the $A$ and $B$ there are multiples of vectors from our basis $\B_+$ of $\np$, and in that case so too are all the factors in the products on $k$ on the right-hand sides. 
Using these relations a finite number of times, we may re-write $g_i'$ in the form
\be g_i' = g_i'' \prodr_{\substack{{\ii b m}\in\A\\\height({\ii b m})< i}} 
\Exp{(x^{\ia b m}+ \eps P^{\ia b m}_A(x)) J_{\ia b m}}, \qquad\text{where}\quad g_i'' \in \Exp{\eps \b_-},
\nn\ee
for polynomials $P^{\ia b m}_A(X)\in \Onp$, depending linearly on $A$ and with the stated grading property.
\end{proof}

\begin{lem}\label{vfdef}
The linear map 
\be \vf:\gc \to \Derc\Onp\nn\ee 
given by
\be \vf: A \mapsto \sum_{{\ii b m}\in \A} P_A^{\ia b m}(X) D_{\ia b m} \nn\ee
is a homomorphism of Lie algebras. It respects the $Q$-gradation, and is therefore continuous.
\end{lem}
\begin{proof}
Let $\CC_{\eps,\eta} := \CC[\eps,\eta]\big/\left(\eps^2\CC[\eps,\eta]\oplus \eta^2\CC[\eps,\eta]\right)$. Replacing $\CC_\eps$ by $\CC_{\eps,\nu}$ in the definitions above, we get the Lie algebra $\gc(\CC_{\eps,\eta}):= \CC_{\eps,\eta} \ox \gc$ and its Lie subalgebra $\f_{\eps,\eta} =_\CC \eps \gc \oplus \eta \gc \oplus \npc =_\CC \eps \b_- \oplus \eta \b_- \oplus \npc(\CC_{\eps,\eta})$. Define $U_0(\CC_{\eps,\eta}) := \exxp{\eps\b_- \oplus \eta\b_-}\Big\backslash \exxp{\f_{\eps,\eta}}$. We have $U(\CC_{\eps,\eta}) \cong U_0(\CC_{\eps,\eta})$. We can identify $U_0$ as defined above with the orbit $\exxp{\eps\b_-\oplus \eta\b_-} U$ of  $\exxp{\eps\b_-\oplus \eta\b_-}$ under the right action of $U\subset U(\CC_{\eps,\eta})$. Let $A,B\in \gc$ and consider the coset
\be \Exp{\eps \b_-\oplus \eta \b_-} \left( \prodr_{{\ii b m}\in\A} \Exp{x^{\ia b m} J_{\ia b m}}\right) \Exp{\eps A} \Exp{\eta B} \Exp{-\eps A}\Exp{-\eta B}.\label{encoset}\ee
On the one hand $\Exp{\eps A} \Exp{\eta B} \Exp{-\eps A}\Exp{-\eta B}  = \Exp{\eps\eta[A,B]}$ and hence \cref{encoset} is equal to 
\be   \Exp{\eps \b_-\oplus \eta \b_-}  \prodr_{{\ii b m}\in\A} \Exp{\left( x^{\ia b m} + \eps\eta P^{\ia b m}_{[A,B]}(x) \right) J_{\ia b m}} 
= 
 \Exp{\eps \b_-\oplus \eta \b_-}  \prodr_{{\ii b m}\in\A} \Exp{\left( x^{\ia b m} + \eps\eta \sum_{\gamma} P^\gamma_{[A,B]}(x) D_\gamma \on x^{\ia b m} \right) J_{\ia b m}} 
 .\nn\ee
On the other hand, we see that 
\begin{align}
&\Exp{\eps \b_-\oplus \eta \b_-} \left( \prodr_{{\ii b m}\in\A} \Exp{x^{\ia b m} J_{\ia b m}}\right) \Exp{\eps A} \Exp{\eta B} \nn\\
&=  \Exp{\eps \b_-\oplus \eta \b_-} \prodr_{{\ii b m}\in\A} 
\Exp{\left( x^{\ia b m} + \eps P_A^{\ia b m}(x) + \eta P^{\ia b m}_B(x+ \eps P_A(x)) \right) J_{\ia b m}} \nn\\
& = \Exp{\eps \b_-\oplus \eta \b_-} \prodr_{{\ii b m}\in\A} 
\Exp{\left( x^{\ia b m} + \eps P_A^{\ia b m}(x) + \eta P^{\ia b m}_B(x) + \eps\eta \sum_{{\ii a n}} 
\left(P_A^{\ia a n}(x) (D_{\ia a n} P^{\ia b m}_B)(x) \right) 
 \right) J_{\ia b m}}\nn
\end{align}
and hence \cref{encoset} is equal to
\begin{align} 
&\Exp{\eps \b_-\oplus \eta \b_-} \prodr_{{\ii b m}\in\A} 
\Exp{\left( x^{\ia b m} + \eps\eta \sum_{{\ii a n}} 
\left(P_A^{\ia a n}(x) (D_{\ia a n} P^{\ia b m}_B)(x) - P_B^{\ia a n}(x) (D_{\ia a n} P^{\ia b m}_A)(x) \right) 
 \right) J_{\ia b m}}\nn\\
&=\Exp{\eps \b_-\oplus \eta \b_-} \prodr_{{\ii b m}\in\A} 
\Exp{\left( x^{\ia b m} 
+ \eps\eta \left[ \sum_{{\ii a n}} P_A^{\ia a n}(x) D_{\ia a n}, \sum_{\gamma} P_B^\gamma(x) D_\gamma\right] \on x^{\ia b m} \right)J_{\ia b m}}
\end{align}
This is true for all $x^{\ia b m}$. We conclude that 
\be \left[ \sum_{{\ii a n}} P_A^{\ia a n}(X) D_{\ia a n}, 
           \sum_{\ii b m} P_B^{\ia b m}(X) D_{\ia b m}\right] = 
 \sum_{\ii a n} P^{\ia a n}_{[A,B]}(X) D_{\ia a n},\nn\ee
so that the map $\vf$ is indeed a homomorphism of Lie algebras. The polynomial $P^{\ia b m}_A$ belongs to $ \Onp_{\wgt{\ii b m} - \wgt{\ii a n}}$ whenever $A\in \g_{\wgt\ii a n}$, so $\vf$ respects the $Q$-gradation. Hence it respects the principal $\ZZ$-gradation and is therefore continuous. 
\end{proof}
\begin{rem}
Just as in the case of $\g$ of finite type, there is also another realization $\vf^L:\npc\to \Derc\Onp$ of $\npc$, coming from the \emph{left} action of $U$ on itself. The image $\vf^L(\npc)\subset\Derc\Onp$ lies in the centralizer of $\vf(\npc)$, but not of $\vf(\gc)$. In the case of $\g$ of finite type, $\rho^L$ plays an important role in the definition of screening operators \cite{FFGD,FrenkelBenZvi,Fre07}.
\end{rem}

In what follows, we get more information about image of the homomorphism $\vf$.

\subsection{Bounded grade and widening gap}\label{sec: Dw}
Let us say a collection of polynomials $\{P^{a,n}(X)\}_{(a,n) \in \A}$ in $\Onp$ has \emph{widening gap} if, for every $K\geq 1$, we have 
\be P^{a,n}(X) \in \CC\!\left[X^{b,m}: m < n - K, b\in \I\right] \label{Dwcond}\ee
(for all $a\in \I$) for all but finitely many $n\in \ZZ_{\geq 0}$.  
Let us say that an element 
\be v =  \sum_{{\ii a n}\in \A} P^{\ia a n}(X) D_{\ia a n} \in \Derc\Onp \nn\ee
has \emph{widening gap} if the collection of coefficient polynomials $\{P^{a,n}(X)\}_{(a,n) \in \A}$ has widening gap.
Let us say that $v$ has \emph{bounded grade} if there exists $M\in \ZZ$ such that
\be P^{{\ii a n}}(X) \in \bigoplus_{k=-M}^{M} \Onp_{[n + k]} \nn\ee
for all ${\ii a n}\in \A$, where we denote by $\Onp = \bigoplus_{k\in \ZZ} \Onp_{[k]}$ the $\ZZ$-gradation of $\Onp$ in which $X^{{\ii a n}}\in \Onp_{[-n]}$. 

\begin{exmp}
For any $a,b,c\in \I$:
\begin{enumerate}[(i)]
\item $\sum_{n\in \NN} X^{a,n} X^{b,n} D_{c,2n}$ has bounded grade and widening gap,
\item $\sum_{n\in \NN} X^{a,1} D_{b,n}$ has widening gap but not bounded grade, 
\item $\sum_{n\in \NN} X^{a,n} D_{b,n}$ has bounded grade but not widening gap.
\end{enumerate}
\end{exmp}

Define $\Dw$ to be the subset of $\Derc\Onp$ consisting of elements with bounded grade and widening gap. It is a Lie subalgebra (and an $\Onp$-submodule) of $\Derc\Onp$. (See \cref{lem: dercdw} and its proof, below.)

The following lemma says that any element of bounded grade ``almost'' has widening gap.
\begin{lem}\label{Wlem}
Suppose $\sum_{{\ii a n}\in \A} P^{\ia a n}(X) D_{\ia a n}\in \Derc\Onp$ has bounded grade. For every $K$, we eventually (i.e. for all but finitely many $n$) have  
\be P^{\ia a n}(X) \in \Onp_{<n - K} \ox \left(\CC\oplus \bigoplus_{\substack{\ii b m \in \A \\ m\geq  n - K}} \CC X^{\ia b m} \right).\nn\ee 
\end{lem}
\begin{proof}
Suppose to the contrary that for some $K$ there is, for every $N$, always an $n>N$ such that $P^{\ia a n}(X)$ (for some $a\in \I$) has a nonzero monomial $m(X)$ with $X^{\ia b r} X^{\ia c s}$ as a factor, for some $r,s \geq n-K$, $b,c\in\I$. The grade of $m(X) D_{\ia a n}$ is bounded above by $-r-s+n \leq  2K-n < 2K -N$. So for every $N$ we find terms of grade less than $2K-N$.
\end{proof}

This applies in particular to the image $\vf(J_{\ia a n}) \in \Derc\Onp$ of one of the generators of $\g'$ (which is certainly of bounded grade: it is in grade $n$, since homomorphism $\vf: \gc \to \Derc\Onp$ respects the $Q$-gradation and hence the homogeneous $\ZZ$-gradation).

It is natural to try to isolate the terms in $\vf(J_{\ia a n})$ which fail to have widening gap. This is the content of \cref{linthm} below.
Let $f_{ab}{}^c$ denote the structure constants of $\oc\g$ in our basis $\{J_a\}_{a\in \I}$:
\be \left[ J_a, J_b\right] = \sum_{c\in \I} 
f_{ab}{}^c J_c .\nn\ee
For every $(a,n) \in \Ag$, define the element $\Jp_{a,n}\in \Derc\Onp$ by 
\be \Jp_{a,n} := \sum_{b,c\in \I} 
f_{ba}{}^c \sum_{m > \max(1,n) } X^{b,m-n} D_{c,m} \label{Jpdef}\ee
\begin{thm}\label{linthm} For all $(a,n) \in \Ag$, 
\be \vf(J_{a,n}) - \Jp_{a,n} \in \Dw,\nn\ee
i.e. this difference has widening gap.
\end{thm}
\begin{proof} For the entirety of this proof, pick and fix some $\ii a n \in \A$. The element $\vf(J_{a,n}) - \Jp_{a,n}$ is in grade $n$.
The non-trivial thing we have to check is that it has widening gap. 
We have
\be \vf(J_{a,n}) = \sum_{\ii b m\in \A} P^{\ia b m}(X) D_{\ia b m} ,\nn\ee
for some polynomials $P^{\ia b m}(X)$. 
As in \cref{Plem}, these polynomials are given by the demand that,  for any group element 
\be g=\prodr_{{\ii b m}\in\A} \Exp{x^{\ia b m} J_{\ia b m}} \in U \subset U(\CC_\eps),\nn\ee 
we have, for some  $B\in \b_-$, the following equality in the group $\exxp{\f_\eps}$:
\be \left( \prodr_{\ii b m\in\A} \Exp{x^{\ia b m} J_{\ia b m}}\right) \Exp{\eps J_{\ia a n}} = \Exp{\eps B }
 \prodr_{{\ii b m}\in\A} \Exp{\left(x^{\ia b m} + \eps P^{\ia b m}(x)\right)J_{\ia b m}}.\label{gterma}\ee

Given any positive integer $K$, we can always find a positive integer $N$ large enough that $2(N-K) + \min(n,0) > N$. Consider any $m > N$, so that we have 
\be 2(m-K) + \min(n,0) > m. \label{Mbnd}\ee
To compute $P^{\ia b m}(X)$ in \cref{gterma} it is enough to compute the truncation,  
\be \pi_{m+1}\left( \pi_p(g) \Exp{\eps J_{a,n}}\right), \label{gterm}\ee
where $p = m+1 - \min(n,0)$. 
We can factor the truncation $\pi_p(g)$ as follows:
\be \pi_p(g) = g_{<m-K} g_{\geq m-K} \nn\ee
where $g_{<m-K} = \pi_{m-K}(g)$ is the further truncation, and where $g_{\geq m-K}$ are the remaining factors in $\pi_p(g)$, i.e. some product of $\Exp{x^{\ia d s} J_{\ia d s}}$ with $p\geq s\geq m-K$, $d\in \I$. 

Consider what happens as we use \cref{BCHlem} to move $\Exp{\eps J_{\ia a n}}$ leftwards through the product in \cref{gterm}. We arrive at an equality of the following form:
\be \pi_{m+1}\left( g_{<m-K} g_{\geq m-K}\Exp{\eps J_{a,n}}\right) 
=  \pi_{m+1}\left( g_{<m-K} \Exp{\eps J_{a,n}} g'_{\geq m-K} \right)\label{hwt}\ee
for some 
\be g'_{\geq m-K} = \prodr_{\substack{\ii d s \in \A\\ m\geq s \geq m-K}} \Exp{\left( x^{\ia d s} + \eps y^{\ia d s}\right) J_{\ia d s}}.\nn\ee
Here certainly $y^{\ia d s} \in \CC[x^{\ia e r}]_{\ii e r \in \A; p\geq r \geq m-K}$. Moreover, these  $y^{\ia d s}$ are at most linear:
indeed, our choice \cref{Mbnd} ensures that for any $t,s > m-K$, a nested commutator 
\be [J_{f,t},[J_{e,s},J_{a,n}]] \nn\ee
has grade $t+s+n > 2(m-K) + n > m$ and so does not contribute in $\pi_{m+1}(g)$; a fortiori, nested commutators of $J_{a,n}$ with three or more generators of grade $>m-K$ cannot contribute.
In this we way see that 
\be y^{\ia b m} = f_{ca}{}^b x^{\ia c {m-n}};\nn\ee
it arises from the step 
\be \Exp{x^{c,m-n} J_{c,m-n}} \Exp{\eps J_{a,n}} = \Exp{\eps J_{a,n}} \Exp{(\eps x^{c,m-n} f_{ca}{}^b J_{b,m} + \dots)}  \Exp{x^{c,m-n} J_{c,m-n}}.\nn\ee

It remains to show that, starting from the expression on the right of \cref{hwt}, as we continue to reshuffle terms in \cref{hwt} using \cref{BCHlem} until they are all in order, we never again shift $x^{\ia b m}$ by any quantity that depends on any of the $x^{\ia d s}$ with $s\geq m-K$.
To see that, note first that we have
\be  \pi_{m+1}\left( g_{<m-K} \Exp{\eps J_{a,n}} g'_{\geq m-K} \right) 
   = \pi_{m+1}\left( \pi_{m+1}\left(g_{<m-K} \Exp{\eps J_{a,n}}\right) g'_{\geq m-K} \right).\nn\ee 
Here we can use \cref{BCHlem} to re-write
\be \pi_{m+1}(g_{<m-K} \Exp{\eps J_{a,n}}) = \pi_{m+1}(\Exp{\eps B} g'_{<m-K} g'') \nn\ee
for some $B\in \b_-$ and some 
\be g'_{<m-K} = \prodr_{\substack{\ii d s \in \A\\ m-K > s \geq 0}} 
\Exp{\left( x^{\ia d s} + \eps y^{\ia d s}\right) J_{\ia d s}}\qquad\text{and}\qquad
g'' = \prodr_{\substack{\ii d s \in \A\\ m \geq s \geq m-K}} 
\Exp{ \eps z^{\ia d s} J_{\ia d s}}.\nn\ee
Here $z^{d,s} \in \CC[x^{\ia e r}]_{\ii e r \in \A; m-K> r \geq 0}$; to stress the point, $z^{d,s}$  manifestly cannot depend on any of the $x^{\ia d s}$ with $s\geq m-K$. At this stage we have arrived at 
\begin{align} \pi_{m+1}\left( g_{<m-K} g_{\geq m-K}\Exp{\eps J_{a,n}}\right) 
&= \pi_{m+1}\left(\Exp{\eps B} g'_{<m-K} g'' g'_{\geq m-K} \right)\nn\\
&= \pi_{m+1}\left(\Exp{\eps B} g'_{<m-K} \pi_{m+1}(g'' g'_{\geq m-K}) \right)
.\nn
\end{align}
All that remains is to put the terms in 
\be \pi_{m+1}(g'' g'_{\geq m-K}) \nn\ee
in the correct order. But since $2(m-K) > m$, as in \cref{Mbnd}, we again have on grading grounds that all commutators we produce as we use \cref{BCHlem} have too high grade to contribute in this truncation, and we simply get
\be  \pi_{m+1}(g'' g'_{\geq m-K})  
=   \prodr_{\substack{\ii d s \in \A\\ m \geq s \geq m-K}} 
\Exp{\left(x^{\ia d s} + \eps y^{\ia d s} + \eps z^{\ia d s}\right) J_{\ia d s}}. \nn\ee

This completes the proof that, for any positive integer $K$, we have that for all sufficiently large $m$,
\be P^{\ia b m}(x) = f_{ca}{}^b x^{\ia c {m-n}} \mod \CC[x^{\ia e r}]_{r<m-K} .\ee
That, is $\vf(J_{a,n}) - \Jp_{a,n}$ has widening gap, which is what we had to show.
\end{proof}
It will be helpful to have a name for the difference in \cref{linthm}: let us write
\be \vf(J_{a,n}) - \Jp_{a,n} =: \sum_{(b,m)\in \A} \Rp^{b,m}_{a,n}(X) D_{b,m}\label{Rpdef}\ee
for certain polynomials $\Rp$. 

\section{Cartan involution and doubling trick}\label{sec: cai}
\subsection{Definition of $\Dg$}\label{sec: wg}
Let us introduce the polynomial algebra
\be \Og := \CC[X^{\ia a n}]_{\ii a n \in \Ag} \nn\ee
and define the vector spaces
\be    \Der\Og := \bigoplus_{\ii a n \in \Ag} \Og D_{\ia a n},
\qquad \Derc\Og := \prod_{\ii a n \in \Ag} \Og D_{\ia a n} .\nn\ee
Let us say a collection of polynomials $\{P^{a,n}(X)\}_{(a,n) \in \Ag}$ in $\Og$ has \emph{widening gap} if, for every $K\geq 1$, we have, 
\be P^{a,n}(X) \in \CC\!\left[X^{b,m}: |m| < |n| - K, b\in \I\right]. \label{Dwcond}\ee
(for all $a\in \I$) for all but finitely many $n\in \ZZ$.  
\begin{lem}\label{lem: wg} 
Equivalently, the polynomials $\{P^{a,n}(X)\}_{(a,n) \in \Ag}$ have widening gap if and only if, for every $K\geq 1$ there is some $B(K)$ such that
\be P^{a,n}(X) \in \CC\!\left[X^{b,m}: |m| < \max(|n| - K, B(K)), b\in \I\right] \label{Dwcondprime}\ee
for all $\ii a n \in \Ag$. 
\end{lem}
\begin{proof} Suppose $\{P^{a,n}(X)\}_{(a,n) \in \Ag}$ have widening gap. Then for all $K\geq 1$ there exists some $m(K)$ such that whenever $|n|> m(K)$,  $P^{a,n}$ has no factor $X^{c,m}$ with $|m|> |n|- K$. But that leaves only finitely many polynomials $P^{a,n}$, and we can find $B(K)$ such that none of them has any factor $X^{c,m}$ with $|m| > B(K)$. Conversely, suppose $\{P^{a,n}(X)\}_{(a,n) \in \Ag}$ obeys the condition in the lemma. Then for all $K\geq 1$,  $|n|> B(K)$ for all but finitely many $n$, so $\{P^{a,n}(X)\}_{(a,n) \in \Ag}$ has widening gap.
\end{proof}
\begin{cor}\label{cor: wg}
Suppose we are given collections $\{P^{a,n}\}$, $\{Q^{a,n}\}$, $\{R^{a,n}\}$, \dots, of polynomials with widening gap. 
Then 
\be \sum_{(b,m)\in \Ag} P^{b,m} \frac{\del Q^{a,n}}{\del X^{b,m}},\qquad \ii a n \in \Ag, \nn\ee
is a collection of polynomials with widening gap (and, hence, so is
\be \sum_{(b,m),(c,p)\in \Ag} P^{b,m} \frac{\del Q^{c,p}}{\del X^{b,m}} 
             \frac{\del R^{a,n}}{\del X^{c,p}}, \qquad \ii a n \in \Ag, \nn\ee
and so on).
\qed
\end{cor}
For future use, let us note also the following.
\begin{cor}\label{cor: wgtree}
More generally, if $\{P^{a,n}_i\}_{(a,n)\in \Ag}$, $i=1,\dots,N$ and $\{Q^{a,n}\}$ are collections of polynomials with widening gap then
\be \sum_{(b_1,m_1),\dots,(b_N,m_N)\in \Ag} P_1^{b_1,m_1} \dots P_N^{b_N,m_N} 
\frac{\del^N Q^{a,n}}{\del X^{b_1,m_1}\dots\del X^{b_N,m_N}},\qquad \ii a n \in \Ag, \nn\ee
is a collection of polynomials with widening gap.
\qed
\end{cor}
\begin{cor}\label{cor: wgcycle}
If $\{Q^{a,n}\}$ is a collection of polynomials with widening gap then the sum 
\be \sum_{(a,n)\in \Ag} \frac{\del Q^{a,n}}{\del X^{a,n}} \nn\ee
has only finitely many non-zero terms, and hence is a well-defined polynomial in $\Og = \CC[X^{a,n}]_{(a,n)\in\Ag}$.  
\qed
\end{cor}
Let $\Dwg\subset \Derc\Og$ denote the subspace consisting of elements of the form 
\be \sum_{(a,n)\in \Ag} P^{a,n}(X) D_{a,n} \nn\ee
such that the polynomials $\{P^{a,n}(X)\}_{(a,n) \in \Ag}$ have widening gap. 
(We shall also refer to elements of $\Dwg$ themselves as having widening gap.)
Evidently, we have 
\be \Der\Og \subset \Dwg \subset \Derc\Og \nn\ee
and the Lie algebra $\Dw$ from \cref{sec: Dw} embeds in $\Dwg$.

\begin{lem}\label{lem: dercdw}
$\Derc\O$ is a Lie algebra and $\Dwg$ is a Lie subalgebra. 
\end{lem} 
\begin{proof} The Lie bracket of the Lie algebra $\Der\O$, given by
\begin{align} &\left[\sum_{(a,n)\in \Ag} P^{a,n}(X) D_{a,n},
                    \sum_{(b,m)\in \Ag} Q^{b,m}(X) D_{b,m}\right]\nn\\
&\qquad\qquad  = \sum_{(a,n) , (b,m)\in \Ag} 
\left( P^{b,m} \frac{\del Q^{a,n}}{\del X^{b,m}} - Q^{b,m} \frac{\del P^{a,n}}{\del X^{b,m}} \right) D_{a,n} ,\nn
\end{align} 
extends to a well-defined Lie bracket on $\Derc\O$: for each $(a,n)$ on the right, the sum on $(b,m)$ contains only finitely many non-zero terms since $Q^{a,n}$ and $P^{a,n}$ are polynomials. 
\Cref{cor: wg} then implies $\Dwg$ is a Lie subalgebra.
\end{proof}

Let $\{\SS^a_{b,n}\}_{a,b\in \I, n\in \ZZ}$ denote the generators of the loop algebra $\Lglog$, obeying the commutation relations
\be \left[\SS^a_{b,n},\SS^c_{d,m}\right] = \delta^c_b \SS^a_{d,n+m}- \delta^a_d \SS^c_{b,n+m}. \label{Lgl}\ee
Let $\Cocent$ denote the derivation element for the homogeneous gradation of this loop algebra. By definition, it obeys
\be [\Cocent, \SS^a_{b,n}] = n \SS^a_{b,n} .\nn\ee
We have the homomorphism $\lad: \Log\to \Lglog$ given by
\ifdefined\short
\else
\footnote{Indeed, 
\begin{align} 
[\J_{a,n},\J_{b,m}] 
&= f_{ca}{}^d f_{eb}{}^{f} [\SS^c_{d,n},\SS^e_{f,m}] \nn\\
&= f_{ca}{}^d f_{eb}{}^{f} \left(\delta^e_d \SS^c_{f,n+m} - \delta^c_f \SS^e_{d,n+m}\right) \nn\\
&= f_{ca}{}^d f_{db}{}^{f} \SS^c_{f,n+m} - f_{ca}{}^d f_{eb}{}^c \SS^e_{d,n+m} \nn\\
&= -f_{ab}{}^d f_{dc}{}^{f} \SS^c_{f,n+m} = f_{ab}{}^d \J_{d,n+m}\nn
\end{align}
}
\fi
\be J_{a,n}\mapsto \J_{a,n} := \sum_{b,c\in \I} f_{ba}{}^c \SS^b_{c,n},\qquad n\in \ZZ.\label{Jdef}\ee
and hence the homomorphism $\lad: \g \to \Lglog \rtimes \CC\Cocent$, with
\be \cent \mapsto 0 \quad\text{and}\quad \cocent \mapsto \Cocent.\nn\ee

There is an embedding $\iota:\Lglog \rtimes \CC\Cocent \into \Derc\Og$ given by
\be \iota(\SS^b_{c,n}) := \sum_{m\in \ZZ}  X^{b,m-n} D_{c,m},\qquad \iota(\Cocent) := \sum_{a\in \I,m\in \ZZ}  m X^{a,m} D_{a,m}.\nn\ee
By means of this embedding, $\Lglog \rtimes \CC\Cocent$ acts (via the adjoint action) on $\Derc\Og$.
\begin{lem}\label{lem: stabdw} This action stabilizes the Lie subalgebra $\Dwg$. 
\end{lem}
\begin{proof} Suppose $\sum_{(a,n)\in \Ag} P^{a,n}(X) D_{a,n}\in \Dwg$. Pick any $K\geq 1$ and any generator $\SS^b_{c,m}$ of $\Lglog$ (it is clear that $\Dwg$ is stable under $\Cocent$). Since the $\{P^{a,n}\}$ have widening gap, eventually $P^{a,n}$ has no factor $X^{d,p}$ with $|p| > |n| - K - |m|$. Therefore $\left[\SS^b_{c,m} , \sum_{(a,n)\in \Ag} P^{a,n}(X) D_{a,n} \right]$ again has widening gap. (Note that for this argument to work it is necessary that the gap is really \emph{widening}, i.e. that the condition in \cref{Dwcond} is that \emph{for every $K$} eventually $|m|< |n| -K$.)
\end{proof}

Let us define $\Dg$ as the corresponding semi-direct product of Lie algebras, 
\be \Dg := \Dwg \rtimes \left(\Lglog \rtimes \CC\Cocent\right) . \label{Dgdef}\ee

\subsection{Cartan involution}
Now let $\tau: \Derc\Og \to \Derc\Og$ be the involutive (i.e. $\tau^2 = \id$) automorphism defined by
\begin{alignat}{2} 
\tau(X^{\ia \alpha n}) &= X^{\ia{-\alpha}{-n}},\quad& \tau(X^{\ia i n}) &= - X^{\ia i {-n}},\nn\\
\tau(D_{\ia \alpha n}) &= D_{\ia{-\alpha}{-n}},\quad& \tau(D_{\ia i n}) &= - D_{\ia i {-n}},
\end{alignat}
for $\alpha\in \oc\Delta\setminus\{0\}$, $i\in \oc I$ and $n\in \ZZ$.
Recall $\Onp = \CC[X^{\ia a n}]_{\ii a n \in \A}$ where $\A := \{(\alpha,0)\}_{\alpha\in \oc\Delta_+} \,\,\cup\,\, \Ag_{\geq 1}$ indexes a basis of $\np$. Let us introduce 
$\Onm := \CC[X^{\ia a n}]_{\ii a n \in \Am}$, where $\Am := \{(\alpha,0)\}_{\alpha\in \oc\Delta_-} \,\,\cup\,\, \Ag_{\leq -1}$. We define $\Derc\Onm$ by obvious analogy with $\Derc\Onp$.
Clearly, $\tau(\Derc\Onp) = \Derc\Onm$ and $\tau(\Derc\Onm) = \Derc\Onp$. 
Let $\cai:\g \to \g$ be the Cartan involution of the affine Kac-Moody algebra $\g$, defined by
\be \cai(e_i) = f_i,\quad \cai(f_i) = e_i,\quad \cai(H) = -H,\label{caidef}\ee
for $i\in I$ and $H\in \h$. One has
\be \cai(J_{\alpha,n}) = J_{-\alpha,-n},\quad \cai(J_{i,n}) = -J_{i,-n},\nn\ee
for $\alpha\in \oc\Delta\setminus\{0\}$, $i\in \oc I$ and $n\in \ZZ$.
We get a homomorphism of Lie algebras
\be \tau \circ \vf \circ \cai : \g \to \Derc\Onm \nn\ee
and hence a homomorphism of Lie algebras
\begin{align} \vfd := (\vf + \tau\circ\vf\circ\cai) : \g 
&\to \Derc\Onp \oplus \Derc\Onm\nn\\ 
&\into \Derc\Og.\label{vfddef}
\end{align}
By construction, $\vfd$ has the equivariance property
\be \tau\circ\vfd = \vfd \circ \cai .\label{vfdequi}\ee

Note that
\be \iota(\J_{a,n}) = (\iota\circ\lad)(J_{a,n}) = \sum_{b,c\in \I} f_{ba}{}^c \sum_{m\in \ZZ}  X^{b,m-n} D_{c,m} \in \Derc\Og.\nn\ee
\begin{lem}\label{lem: Jdif}
For all $(a,n)\in \I\times \ZZ$, 
\be \vfd(J_{a,n}) - \iota(\J_{a,n}) \in \Dwg. \nn\ee
\end{lem}
\begin{proof}
Recall the definition \cref{Jpdef} of $\Jp_{a,n}$. Let us give a name to the linear map 
$f_+ : \g \to\Derc\Onp; J_{a,n} \to \Jp_{a,n}$. The key observation is that the difference
\be (f_+ + \tau\circ f_+ \circ \sigma)(J_{a,n}) - \iota(\J_{a,n}) \label{Jdif}\ee
is a finite sum of terms of the form $X^{b,m} D_{c,p}$, and thus an element of $\Der\Og\subset \Dwg$. 

Together with \cref{linthm}, this implies the result.
\end{proof}

It follows from \cref{lem: Jdif} that the image $\vfd(\g)$ of $\g$ in $\Derc\Og$ lies in the embedded copy of $\Dg$. In what follows we shall regard $\vfd$ as a homomorphism 
\be \vfd : \g \to \Dg \label{vfddef}\ee
into the abstract copy of $\Dg$ we defined in \cref{Dgdef}. 

Observe that the difference $(f_+ + \tau\circ f_+ \circ \sigma)(J_{a,n}) - \iota(\J_{a,n})$ in \cref{Jdif} generically contains terms which do not stabilize both $\Onp$ and $\Onm$. (For example, terms like $X^{a,5} D_{b,-7}$ or $X^{a,-3} D_{b,8}$.) 
For that reason, it is worth stressing the following crucial property (which is true by construction).
\begin{prop}\label{stabprop}
$\vfd(\g)$ stabilizes $\Onp$ and $\Onm$ in $\Og$.
\qed\end{prop}
This is contrast to the obvious homomorphism $\g \to\Derc\Og$ sending $J_{a,n}\mapsto \iota(\J_{a,n})$ (and $\cent\mapsto 0$, $\cocent \mapsto \sum_{n} n X^{a,n} D_{a,n}$). So what we have shown is that it is possible to add, to each generator $\iota(\J_{a,n})$, an infinite sum belonging to $\Dwg$ of ``correction terms'', in such a way that the resulting action does stabilize $\Onp$ and $\Onm$. 
It will be useful to have a name for these ``correction terms''. Let us write
\be \vfd(J_{a,n}) - \iota(\J_{a,n}) = \sum_{(b,m)\in \I\times \ZZ} R_{a,n}^{b,m}(X) D_{b,m}, \label{Rdef}\ee
where the collection $\{R_{a,n}^{b,m}(X)\}_{\ii bm \in \Ag}$ of polynomials has widening gap, for each $(a,n)\in \Ag$.

\section{Vertex algebras and main results}\label{sec: va}

\subsection{Weyl algebra $\Hg$}\label{sec: Heis}
Let $\Hg$ denote the associative unital $\CC$-algebra obtained by quotienting the free associative unital $\CC$-algebra with generators 
\be \bet_{{\ia a n}}[N],\quad \gam^{{\ia a n}}[N],\nn\ee 
with $N\in \ZZ$ and $\ii a n \in \Ag$, by the two-sided ideal generated by the commutation relations 
\be [\bet_{\ia a n}[N],\bet_{\ia b m}[M]] = 0,\quad 
[\bet_{\ia a n}[N],\gam^{\ia b m}[M]] = \delta_{\ia a n}^{\ia b m} \delta_{N,-M} \bm 1,\quad 
[\gam^{\ia a n}[N],\gam^{\ia b m}[M]] = 0.\nn\ee


As a vector space,
\begin{align} \Hg &\cong_\CC \Hg^- \ox \Hg^+,  \nn
\end{align}
where
\be \Hg^- \cong \CC\left[ \gam^{\ia a n}[N], \bet_{\ia a n}[N-1] \right]_{N \leq 0; \ii a n \in \Ag}\label{Hm}\ee
is the algebra of  \emph{creation operators} and 
\be \Hg^+\cong \CC\left[ \gam^{\ia a n}[N], \bet_{\ia a n}[N-1] \right]_{N > 0; \ii a n \in \Ag}\label{Hp}\ee
is the algebra of  \emph{annihilation operators}. The fact that $\Hg^+$ and $\Hg^-$ are commutative (and that $\Hg$ is commutative modulo $\bm 1$) makes this an example of a \emph{system of free fields}. 

\subsection{Fock module $\Mh$}
Define $\Mh$ to be the induced $\Hg$-module generated by a vector $\vac$, the \emph{vacuum}, annhilated by $\Hg^+$, i.e.
\be \bet_{\ia a n}[M]\vac = 0,\quad M\in \ZZ_{\geq 0}, \qquad 
    \gam^{\ia a n}[M]\vac = 0,\quad M\in \ZZ_{\geq 1} \label{Mdef}\ee
for all ${\ii a n}\in \Ag$, and on which $\bm 1\vac = \vac$. Vectors in $\Mh$ are called \emph{states}. 

There is an obvious $\ZZ\times Q$-gradation 
of $\Hg$ and of $\Mh$ in which $\bet_{\ia a n}[N]$ has grade $(N,\alpha)$ 
and  $\gam^{\ia a n}[N]$ has grade $(N,-\alpha)$ 
whenever $J_{\ia a n} \in \g_\alpha$, and $\vac$ has grade $(0,0)$. 

Call the first factor here the \emph{depth gradation}. Let $\Mh[N]$ denote the subspace of depth $n$, so that
\be \Mh = \bigoplus_{n=0}^\8 \Mh[N] .\label{depthgrad}\ee 
Call the corresponding filtration, given by
\be \Mh\df m := \bigoplus_{n=0}^m \Mh[N] ,\nn\ee
the \emph{depth filtration}. 
Every  $v\in \Mh$ belongs to $\Mh\df m$ for some $m$. 


\subsection{The subspace $\Mh\df 1$}
Let us introduce the space 
\be \Omega_\Og := \Hom_\Og(\Der\Og,\Og).\nn\ee
It comes equipped with the derivative 
$\del: \Og \to \Omega_\Og$ 
defined by $(\del f)(v) = v(f)$. 
It is a free left $\Og$-module, with an $\Og$-basis consisting of basis vectors $\del X^{\ia a n}$, $\ii a n\in \A$, which obey $\del X^{\ia a n}(D_{\ia b m}) = \delta^{\ia a n}_{\ia b m}$: 
$\Omega_\Og \cong_\Og \bigoplus_{\ii a n\in \A} \Og \del X^{\ia a n}$. 
There is an action of $\Der\Og$ on $\Omega_\Og$, given by $v \on (f\del g) = (v\on f) \del g + f \del(v\on g)$. 
This action extends to a well-defined action of $\Derc\Og$.
The space $\Omega_\Og$ is graded by the root lattice $Q$, $\del$ has grade 0, and these structures respect the $Q$-grading. 

The subspace $\Mh[0]$ consists of states of the form $P(\gam[0]) \vac$ where $P(X) \in \Og$.
Meanwhile the subspace $\Mh[1]$ consists of states of the form
\be \sum_{\ii a n \in \Ag}P^{\ia a n}(\gam[0]) \bet_{\ia a n}[-1]\vac 
  + \sum_{\ii a n \in \Ag}Q_{\ia a n}(\gam[0]) \gam^{\ia a n}[-1]\vac \nn\ee
for polynomials $P^{\ia a n}(X), Q_{\ia a n}(X) \in \Og$, $\ii a n \in \Ag$.  
Thus, there are isomorphisms of vector spaces,
\begin{align} \Mh[0] &\cong \Og, \label{Mnought}\\
\Mh[1] &\cong \Omega_\Og \oplus \Der\Og,\label{vdecomp}
\end{align}
the latter 
given by identifying 
$\bet_{\ia a n}[-1]\vac$ with $D_{\ia a n}$ and $\gam^{\ia a n}[-1]\vac$ with $\del X^{\ia a n}$.

\begin{rem}
The subspace $\Mh\df 1$, equipped with the restriction of the vertex algebra structure on $\Mh$ we are about to recall, is an example of a \emph{$1$-truncated vertex algebra}. This in turn makes $\Mh[1]$ into a \emph{vertex $\Og$-algebroid}. See \cite{GMSii, Bressler, GMShomog}. 
\end{rem}

\subsection{Vertex algebra structure}\label{sec: vas}
For every $N\in \ZZ$, there is a linear map \be \Mh \to \End\Mh;\quad A\mapsto A\vap N\nn\ee sending any given state $A$ to its \emph{$N$th mode}, $A\vap N\in \End\Mh$, and these modes can be arranged in a formal series, the \emph{field}
\be Y(A, x) := \sum_{n\in \ZZ} A\vap N x^{-N-1} \in \Hom\left(\Mh, \Mh((x))\right).\nn\ee 
The \emph{state-field map} $Y(\cdot,x): \Mh \to  \Hom\left(\Mh, \Mh((x))\right)$ obeys a collection of axioms that make $\Mh$ into a \emph{vertex algebra}; see e.g. \cite{LLbook,KacVertexAlgBook,FrenkelBenZvi}. 
It is defined as follows. First,  
\begin{alignat}{2}
\bet_{\ia a n}(x) &:= Y\left(\bet_{\ia a n}[-1] \vac,x\right) 
                   &&= \sum_{N\in \ZZ} \bet_{\ia a n}[N] x^{-N-1},\nn\\
\gam^{\ia a n}(x) &:= Y\left(\gam^{\ia a n}[0] \vac,x\right) 
                   &&= \sum_{N\in \ZZ} \gam^{\ia a n}[N] x^{-N} .\nn
\end{alignat}
Next, let us write $f^{(N)}(u) := \frac{1}{N!} \frac{\del^N}{\del u^N} f(u)$. Then $Y( \bet_{\ia b m}[-M] \vac , u ) = \bet_{\ia b m}^{(M-1)}(u)$ and  $Y( \gam^{\ia a n}[-N] \vac , u ) = \gam^{\ia a n (N)}(u)$
and more generally if
\be A = \gam^{\ia{a_1}{n_1}}[-N_1]\dots \gam^{\ia{a_r}{n_r}}[-N_r] 
                \bet_{\ia{b_1}{m_1}}[-M_1]\dots \bet_{\ia{b_s}{m_s}}[-M_s] \vac, \label{monstate}\ee 
then
\begin{multline} A(u) := Y(A,u) =\nord{\gam^{\ia{a_1}{n_1} (N_1)}(u) \dots \gam^{\ia{a_r}{n_r} (N_r)}(u) 
               \bet_{\ia{b_1}{m_1}}^{(M_1-1)}(u)\dots \bet_{\ia{b_s}{m_s}}^{(M_s-1)}(u)}
.\label{Au}
\end{multline}           
Here $\nord{\dots}$ denotes the \emph{normal-ordered product} of fields, which is defined in general as follows. For any states $A,B\in \Mh$, 
\be \nord{Y(A,u) Y(B,v)} \,\,\, := \Biggl(\sum_{M<0} A\vap M u^{-M-1} \Biggr)Y(B,v) + Y(B,v) \Biggl(\sum_{M\geq 0} A\vap M u^{-M-1}\Biggr),\label{nord}\ee 
and $\nord{Y(A, u) Y(B, u)}$ is the specialization at $u=v$.
For more than two fields, the normal ordered product is understood to be right-associative, i.e.
$\nord{Y(A,u) Y(B,v) Y(C,w)} = \nord{Y(A,u)\bigl(\nord{Y(B,v) Y(C,w)}\bigr)}$
and so on.

For the $\bet\gam$-system (and for systems of free fields more generally) there is a simpler definition of the normal ordered product for monomial states like $A$ in \cref{monstate}. Given any monomial $m\in \Hg$, one defines $\nord m\in \Hg$ to be the monomial with the same factors as $m$ but ordered so that all annihilation operators stand to the right of all creation operators (see \cref{Hm,Hp}). Then 
\begin{equation}   \nord{\gam^{\ia{a_1}{n_1} (N_1)}(u_1) \dots \gam^{\ia{a_r}{n_r} (N_r)}(u_r)  \bet_{\ia{b_1}{m_1}}^{(M_1-1)}(v_1)\dots \bet_{\ia{b_s}{m_s}}^{(M_s-1)}(v_s)} 
\nn\end{equation}
is defined by normal-ordering the monomials in the series, term by term. 

The depth gradation, \cref{depthgrad}, makes $\Mh$ into a graded vertex algebra: the vacuum $\vac$ is in grade zero; for any states $A\in \Mh[R]$, $B \in \Mh[S]$ and any $N \in \ZZ$,
$A\vap N  B  \in \Mh[R + S  - N - 1]$.

We have the \emph{translation operator} $T\in \End\Mh$ of the vertex algebra $\Mh$, defined by $TA := A\vap{-2}\vac$. It acts on monomials in the generators of $\Hg$ as a derivation, according to
\be \left[T, \bet_{a,n}[-N]\right] = -N \bet_{a,n}[N-1] ,\qquad \left[T, \gam^{a,n}[-N]\right] = -(N-1)\gam^{a,n}[N-1], \nn\ee
and by definition $T\vac = 0$. 


\subsection{The OPE}
One has the commutator formula for modes of states:
\be \label{com Am Bn}
\bigl[ A\vap M, B\vap N  \bigr] = \sum_{K \geq 0} \binom M K \big( A\vap K B \big)\vap {M+N-K}.
\ee
Here, for all $M\in \ZZ$, $\binom MK := \frac{M(M-1)\dots (M-K+1)}{K!}$ for $K\neq 0$ and $\binom M 0 := 0$.
From 
this and the definition of the normal ordered product, \cref{nord}, one obtains the \emph{operator product expansion (OPE)}:
\be Y(A,u) Y(B,v) = \sum_{N\geq 0} \frac{Y(A\vap N B, v)}{(u-v)^{N+1}} + \nord{Y(A,u) Y(B,v)} \nn\ee
as an equality in $\Hom(\Mh,\Mh((u))((v)))$. Here $(u-v)^{-N-1}$ is understood to be expanded in positive powers of $v/u$.
\ifdefined\short
\else
\footnote{Indeed, in  $\wh\Hg[[u^{\pm 1},v^{\pm 1}]]$, one has
\begin{align} 
Y[A,u] Y[B,v]- \nord{Y[A,u] Y[B,v]} &= \left[ \left(\sum_{M\geq 0} A\vap M u^{-M-1}\right), Y[B,v] \right]\nn\\
&= 
\sum_{M\geq 0} \sum_{N} [ A\vap M, B \vap N ] u^{-M-1} v^{-N-1} \nn\\
&=
\sum_{M\geq 0} \sum_{N} \sum_{P\geq 0} \binom M P \left( A \vap P B\right)\vap{M+N-P} u^{-M-1} v^{-N-1} \nn\\
&=
\sum_{M\geq 0} \sum_{N} \sum_{P\geq 0} \binom M P 
\left( A \vap P B\right)\vap{M+N-P} v^{-N-M+P-1} u^{-M-1} v^{M-P}  \nn\\
&=
\sum_{M\geq 0} \sum_{P\geq 0} \binom M P 
Y[ A \vap P  B,v] u^{-M-1} v^{M-P}  \nn\\
&=
\sum_{P\geq 0} 
Y[ A \vap P  B,v] \frac{1}{P!} \del_v^P \sum_{M\geq 0} u^{-M-1} v^{M}  \nn\\
&=
\sum_{P\geq 0} 
Y[ A \vap P  B,v] \frac{1}{(u-v)^{P+1}}  \nn
\end{align}
}
\fi
Thus, computing the singular terms in the OPE is the same things as computing the non-negative vertex algebra products. 
\subsection{The Wick formula}\label{sec: wick}
For the $\bet\gam$-system (and for systems of free fields more generally) one has the usual \emph{Wick formula} for OPEs of monomial states. 
We take the statement from \cite[\S12.2.6]{FrenkelBenZvi}.

Let $A(u)$ and $B(u)$ be normal-ordered monomials in $\bet_{\ia a n}(u)$, $\gam^{\ia a n}(u)$, $\ii a n \in \Ag$, and their derivatives as in \cref{Au}. A \emph{single pairing} between $A(u)$ and $B(v)$ is a choice, for some fixed $\ii a n\in \Ag$, of
\begin{enumerate}[i)]
\item one factor $\bet_{\ia a n}^{(N)}(u)$ from $A(u)$ and one factor $\gam^{\ia a n (M)}(v)$ from $B(v)$ \\ -- to such a pairing we associate the function $(-1)^N \binom{N+M}{N}\frac{1}{(u-v)^{N+M+1}}$ -- or  
\item  one factor $\gam^{\ia a n (N)}(u)$ from $A(u)$ and one factor $\bet_{\ia a n (M)}(v)$ from $B(v)$ \\ -- to such a pairing we associate the function $(-1)^{M+1} \binom{N+M}{N}\frac{1}{(u-v)^{N+M+1}}$.
\end{enumerate}
A pairing is a disjoint union of zero or more single pairings. To a pairing $P$ we associate the function $f_P(u,v)$ obtained by taking the product the functions associated to each constituent single pairing (or 1, for the empty pairing). Given a pairing $P$ let $(A(u)B(v))_P$ denote the product $A(u)B(v)$ but with all factors belonging to the pairing removed (to leave $1$, in the special case that there are no factors left). The \emph{contraction} $\nord{A(u) B(v)}_P$ associated to a pairing $P$ is by definition  
\be \nord{A(u) B(v)}_P := f_P(u,v) \,\nord{(A(u)B(v))_P}\,\,.\nn\ee

\begin{lem}[Wick formula]\label{Wicklem}  The product $A(u) B(v)$ is equal to the sum of contractions $\nord{A(u) B(v)}_P$ over all pairings $P$ between the monomials $A$ and $B$, counted with multiplicity (and including the empty one). 
\qed\end{lem}
One can then compute the OPE by Taylor-expanding the fields $\bet_{\ia a n}(u)$ and $\gam^{\ia a n}(u)$ about $u=v$. 

\subsection{Completion $\wt\Mh$ of $\Mh$}\label{sec: wtMh}
Having in mind the completion $\Derc\Og$ of $\Der\Og$ introduced in \cref{sec: cai} (and see \cref{sec: derO}), one sees that $\Mh\df 1$ is not quite big enough for our purposes. 
Let us introduce a completion of $\wt\Mh$ of $\Mh$ in such a way as to preserve the depth filtration. To that end, we start by completing each filtered subspace $\Mh\df m$, as follows. 

Let $\Hg^-_{\geq k}$ denote the two-sided ideal in the commutative algebra $\Hg^-$ of creation operators, cf. \cref{Hm}, generated by 
\be \{\bet_{\ia a n}[N]:  a \in \I, |n| \geq k, N\in \ZZ\} .\nn\ee

Let $\I\df m_k$ denote the subspace of $\Mh\df m$ given by
\be \I\df m_k := \Mh\df m \cap \left(\Hg^-_{\geq k}\vac\right) .\nn\ee
Thus, $\I\df m_k$ is the subspace of $\Mh\df m$ spanned by monomials in the creation operators that have some factor $\bet_{\ia a n}[N]$ with $|n|\geq k$.
We have
\be \I\df m_0 \supset \I\df m_1 \supset \I\df m_2 \supset \dots \nn\ee
and $\bigcap_{i=0}^\8 \I\df m_i = \{0\}$ and we define
\be \wt\Mh\df m := \invlim_k \Mh\df m\big/ \I\df m_k .\nn\ee
These completed subspaces $\wt\Mh\df m$ form a directed system, $\wt\Mh\df 0 \subset \wt\Mh\df 1 \subset \wt\Mh\df 2 \dots$, and we define $\wt\Mh$ to be the direct limit
\be \wt\Mh := \varinjlim_m \wt\Mh\df m .\nn\ee
In other words, each element of $\wt\Mh$ is by definition an element of $\wt\Mh\df m$ for some sufficiently large $m$ (depending on the element) with two elements of $\wt\Mh$ considered equal if they are equal in $\wt\Mh\df m$ for some $m$. 

Explicitly, the sum
\be \sum_{\substack{a_1,\dots,a_p\in \I \\ n_1,\dots,n_p\in \ZZ}} 
P^{a_1,n_1}(\gam) \dots P^{a_p,n_p}(\gam) \bet_{a_1,n_1}[-N_1] \dots \bet_{a_p,n_p}[-N_p] \vac,\quad p\in \ZZ_{\geq 0}, \label{Mstate}\ee 
belongs to $\wt\Mh$ if, for each $i\in \{1,\dots,p\}$, $N_i\in \ZZ_{\geq 1}$ and there is a bound on the depth of the polynomials
\be P^{a_i,n_i}(\gam)\in \CC[\gam^{b,m}[-M]]_{b\in \I, m\in \ZZ, M\in \ZZ_{\geq 0}}\nn\ee 
as $(a_i,n_i)$ ranges over $\I\times \ZZ$. Elements of $\wt\Mh$ are finite linear combinations of such sums.

The vertex algebra structure on $\Mh$ does not extend to a well-defined vertex algebra structure on $\wt\Mh$. For example in $\wt\Mh$ we have the state
\be  S:= \sum_{\ii a n \in \Ag} \gam^{\ia a n}[0] \bet_{\ia a n}[-1]\vac \in \wt \Mh .\nn\ee 
If $\g$ were finite-dimensional, this would be a conformal vector.
But here it is an infinite sum, and when we attempt to compute, for example, the action of the would-be first mode of $S$, 
\be S\vap{1} = \sum_{N\geq 0} \sum_{\ii a n \in \Ag} 
    \left( \bet_{\ia a n}[-N] \gam^{\ia a n}[N+1] + \gam^{\ia a n}[-N] \bet_{\ia a n}[N+1]\right)
\nn\ee
on the state $S$, we encounter a double contraction yielding the ill-defined sum
\begin{align}  \sum_{\ii a n \in \Ag}  \bet_{\ia a n}[0] \gam^{\ia a n}[1] \sum_{\ii b m \in \Ag} \gam^{\ia b m}[0] \bet_{\ia b m}[-1]\vac\nn\\
= \sum_{\ii a n \in \Ag} \sum_{\ii b m \in \Ag} \delta_{\ia a n}^{\ia b m} \delta_{\ia b m}^{\ia a n}
= \sum_{\ii a n \in \Ag} 1 .\nn
\end{align} 


\subsection{The vertex algebra $\Mw$}\label{sec: mw}
Let $\Mw\subset\wt\Mh$ denote the subspace of $\wt\Mh$ spanned by states of the form \cref{Mstate} such that, for each $i=1,\dots,p$, the collection of polynomials $\{P^{a_i,n_i}\}_{\ii{a_i}{n_i}\in \Ag}$ has widening gap in same sense as in \cref{sec: wg}, i.e. for every $K\geq 1$, we have 
\be P^{a_i,n_i}(\gam) \in \CC\!\left[\gam^{b,m}[-M]: |m| < |n_i|-K, b\in \I, M\in \ZZ_{\geq 0} \right], \label{Mwcond}\ee
(for all $a_i\in \I$) for all but finitely many $n_i\in \ZZ$.

Evidently $\Mh\subsetneq \Mw \subsetneq \wt\Mh$, and
\begin{align} \wt\Mh[1] &\cong \Omega_\Og \oplus \Derc\Og \nn\\
                 \Mw[1] &\cong \Omega_\Og \oplus \Dwg\nn
\end{align} 
(and $\Mh[0] = \Mw[0] = \wt\Mh[0] \cong \Og$).

\begin{lem}\label{lem: Mw}
The vertex algebra structure on $\Mh$ extends uniquely to a well-defined vertex algebra structure on $\Mw$. 
\end{lem}
\begin{proof} 
The non-trivial thing we have to check is the closure of the OPE, i.e. that the OPE of any two fields of states in $\Mw$ has coefficients that are fields of states in $\Mw$. 
 
This follows from the Wick formula, \cref{sec: wick}. Consider a typical contraction: 
\be \sum_{(a,n)\in \Ag}\wick{\nord{\dots P^{a,n}(\gam)(z) \dots \c\bet_{a,n}(z)} \quad \sum_{(b,m)\in \Ag}\nord{\dots Q^{b,m}(\c \gam)(w) \dots \bet_{b,m}(w) }}. \nn\ee
The resulting sum has, associated to the factor $\bet_{b,m}$, the coefficient polynomial (we suppress the mode numbers $[M]$, which play no essential role here) 
\be \sum_{(a,n)\in \Ag} P^{a,n} \frac{\del Q^{b,m}}{\del \gam^{a,n}} .\nn\ee
As in \cref{cor: wg}, this collection of polynomials again has widening gap, and the same logic extends to longer chains of contractions. More generally, by \cref{cor: wgtree} the same is true in any situation in which the contractions form an acyclic directed graph. (We think of a contraction of $\bet_{a,n}$ into $P^{b,m}(\gam)$ as a directed edge $(a,n) \to (b,m)$.)
And, in view of \cref{cor: wgcycle}, any cycle in the directed graph of the contractions merely gives rise to an overall factor belonging to $\CC[\gam^{c,p}[-M]]_{c\in \I,p\in \ZZ,M\in \ZZ_{\geq 0}}$. For example, in  
\be \sum_{(a,n)\in \Ag}\wick{\nord{\dots P^{a,n}(\c2\gam)(z) \dots \c1\bet_{a,n}(z)} \quad\sum_{(b,m)\in \Ag}\nord{\dots Q^{b,m}(\c1 \gam)(w) \dots \c2\bet_{b,m}(w) }}, \nn\ee
only finitely many $n,m$ can give nonzero summands.
\end{proof}
\begin{rem}
As a vertex algebra, $\Mw$ is generated by the fields $(\gam^{a,n}(z))_{(a,n)\in \Ag}$ together with all the fields 
\be  \sum_{(a,n)\in \Ag} \nord{P^{a,n}(\gam(z)) \bet_{a,n}(z)} \nn\ee
as
\be \left(P^{a,n}(\gam(z))\in \CC[\gam^{b,m\,(N)}(z): (b,m)\in \Ag, N\geq 0]\right)_{(a,n)\in \Ag} \nn\ee
runs over all collections of polynomials that have widening gap and bounded depth.
\end{rem}

\subsection{The vertex algebras $\Mf$ and $\Mfw$}\label{sec: Mfw}
Let $L$ denote the Lie algebra with generators $\{\SS^a_{b,n}[N]\}_{a,b\in \I, n\in \ZZ,N\in \ZZ}$ and $\{\Cocent[N]\}_{N\in \ZZ}$ subject to the relations
\begin{align} \left[\SS^a_{b,n}[N],\SS^c_{d,m}[M]\right] &= \delta^c_b \SS^a_{d,n+m}[N+M]- \delta^a_d \SS^c_{b,n+m}[N+M].\nn\\
 \left[\Cocent[N], \SS^a_{b,n}[M]\right] &= n \SS^a_{b,n}[N+M] .\nn
\end{align}
In other words, $L$ is loop algebra of $\Lglog \rtimes \CC\Cocent$. 
We can take the semi-direct product $L\ltimes \Hg$, where the action of $L$ on $\Hg$ is given by
\begin{alignat}{2} 
\left[\SS^a_{b,n}[N], \bet_{c,m}[M]\right] &= - \delta_c^a \bet_{b,m+n}[N+M],\quad &
\left[\Cocent[N], \bet_{c,m}[M]\right] &= - m \bet_{c,m}[N+M] \nn\\
\left[\SS^a_{b,n}[N], \gam^{c,m}[M]\right] &= \phantom +  \delta_b^c \gam^{a,m-n}[N+M],\quad &
\left[\Cocent[N], \gam^{c,m}[M]\right] &= \phantom +  m \gam^{c,m}[N+M] .\nn
\end{alignat}
Let $\Mf$ denote the module over $L\ltimes \Hg$ induced from a vector $\vac$ obeying the conditions \cref{Mdef} together with
\be \SS^a_{b,n}[N] \vac = 0, \qquad \Cocent[N] \vac = 0,  \qquad N \in \ZZ_{\geq 0}.\nn\ee 

Thus as a vector space $\Mf$ is the tensor product of the Fock module $\Mh$ over $\Hg$ and the vacuum Verma module (at level 0) $\VV_0^{\Lglog\rtimes\CC\Cocent,0}$ over $L$:
\be \Mf \cong_\CC \Mh \ox \VV_0^{\Lglog\rtimes\CC\Cocent,0}. \nn\ee

This module $\Mf$ is a vertex algebra. The state-field map is given by the formulas in \cref{sec: vas} together with
\begin{alignat}{2}
\SS^a_{b,n}(x) &:= Y\!\left( \SS^a_{b,n}[-1]\vac, x\right)
               && = \sum_{N\in \ZZ} \SS^a_{b,n}[N] x^{-N-1},\nn\\
    \Cocent(x) &:= Y\!\left( \Cocent[-1]\vac, x\right) 
               &&= \sum_{N\in \ZZ} \Cocent[N] x^{-N-1}.
\end{alignat}

As a vertex algebra, $\Mf$ is generated (in the sense of the Strong Reconstruction Theorem \cite{FKRW},\cite[\S4.5]{KacVertexAlgBook},\cite[\S4.4]{FrenkelBenZvi}) by the countable collection of fields $\bet_{a,n}(z)$, $\gam^{a,n}(z)$, $\SS^a_{b,n}(z)$ and $\Cocent(z)$, subject to the OPEs 
\begin{align} \bet_{a,n}(z) \gam^{b,m}(w) &= \delta^{b,m}_{a,n} \frac{1}{z-w} + \dots \nn
\end{align}
\begin{alignat}{2} 
\SS^a_{b,n}(z) \gam^{c,m}(w) &= \phantom + \delta^c_b\frac{\gam^{a,m-n}(w)}{z-w} + \dots \qquad &
\SS^a_{b,n}(z) \bet_{c,m}(w) &=          - \delta^a_c\frac{\bet_{b,m+n}(w)}{z-w} + \dots ,\nn\\
           \Cocent(z) \gam^{c,m}(w) &= \phantom +m\frac{\gam^{a,m}(w)}{z-w} + \dots &
\qquad     \Cocent(z) \bet_{c,m}(w) &=          -m\frac{\bet_{b,m}(w)}{z-w} + \dots ,\nn
\end{alignat}
\begin{align} \SS^a_{b,n}(z) \SS^c_{d,m}(w) &=  
        \frac{\delta^c_b \SS^a_{d,n+m}(w)- \delta^a_d \SS^c_{b,n+m}(w)}{z-w} + \dots ,\nn\\
\Cocent(z) \SS^a_{b,n}(w) &=  
        n\frac{\SS^a_{b,n}(w)}{z-w} + \dots .\nn
\end{align}
(with the others being trivial). 

We can enlarge the tensor factor $\Mh$ to $\Mw$ as in \cref{sec: mw} above. 
Let $\Mfw$ denote the resulting vector space,
\be \Mfw \cong_\CC \Mw \ox \VV_0^{\Lglog\rtimes\CC\Cocent,0}. \nn\ee

\begin{lem}
The vertex algebra structure on $\Mf$ extends uniquely to a well-defined vertex algebra structure on $\Mfw$. 
\qed
\end{lem}
\begin{proof}
Given \cref{lem: Mw}, the remaining thing to check is that the OPE of $\SS^a_{b,n}(z)$ with the field of any state in $\Mw$ has coefficients which are again the fields of well-defined states in $\Mw$. This is true by the same reasoning as in \cref{lem: stabdw}. 
\end{proof}

We have the depth gradation on $\Mfw$, in which $\vac$ has grade $0$ and $\gam^{a,n}[N]$, $\bet_{a,n}[N]$, $\SS^a_{b,n}[N]$ and $\Cocent[N]$ contribute grade $-N$. 

We continue to identify
\be \Mfw[0] = \Mh[0] \cong \Og \nn\ee
by identifying $P(X) \in \Og$ with the state $P(\gam[0])\vac$ in $\Mfw[0]$, and to identify $\Omega_\Og$ with the subspace of $\Mfw[1]$ consisting of states of the form 
\be \sum_{a,n\in \Ag} P_{a,n}(\gam[0]) \gam^{a,n}[-1]\vac \in \Mh[1] \subset \Mfw[1],\nn\ee
by identifying this state with $\sum_{a,n\in \Ag} P_{a,n}(X) \del X^{a,n} \in \Omega_\Og$. Note that $\del: \Og \to \Omega_\Og$ then corresponds to the vertex algebra translation operator $T$. 

We also have the obvious injective linear map 
\be \jota: \Dg \into \Mfw[1] \label{jotadef}\ee
which maps the element
\be \sum_{a,n\in \Ag} P^{a,n}(X) D_{a,n} + \sum_{\substack{a,b\in \I,n\in \ZZ}} p^{b,n}_{a} \SS^a_{b,n} \in \Dg\nn\ee
to the state
\be \sum_{a,n\in \Ag} P^{a,n}(\gam[0]) \bet_{a,n}[-1]\vac + \sum_{\substack{a,b\in \I,n\in \ZZ}} p^{b,n}_{a} \SS^a_{b,n}[-1]\vac \in \Mfw[1]\nn\ee
(The first sum is possibly infinite but must obey the condition \cref{Dwcond}, cf. \cref{Mwcond}. By definition of $\Lglog$ the second sum must have only finitely many nonzero summands.)

In this way,  
\begin{align} \Mfw[0] &\cong \Og \nn\\
              \Mfw[1] &\cong  \Omega_\Og \oplus \jota(\Dg).\label{vtdecomp2}\end{align}
Note that while we identify $\Og$ and $\Omega_\Og$ with their images in $\Mfw\df 1$, from now on we keep the explicit notation $\jota$ for the embedding of $\Dg$. 

\subsection{Local Lie algebras}\label{sec: loc lie}
Given any vertex algebra $\va$ one has the Lie algebra $\L(\va)$ of formal modes of states in $\va$.  (See \cite[\S4]{FrenkelBenZvi}.) Namely, 
\be \L(\va) := \bigl.\va \ox \CC((t)) \bigr/ \im(T\ox 1 + \id \ox \del_t). \label{LVdef}\ee 
It is generated by formal modes $A\fm M$, $A\in \va$, $M\in \ZZ$, modulo the relations $(TA)\fm M = -M A\fm{M-1}$.\footnote{Strictly speaking, it is spanned by linear combinations of the form $\sum_{N} f_N A\fm N$, $f_N\in \CC$, $A\in \va$, of these formal modes.} The Lie bracket is given by the same commutator formula \cref{com Am Bn} obeyed by the modes $A\vap M$ (living in $\End(\va)$), i.e. 
\be\bigl[ A\fm M ,B\fm N  \bigr] = \sum_{K \geq 0} \binom M K \big( A\vap K B \big)\fm{M+N-K}.\label{comf}\ee

Since this formula involves only the non-negative products (i.e., only the singular terms of the OPE), we get a Lie subalgebra $\L(\vla)$ associated to any subspace $\vla \subset \va$ closed under translation $T$ and all the non-negative products. (Such a subspace $\vla$ is called a \emph{vertex Lie subalgebra} of $\va$ \cite{Primc,DLM}.) 
In fact, we don't need to insist on closure under translation: let $\vla$ be any subspace closed under the non-negative products; then $\sum_{n=0}^\8 T^n \vla\subset \va$ is also closed under all the non-negative products\footnote{Indeed, we have $(TA)\vap n B = -n A\vap{n-1} B$ (note the right-hand side is zero when $n=0$) and $A\vap n TB = T(A\vap n B) - (TA)\vap n B$.}, and under translation.

Thus, given any vector subspace $\vla\subset \va$, let
\be \L(\vla) := \bigl.\left(\sum_{n=0}^\8 T^n \vla\right)\ox \CC((t)) \bigr/ \im(T\ox 1 + \id \ox \del_t);\nn\ee
then \cref{comf} defines the structure of a Lie algebra on $\L(\vla)$ whenever $\vla$ closes under all the non-negative products.

Such Lie algebras $\L(\vla)$ are called \emph{local} Lie algebras. 

For any such Lie algebra $\L(\vla)$ we have (as one sees from the commutator formula) subalgebras $\L_{\geq 0}(\vla)$ and $\L_0(\vla)$ consisting of, respectively, the non-negative and zero modes of states in $\vla$. 

\subsection{The extension of $L\Dg$ by $\L(\Ogp)$}
In our case we have the Lie algebra $\L(\Mfw)$ associated to the vertex algebra $\Mfw$. The following is clear on inspection. 
\begin{lem}\label{Midlem} $ $
\begin{enumerate}[(i)]
\item The subspace $\Mfw\df1$ is closed under all the non-negative products. 
\item The subspace $\Ogp$ is closed under all the non-negative products and moreover these non-negative products all vanish. 
\item $\Ogp\subset \Mfw\df 1$ is an ideal, in the sense that $A\vap M B \in \Ogp$ for all $A\in \Ogp$, $B\in \Mfw\df 1$ and $M\geq 0$.
\end{enumerate}
\qed\end{lem}
As a result, 
we have the short exact sequence of Lie algebras
\be 0 \to \L(\Ogp) \to \L(\Mfw\df1) \to \L(\Mfw\df1)\big/\L(\Ogp) \to 0 \nn\ee
where $\L(\Ogp)$ is commutative. 

\begin{rem}
We can describe the commutative Lie algebra $\L(\Ogp)$ more explicitly: since $\del=T: \Og \to \Omega_\Og$ is injective and has kernel $\CC\vac$, we have that
\begin{align} 
\L(\Ogp) &\cong \L(\Omega_\Og) \oplus \CC\bm 1
\cong L\Omega_\Og \oplus \CC\bm 1
\nn\end{align}
where $L\Omega_\Og := \Omega_\Og \ox \CC((t))$ is the loop algebra of $\Omega_\Og$, the latter regarded as a commutative Lie algebra, and where
$\bm 1 := \vac\fm{-1} $
is the only nonzero mode of the state $\vac$. 
\end{rem}

Now let $L\Dg := \Dg \ox \CC((t))$ 
denote the loop algebra of $\Dg$.

\begin{lem}\label{LDextlem} 
There is an isomorphism
\be L\Dg \cong \L(\Mfw\df1)\big/ \L(\Ogp)  \nn\ee 
of Lie algebras, so we have the exact sequence of Lie algebras
\be 0 \to \L(\Ogp)  \to \L(\Mfw\df1) \to L\Dg \to 0.\label{LDext}\ee 
\end{lem}
\begin{proof}
Checking the definitions, one first sees that at the level of vector spaces we have
\be \L(\Mfw\df 1) \cong_\CC \L(\Ogp) \oplus \L(\jota(\D(g))) \nn\ee
and 
\be \L(\jota(\D(g))) \cong_\CC L\Dg.\nn\ee
The fact that $L\Dg \cong \L(\Mfw\df1)\big/\L(\Ogp)$ is also an isomorphism of Lie algebras is a consequence of the following observation.
\end{proof}
\begin{lem}
\label{jotalem}
For any $X,Y\in \Dg$, we have
\be \jota(X) \vap 0 \jota(Y) \equiv \jota([X,Y]) \mod \Omega_\Og ,\qquad 
    \jota(X) \vap 0 \jota(Y) \equiv 0 \mod \Og. \nn\ee
\end{lem}
\begin{proof}
Let us compute the OPE of two states in $\jota(\Dg)$. 
Throughout this proof we use summation convention over repeated pairs of indices not only in $\I$ but also in $\ZZ$. Let
\be A = P^{a,n}(\gam[0]) \bet_{a,n}[-1]\vac + p^{b,n}_a \SS^a_{b,n}[-1] \vac ,\nn\ee
\be B = Q^{a,n}(\gam[0]) \bet_{a,n}[-1]\vac + q^{b,n}_a \SS^a_{b,n}[-1] \vac .\nn\ee
Then we find
\begin{align}
Y(A,z) Y(B,w) &= \nord{Y(A,z)Y(B,z)} \label{ABope} \\
& + \frac{1}{z-w} 
\nord{Y(P^{a,n}(\gam) \vac,z) 
      Y\left(\frac{\del Q^{b,m}}{\del \gam^{a,n}[0]} \bet_{b,m}[-1]\vac,w\right)}\nn\\
& + \frac{1}{z-w} 
p^{b,n}_a \nord{ 
      Y\left(\frac{\del Q^{c,m}}{\del \gam^{b,p}[0]} \gam^{a,p-n}[0]\bet_{c,m}[-1]\vac,w\right)}\nn\\
& + \frac{1}{z-w} 
p^{b,n}_a \nord{ 
      Y\left(Q^{a,m}(\gam) \bet_{b,m+n}[-1]\vac,w\right)}\nn\\
& - \frac{1}{z-w} 
\nord{Y\left(\frac{\del P^{a,n}}{\del \gam^{b,m}[0]} \bet_{a,n}[-1]\vac,z\right) 
      Y\left(Q^{b,m}(\gam) \vac,w\right)} \nn\\
& - \frac{1}{z-w} 
q^{b,n}_a \nord{ 
      Y\left(\frac{\del P^{c,m}}{\del \gam^{b,p}[0]} \gam^{a,p-n}[0]\bet_{c,m}[-1]\vac,z\right)}\nn\\
& - \frac{1}{z-w} 
q^{b,n}_a \nord{ 
      Y\left(P^{a,m}(\gam) \bet_{b,m+n}[-1]\vac,z\right)}\nn\\
& + \frac1{z-w} p^{b,n}_a q^{d,m}_b \left(\delta^c_b \SS^a_{d,n+m}(w)- \delta^a_d \SS^c_{b,n+m}(w)\right)\nn\\
& -  \frac {1}{(z-w)^{2}}  
 \nord{ 
      Y\left(\frac{\del P^{a,n}}{\del \gam^{b,m}[0]} \vac,z\right)
      Y\left(\frac{\del Q^{b,m}}{\del \gam^{a,n}[0]} \vac,w\right)
}. \nn
\end{align}
All but the final line are modes of states in $\jota(\Dg)$ and involve only single contractions, and we recognise the terms as correctly reproducing the Lie bracket in $\Dg$. 
\end{proof}

The final line in \cref{ABope} involves a double contraction, and defines the extension of $L\Dg$ by $\L(\Ogp)$ in \cref{LDextlem}. Namely, we see that
\be [A\fm K, B\fm L ] = [A,B]\fm{K+L} + \omega(A\fm K,B\fm L),\nn\ee
with
\begin{subequations}\label{omegadef}
\begin{multline} \omega(A\fm K,B\fm L) \\:=  -K  \left(\frac{\del P^{a,n}}{\del \gam^{b,m}[0]} 
                 \frac{\del Q^{b,m}}{\del \gam^{a,n}[0]} \vac\right)\fm{K+L-1}
 - \left(\left[T,\frac{\del P^{a,n}}{\del \gam^{b,m}[0]}\right] 
                 \frac{\del Q^{b,m}}{\del \gam^{a,n}[0]} \vac\right)\fm{K+L}
\end{multline}
where we continue to employ summation convention over repeated indices.
Let us stress that these implicit sums on $b,m$ and $a,n$ have only finitely many nonzero terms, by virtue of our definition of $\Mfw$; cf. \cref{Mwcond}. 

The formula above defines $\omega$ on the subalgebra $L\!\left(\Lglog\right)\subset L\Dg$. When one of the arguments to $\omega$ is a mode $\Cocent\fm K$ of $\Cocent$, there are obviously no double contractions, so 
\be \omega(\Cocent\fm K,\cdot) = 0.\ee
\end{subequations}

This map $\omega$ defines a cocycle 
\be [\omega] \in \H^2( L\Dg, \L(\Ogp)) \nn\ee 
in the usual Chevalley-Eilenberg cohomology of $L\Dg$ with coefficients in $\L(\Ogp)$. 
(It is easy to see that this cocycle is non-trivial, as in the case of $\g$ of finite type in \cite[\S5.5.3]{Fre07}, so the sequence in \cref{LDextlem} does not split.)

\subsection{Main result}\label{sec: mr} 
The homomorphism $\vfd: \g \to \Dg$ gives rise to a homomorphism of the corresponding loop algebras
\be \wt\vfd : L\g \to L\Dg,\label{vfdhatdef}\ee
given by
\be{}\qquad\qquad\qquad \wt\vfd( A[N] ) :=  \vfd(A)[N], \qquad A\in \g, N\in \ZZ .\nn\ee
Pulling back the cocycle $\omega \in \H^2( L\Dg, \L(\Ogp))$ from \cref{omegadef} by this homomorphism $\wt\vfd$, we obtain a cocycle
\be \wt\vfd^*(\omega) \in  \H^2( L\g, \L(\Ogp)) \label{pbo}\ee
of $L\g$ with coefficients in $\L(\Ogp)$. 
It defines an extension $\wt{L\g}$ of $L\g$ by $\L(\Ogp)$,
\be 0 \to \L(\Ogp) \to \wt{L\g} \to L\g \to 0.\nn\ee

We can now state the main result of this paper. 
\begin{thm}\label{mainthm}
The cocycle $\wt\vfd^*(\omega)$ is trivial, so this sequence splits. 

Specifically, there is a linear map
\be \vphi:\g \to \Omega_\Og\nn\ee
of $Q$-grade 0 and obeying
\be \vphi(\h) = 0,\quad \vphi(\nmp) \subset \Omega_\Onpm \quad\text{and}\quad \vphi \circ \sigma = \tau \circ \vphi,\nn\ee
such that the map
$\jota\circ\vfd + \vphi : \g \to \jota(\D) \oplus \Omega_\Og \cong \Mfw[1]$
gives rise to a homomorphism of Lie algebras,
\begin{align} L\g &\to \L(\Mfw\df 1);\nn\\
{}\qquad\qquad\qquad A\fm N &\mapsto \bigl((\jota \circ\vfd + \vphi)(A)\bigr)\fm N, \qquad A\in \g,N\in \ZZ. \nn
\end{align}
\end{thm}
\begin{proof}
The proof occupies \cref{sec: proof} below.
\end{proof}

To state this result in a more concrete form, we note the map $\vphi$ is given by 
\be J_{a,n} \mapsto \sum_{(b,m)\in \Ag} Q_{a,n;b,m}(X) \del X^{b,m} \nn\ee
(and $\cent \mapsto 0;\cocent\mapsto 0$) for certain polynomials $Q_{a,n;b,m}(X) \in \Og$.
We introduce also the generating series of the generators of $L\g$, 
\be \cent(z) := \sum_N\cent\fm N z^{-N-1},\quad \cocent(z) := \sum_N\cocent\fm N z^{-N-1},
\quad J_{a,n}(z) := \sum_N J_{a,n}\fm N z^{-N-1}, \nn\ee
and let $\jota(\J_{a,n})(z) = \sum_{N\in \ZZ} \jota(\J_{a,n})[N] z^{-N-1}$ with $\J_{a,n}$ as in \cref{Jdef}. 
Recall the polynomials $R_{a,n}^{b,m}(X)\in \Og$ from \cref{Rdef}.

\begin{thm}\label{homcor}
There is a homomorphism of $\ZZ\times Q$-graded Lie algebras  
\be L\g \to \L(\Mfw\df 1)\nn\ee 
given by
\begin{align} 
\cent(z) &\mapsto 0\nn\\
\cocent(z) & \mapsto \jota(\Cocent)(z)\nn\\
J_{a,n}(z) &\mapsto \jota(\J_{a,n})(z) +  \sum_{(b,m)\in \Ag} Y(R_{a,n}^{b,m}(\gam[0]) \bet_{b,m}[-1] \vac,z) \nn\\ 
          & \qquad    \qquad  \qquad   + \sum_{(b,m)\in \Ag} Y(Q_{a,n; b,m}(\gam[0]) \gam^{b,m}[-1]\vac,z) .\nn
\end{align}
Equivalently, there is a homomorphism of $\ZZ_{\geq 0}\times Q$-graded vertex algebras 
\be \theta: \VV_0^{\g,0} \to \Mfw \nn\ee
given by
\begin{align} 
\cent\fm{-1}\vac &\mapsto 0\nn\\
\cocent\fm{-1}\vac & \mapsto \jota(\Cocent)\fm{-1}\vac\nn\\
J_{a,n}\fm{-1}\vac &\mapsto \jota(\J_{a,n})\fm{-1}\vac +  \sum_{(b,m)\in \Ag} R_{a,n}^{b,m}(\gam[0]) \bet_{b,m}[-1]\vac  \nn\\ 
          & \qquad    \qquad  \qquad   + \sum_{(b,m)\in \Ag} Q_{a,n; b,m}(\gam[0]) \gam^{b,m}[-1]\vac .\nn
\end{align}
\end{thm}
\begin{proof}
The first part is a restatement of \cref{mainthm} in more concrete terms. 
For the equivalence of the statement about vertex algebras, see \cite[\S6.1]{Fre07}.
\end{proof}

Let us give the form of the homomorphism on the Chevalley-Serre generators of $\g$. 
We write $\gam^{e_i} = \gam^{\alpha_i,0}$ for $i\in \oc I$ and $\gam^{e_0} = \gam^{-\delta+\alpha_0,1}$, cf. \cref{highestrootdef}, and so on.

\begin{prop}\label{CScor}
The homomorphism $\theta: \VV_0^{\g,0} \to \Mfw$ from \cref{homcor} is given by
\begin{align}
e_i\fm{-1}\vac &\mapsto \jota(\vfd(e_i))  
                          + c_i \gam^{f_i}\fm{-1}\vac \nn\\
h\fm{-1}\vac &\mapsto \jota(\vfd(h)) \nn\\
f_i\fm{-1}\vac & \mapsto \jota(\vfd(f_i)) + c_i \gam^{e_i}\fm{-1}\vac \nn  
\end{align}
for $i\in I$ and $h\in \h$, where 
\begin{align}
c_i &:=  -2  + \sum_{\substack{j\in I\\ j\prec i}} a_{ij}, \qquad i\in I.\nn
\end{align}
\end{prop}
\begin{proof} On $\ZZ_{\geq 1} \times Q$-grading grounds, the homomorphism must be of this form for some values of coefficients $c_i$. We compute these values in \cref{sec: cis}.
\end{proof}

\subsection{Action of $L_+\g$ on $\Mh(\np)$}\label{sec: Lpgact}
The vertex algebra $\Mh= \Mh(\g)$ was the vacuum Fock module of the $\bet\gam$-system associated to the vector space $\g$. In exactly the same way, one defines $\Mh(\npm)$, with
\be \Mh(\npm) \cong_\CC \CC\!\left[ \gam^{\ia a n}[N]\right]_{N \leq 0; \ii a n \in \Apm} \ox \CC\!\left[\bet_{\ia a n}[N-1] \right]_{N \leq 0; \ii a n \in \Apm} \nn\ee
as vector spaces. 
These $\Mh(\npm)$ are vertex subalgebras of $\Mh$. 

We have the subalgebra $\L_{\geq 0}(\Mfw(\df 1)$ of non-negative modes of states in $\Mfw\df 1$, as in \cref{sec: loc lie}. 
For all $A\in \Mfw\df 1$ and $B\in \Mh$, $A\vap N B\in \Mh$ for all $N\geq 0$. 
Therefore $\L_{\geq 0}(\Mfw\df1)$ acts on $\Mh$, via
\be A\fm N v := A\vap N v\nn\ee 
for $A\in \Mfw\df1$, $N\geq 0$ and $v\in \Mh$.
The Lie algebra homomorphism $L\g \to \L(\Mfw\df 1)$ of \cref{mainthm} restricts to a homomorphism $L_+\g \to \L_{\geq 0}(\Mfw(\df 1)$ from the half-loop algebra 
\be L_+\g := \g \ox \CC[[t]]. \nn\ee 
In this way, $L_+\g$ acts on $\Mh$. The following is analogous to \cref{stabprop}. 
\begin{prop}\label{Lpgstab} This action of $L_+\g$ stabilizes $\Mh(\np)$ and $\Mh(\nm)$. 
That is, for all $A\in \g$, $N\geq 0$, and $v\in \Mh(\npm)$,
\be \theta(A\fm{-1}\vac)\vap N\, v \in \Mh(\npm).\nn\ee
\end{prop}
\begin{proof} Let $N\geq 0$. Suppose without loss that $A=J_{a,n}$. (The result is clear for $A=\cocent$ and trivial for $A=\cent$.)

The conditions on $\vphi$ given in \cref{mainthm} imply that $\vphi(J_{a,n})\vap N$ stabilizes the subspace $M(\npm)$, and annihilates $M(\nmp)$, whenever $J_{a,n} \in \nmp$. And $\vphi(\h) = 0$. Hence $\vphi(J_{a,n})\vap N$ stabilizes both $M(\npm)$.

The same applies to all terms in $\jota(\vfd(J_{a,n}))$ cubic or higher in the generators $\bet$, $\gam$. (Recall that these terms are either in $\jota(\Dw)$ or in $\jota(\Dwm)$.)

The term $\jota(\J_{a,n})\vap N$ in  $\jota(\vfd(J_{a,n}))\vap N$ stabilizes $\Mh(\npm)$ if $a\in \oc I$ and $n=0$, i.e. if $J_{a,n}\in \h$. Otherwise it does not, but by our construction there is then also a sum of compensating quadratic terms in $\jota(\vfd(J_{a,n}))\vap N$, of the form $\gam^{b,m}\vap M\bet_{c,p}\vap{N-M}$ with $\{(b,m),(c,p)\} \not\subset \pm\A$. The latter condition ensures the double contractions between such quadratic terms and states in $\Mh(\npm)$ vanish. The single contractions of such compensating quadratic terms ensure that $\Mh(\npm)$ is stabilized, just as in \cref{stabprop} and the discussion following. 
\end{proof}

\subsection{Homomorphism to $\Mfw \ox \pi_0$}\label{sec: pi0}
We have the loop algebra $L\h := \h \ox \CC((t))$ of the Cartan subalgebra $\h\subset \g$. 
Let $\{b_i\}_{i=1,\dots,\dim\h} = \{H_i\}_{i\in \oc I} \sqcup\{\cent,\cocent\}$ denote a copy of our basis of $\h$
and let $\{b^i\}_{i=1}^{\dim\h}\subset \h$ be its dual basis with respect to the form $\bilin\cdot\cdot$:
\be \bilin{b_i}{b^j} = \delta_i^j .\nn\ee
Then $L\h$ has basis $\{b_i\fm N\}_{i=1,dots,\dim\h;N\in \ZZ}$. 
Let $\pi_0$ denote the $L\h$-module 
\be \pi_0 := U(L\h) \ox_{U(\h\ox \CC[[t]])} \CC \vac \nn\ee
induced from the trivial one-dimensional $\h\ox \CC[[t]]$-module $\CC \vac$.
There is a linear isomorphism $\pi_0 \cong \CC[b_{i,n}]_{i=1,\dots,\dim\h; n<0}$ of 
of vector spaces, and of modules over $U(t^{-1}\h[t^{-1}])\cong \CC[b_{i,n}]_{i=1,\dots,\dim\h; n<0}$.

The space $\pi_0$ is a commutative vertex algebra, the state-field map being given by
\be b_i(x) := Y(b_i\fm{-1}\vac,x) =  \sum_{N\in \ZZ} b_i[N] x^{-N-1},\nn\ee
cf. \cref{sec: vas}. (Like $\Mh$, it is a system of free fields.)
It has the depth gradation, in which $b_i\fm N$ contributes grade $-N$, and it inherits the $Q$-gradation from $\h$.

We continue to write $\gam^{e_i} = \gam^{\alpha_i,0}$ for $i\in \oc I$ and $\gam^{e_0} = \gam^{-\delta+\alpha_0,1}$, etc.

We have the tensor product of vertex algebras $\Mfw \ox \pi_0$, which is again $\ZZ_{\geq 0}\times Q$-graded. 

\begin{thm}\label{bthm}
There exists a $\ZZ_{\geq 0} \ox Q$-graded homomorphism of vertex algebras
\be w : \VV_0^{\g,0} \to \Mfw \ox \pi_0 \nn\ee
given by
\begin{align}
e_i\fm{-1}\vac &\mapsto \jota(\vfd(e_i)) 
                          + c_i \gam^{f_i}\fm{-1}\vac + \la b^j,\check\alpha_i\ra\gam^{f_i}\fm0  b_j\fm{-1}\vac\nn\\
h\fm{-1}\vac &\mapsto \jota(\vfd(h)) + \la b^i,h\ra b_i\fm{-1}\vac\nn\\
f_i\fm{-1}\vac & \mapsto \jota(\vfd(f_i)) + c_i \gam^{e_i}\fm{-1}\vac
                      + \la b^j,\check\alpha_i\ra \gam^{e_i}\fm0  b_j\fm{-1}\vac   \nn  
\end{align}
for $i\in I$ and $h\in \h$. 
\end{thm}
\begin{proof}
The proof is given in \cref{sec: proofb}.
\end{proof}
Equivalently, there is a $\ZZ\times Q$-graded homomorphism of Lie algebras
\be L\g \to \L(\Mfw \ox \pi_0) \nn\ee
given by
\begin{align}
e_i(z) &\mapsto \jota(\vfd(e_i))(z) 
                 + c_i \del_z\gam^{f_i}(z) + \la b^j,\check\alpha_i\ra \gam^{f_i}(z)  b_j(z)\nn\\
h(z) &\mapsto \jota(\vfd(h))(z) + \la b^i,h\ra b_i(z) \nn\\
f_i(z) & \mapsto \jota(\vfd(f_i))(z) + c_i \del_z\gam^{e_i}(z)
                      + \la b^j,\check\alpha_i\ra \gam^{e_i}(z) b_j(z) \nn  
\end{align}
for $i\in I$ and $h\in \h$. 

\subsection{On zero modes}\label{sec: zeromodes}
We have the homomorphism of Lie algebras $\vf:\g \to \Derc\Onp$
from \cref{vfdef}, and the embedding $\imb:\Derc\Onp\into \wt\Mh(\np) \subset \wt\Mh$
given by 
\be \sum_{\ii an \in \A} P^{a,n}(X) D_{a,n} \mapsto \sum_{\ii an\in \A} P^{a,n}(\gam\fm0) \bet_{a,n}\fm{-1}\vac, \nn\ee
where $P^{a,n}(X) \in \Onp$ for each $(a,n) \in \A$. 
Let $\phi:\g \to \Omega_{\Onp}$ be the linear map defined by 
\be \phi(x) = \begin{cases} 0 & x\in \b_+\\
                           \vphi(x) & x\in \nm 
                         \end{cases}
\nn\ee
where $\vphi: \g \to \Omega_\Og$ was the splitting map from \cref{mainthm}. We continue to identify $\Omega_\Onp$ with a subspace of $\Mh(\np)\subset\wt\Mh(\np)$ (with $f(X) dX^{a,n} \mapsto f(\gam\fm0) \gam^{a,n}\fm{-1}\vac$).
We get a linear map
\be \nff = \imb\circ\vf + \phi : \g \cong \VV_0^{\g,0}[1] \to \wt\Mh(\np)[1].\nn\ee
(The kernel is the centre $\CC\cent\subset\g$; if one wanted an embedding, one could introduce a tensor factor $\pi_0$ as in \cref{sec: pi0}.)

The would-be non-negative vertex algebra products $\nff(x)\vap N \nff(y)$, $x,y\in \g$, $N\geq 0$, of states in the image of this map are generically ill-defined, i.e. divergent, sums (as we saw in the examples in \cref{sec: zeta}). Thus, we don't get a direct analog in affine types of usual Feigin-Frenkel free-field realization $\VV_0^{\g,-h^\vee} \to \Mh(\np)$ in finite types. Nonetheless, we do have the following. 

\begin{thm}\label{zeromodethm}
The vertex algebra $0$th product $\vap 0 :\Mh(\np)\times \Mh(\np) \to \Mh(\np)$ extends to a well-defined product 
\be \vap 0 : \nff(\g) \times \nff(\g) \to \nff(\g)\nn\ee 
on the image $\nff(\g)\subset \wt\Mh(\np)[1]$ of $\g$ in $\wt\Mh(\np)$. 

Moreover, we have
\be \nff(x) \vap 0 \nff(y) = \nff([x,y]) \nn\ee
for all $x,y\in \g$. 
\end{thm}
\begin{proof} Let $x,y\in \g$. 
We first want to show that $\nff(x) \vap 0 \nff(y)$ is well-defined. 

Calculating as in the proof of \cref{jotalem}, one sees that for any two $X,Y\in \Der\Onp$, we have 
\be \imb(X)\vap 0 \imb(Y) = \imb([X,Y]) + \dd X Y \label{zeroprod}\ee
where we define a ($\CC$-)bilinear map $\dd\cdot\cdot:\Der\Onp\times \Der\Onp \to \Omega_{\Onp}$ by 
\begin{align}
    \dd{ \sum_{\ii a n \in \A} P^{\ia a n}(X) D_{\ia a n} }  
       { \sum_{\ii b m \in \A} Q^{\ia b m}(X) D_{\ia b m} } \nn\\
 := - \sum_{\ii a n,\ii b m \in \A} \bigl( \del D_{\ia b m} P^{\ia a n}(X)\bigr) 
                                   \bigl(D_{\ia a n} Q^{\ia b m}(X)\bigr).
\end{align}
Now, $\dd\cdot\cdot$ does not extend to a well-defined bilinear form on $\Derc\Onp\times\Derc\Onp$: for example if $v=  
\sum_{n\in \NN} X^{\ia a n} D_{\ia a n}$ for some fixed $a\in \I$, then 
$- \dd {X^{\ia c 1} v} v  = 
 \Bigl( 1 + \sum_{m\in \NN} 1\Bigr)  \del X^{c,1}
$
is ill-defined. But we can define a Lie subalgebra large enough to contain the image of $\g$ and small enough that $\dd\cdot\cdot$ remains well-defined, as follows. First let us define a Lie subalgebra $\Dl\subset \Derc\Onp$, by stipulating that   
$\sum_{\ii a n \in \A} P^{\ia a n} (X) D_{\ia a n}$ belongs to $\Dl$ if, for some $K\in \NN$,  
\be P^{\ia a n}(X) \in \bigoplus_{\substack{\ii b m \in \A \\ n+K \geq m \geq n- K}} \CC X^{\ia b m} \nn\ee
for each $\ii a n \in \A$. One checks that $\Dw\rtimes \Dl$ is a Lie subalgebra of $\Derc\Onp$. According to \cref{linthm}, the image $\vf(x)$ of $x$ belongs to this subalgebra:
\be \vf(x) \in \Dw \rtimes \Dl .\nn\ee 

\begin{lem}\label{ddprop} 
$\dd\cdot\cdot$ extends to a well-defined $\CC$-bilinear map
\be \dd\cdot\cdot :\left(\Dw \rtimes \Dl\right) \times \left(\Dw \rtimes \Dl\right) \to \Omega_\Onp,\nn\ee 
with the property that $ \dd\Dl\cdot  = 0$. 
\end{lem}
\begin{proof}[Proof of \cref{ddprop}]
$\dd\cdot\cdot$ vanishes, summand by summand, whenever its first argument is linear in the generators $\{X^{\ia a n}\}_{\ii a n\in \A}$, since $\del (D_{\ia a n} X^{\ia b m}) =0$ for all $\ii a n, \ii b m \in \A$. 
So by linearity it is enough to consider the case when the first argument, $ \sum_{\ii a n \in \A} P^{\ia a n}(X) D_{\ia a n} $ belongs to $\Dw$, i.e. has widening gap. 
The second argument has bounded grade, so there is some $M$ such that $D_{\ia a n} Q^{\ia b m}(X) =0$ for all $n,m$ such that $n-m \geq M$. 
By the assumption of widening gap, there is some $N$ (depending on $M$) such that, for all $n>N$, $D_{\ia b m}P^{\ia a n}(x)=0$ whenever $n-m < M$.
So at most the first $N$ terms in the sum on $n$ can be nonzero:
\begin{align}  \sum_{\ii a n,\ii b m \in \A} \bigl(\del D_{\ia b m} P^{\ia a n}(X)\bigr) 
                                  \bigl(D_{\ia a n} Q^{\ia b m}(X)\bigr) \nn\\
= \sum_{\substack{\ii a n,\ii b m \in \A \\ n\leq N}} \bigl(\del D_{\ia b m} P^{\ia a n}(X)\bigr) 
                                  \bigl(D_{\ia a n} Q^{\ia b m}(X)\bigr),\nn
\end{align}
and then for each $n$ the sum on $m$ is also finite, since $P^{\ia a n}(X)$ is a polynomial. 
\end{proof}

At this stage, we have shown that \cref{zeroprod} holds for all $X,Y\in \Dw \rtimes \Dl$. In particular, it holds for the images $\vf(x),\vf(y)$ of $x,y\in\g$:
\be  \imb(\vf(x))\vap 0\imb(\vf(y)) = \imb(\vf([x,y])) + \dd{\vf(x)}{\vf(y)}. \nn\ee
It follows that $\nff(x) \vap 0 \nff(y)$ is well-defined. (There are no possible double contractions, and hence no possible divergences, of the products between the subspaces $\Omega_{\Onp}$ and $\imb(\Derc\Onp)$ of $\wt\Mh(\np)$.) 

The ``moreover'' part is then essentially a corollary of \cref{mainthm}. Consider the linear map $\pi:\Mfw[1] \to \wt\Mh[1]$ which acts as the identity on $\Mw[1]\subset \Mfw[1]$ and sends $\SS^a_{b,n}\fm{-1}\vac \to \sum_{m\in \ZZ} \gam^{a,m}\fm 0 \bet_{b,m+n}\fm{-1}\vac$ and $\Cocent \to \sum_{(a,m)\in \Ag} m\gam^{a,m}\fm 0 \bet_{a,m}\fm{-1}\vac$. By construction, for any $x\in \g$, $\pi(\theta(x\fm{-1}\vac)) = \nff(x) + \tau(\nff(\cai(x)))$, where $\nff(x) \in \wt\Mh(\np)[1]$ and $\tau(\nff(\cai(x))) \in \wt\Mh(\nm)[1]$. (See \cref{Jdif} and the discussion following.) Thus
\be \pi(\theta(x\fm{-1}\vac))\vap 0 \pi(\theta(y\fm{-1}\vac)) = \nff(x)\vap 0 \nff(y) + \tau(\nff(\cai(x)))\vap0\tau(\nff(\cai(y)))\label{pit1}\ee
for all $x,y\in \g$. At the same time, for any $X,Y\in \Mfw[1]$, we check that $\pi(X)\vap 0 \pi(Y)$ is well-defined and equal to $\pi(X\vap 0 Y)$. Thus
\begin{align} \pi\bigl(\theta(x\fm{-1}\vac)\bigr)\vap 0 \pi\bigl(\theta(y\fm{-1}\vac)\bigr) 
&= \pi\bigl(\theta(x\fm{-1}\vac)\vap 0 \theta(y\fm{-1}\vac)\bigr)\nn\\
&= \pi(\theta([x,y]\fm{-1}\vac))\nn\\
&= \nff([x,y]) + \tau(\nff(\cai([x,y])))\label{pit2}.
\end{align}
On comparing \cref{pit1} and \cref{pit2}, and projecting onto the summand $\wt\Mh(\np)[1]$ of $\wt\Mh(\np)[1]\oplus\wt\Mh(\nm)[1]$, we have the result.
\end{proof}
\begin{cor} There is a well-defined Lie algebra $\L_0(\nff(\g))$ of the formal 0-modes of states in the image of $\nff$, and the map 
\be \g\to \L_0(\nff(\g));\qquad x \mapsto \nff(x)\fm 0 \nn\ee
is a homomorphism of Lie algebras.
\qed\end{cor}
\begin{proof}
This follows from the theorem since, when the Lie bracket formula \cref{comf} for formal modes is specialized to the Lie bracket of formal \emph{zero} modes, only the $0$th product contributes: $[A\fm 0,B\fm 0] = (A\vap 0 B)\fm 0$. 
\end{proof}

\section{Proof of \cref{mainthm}}\label{sec: proof}
In \cref{pbo} we obtained a cocycle $\wt\vfd^*(\omega) \in  \H^2( L\g, \L(\Ogp))$.
In fact, we can be rather more precise: recall $\vfd(\cent) =0$, $\vfd(\cocent) = \Cocent$, and $\omega(\Cocent\fm N,\cdot) =0$. So we have actually defined a $\L(\Ogp)$-valued $2$-cocycle on the double-loop algebra 
\be \LLog := L(\Log) := \oc \g[t,t^{-1}] \ox\CC((s)).\nn\ee 
Our goal is to show that this cocycle is trivial. 

We shall follow closely the strategy of proof due to Feigin and Frenkel \cite{FF1990}, and specifically the treatment in  \cite[\S5.6]{Fre07}, \cite{Frenkel_2005}; cf. also \cite{FeiginBRST,FFGD}. 
The subspace $\Ogp\subset \Mfw$ is contained in the larger subspace 
\be \Mh_0 := \CC[\gam^{a,n}[-N]]_{a\in \I, n\in \ZZ, N\in \ZZ_{\geq 0}} \vac \subset \Mfw .\nn\ee
This subspace is a commutative vertex algebra. It is also an ideal for the action of $\Mfw\df1$ in the same sense as in \cref{Midlem}: $A\vap M B \in \Mh_0$ for all $A\in \Mh_0$, $B\in \Mfw\df 1$ and $M\geq 0$. It follows that
\be \wt\vfd^*(\omega) \in  \H^2( \LLog, \L(\Mh_0)) ,\nn\ee
and it is convenient to show our cocycle is zero in the latter space. 

To do so, we first show that the cocycle $\wt\vfd^*(\omega)$ actually belongs to the \emph{local} subcomplex of this CE complex. Let us define this local complex. 

\subsection{$\bee\cee$-system}
Let $\Cl$ denote the Clifford algebra with generators 
\be \bee_{a,n}[N],\quad \cee^{a,n}[N], \nn\ee
with $N\in \ZZ$ and $\ii a n \in \Ag$, and anticommutation relations 
\be [\bee_{\ia a n}[N],\bee_{\ia b m}[M]]_+ = 0,\quad 
[\bee_{\ia a n}[N],\cee^{\ia b m}[M]]_+ = \delta_{\ia a n}^{\ia b m} \delta_{N,-M} 1,\quad 
[\cee^{\ia a n}[N],\cee^{\ia b m}[M]]_+ = 0,\nn\ee
where we write $[X,Y]_+ := XY + YX$ for the anticommutator. 

Define $\Wh$ to be the induced $\Cl$-module generated by a vector $\vac$ such that
\be \bee_{\ia a n}[M]\vac = 0,\quad M\in \ZZ_{\geq 0}, \qquad 
    \cee^{\ia a n}[M]\vac = 0,\quad M\in \ZZ_{\geq 1}. \label{Wdef}\ee
for all ${\ii a n}\in \Ag$. It is $\ZZ\times Q$-graded just as is $\Mh$, cf. \cref{sec: vas}. 

$\Cl$ is a superalgebra (with all generators $\bee$, $\cee$ of odd degree) and its module $\Wh$ is a vector superspace.
$\Wh$ is moreover a vertex superalgebra (for the definition of which see e.g. \cite{KacVertexAlgBook,FrenkelBenZvi}). The state-field map $Y(\cdot,x): \Wh \to  \Hom\left(\Wh, \Wh((x))\right)$
is defined as follows. First,  
\begin{alignat}{2}
\bee_{\ia a n}(x) &:= Y\left(\bee_{\ia a n}[-1] \vac,x\right) 
                   &&= \sum_{N\in \ZZ} \bee_{\ia a n}[N] x^{-N-1},\nn\\
\cee^{\ia a n}(x) &:= Y\left(\cee^{\ia a n}[0] \vac,x\right) 
                   &&= \sum_{N\in \ZZ} \cee^{\ia a n}[N] x^{-N} ,\nn
\end{alignat}
and then in general if
\be A = \cee^{\ia{a_1}{n_1}}[-N_1]\dots \cee^{\ia{a_r}{n_r}}[-N_r] 
                \bee_{\ia{b_1}{m_1}}[-M_1]\dots \bee_{\ia{b_s}{m_s}}[-M_s] \vac, \label{wonstate}\ee 
then
\begin{multline} A(u) := Y(A,u) =\nord{\cee^{\ia{a_1}{n_1} (N_1)}(u) \dots \cee^{\ia{a_r}{n_r} (N_r)}(u) 
               \bee_{\ia{b_1}{m_1}}^{(M_1-1)}(u)\dots \bee_{\ia{b_s}{m_s}}^{(M_s-1)}(u)}.
\label{WAu}
\end{multline}
The normal-ordering here is defined just as in \cref{sec: vas} only with the addition of signs to allow for the superspace $\ZZ/2\ZZ$-grading: we have
\be \nord{Y(A,u) Y(B,v)} \,\,\, := \Biggl(\sum_{M<0} A\vap M u^{-M-1} \Biggr)Y(B,v) + (-1)^{p(A)p(B)} Y(B,v) \Biggl(\sum_{M\geq 0} A\vap M u^{-M-1}\Biggr).\label{supernord}\ee  
whenever $A$ and $B$ are in grades $p(A)$ and $p(B)$ respectively. 

We have the \emph{(super)translation operator} $T\in \End\Wh$ of the vertex superalgebra $\Wh$, defined by $TA := A\vap{-2}\vac$. It acts on monomials in the generators of $\Cl$ as a superderivation, according to
\be \left[T, \bee_{a,n}[-N]\right]_+ = -N \bee_{a,n}[N-1] ,\qquad \left[T, \cee^{a,n}[-N]\right]_+ = -(N-1)\cee^{a,n}[N-1], \nn\ee
and we have $T\vac = 0$. 


For each $r\geq 0$, let $\Wh_0^r$ denote the subspace of $\Wh$ spanned by states of the form 
\be \cee^{a_1,n_1}[-N_1] \dots \cee^{a_r,n_r}[-N_r]\vac, \nn\ee
with $a_1,\dots,a_r\in \I$, $n_1,\dots,n_r\in \ZZ$, $N_1,\dots,N_r\in \ZZ_{\geq 0}$.
The sum, $\Wh_0 := \bigoplus_{r=0}^\8 \Wh_0^r$ is a supercommutative vertex superalgebra.

\subsection{Chevalley-Eilenberg cochains for $\LpLog$ with coefficients in $\Mh_0$}\label{sec: CE}
Let $\LpLog$ denote the Lie algebra
\be \LpLog := \Log \ox \CC[[s]] .\nn\ee
It is a topological Lie algebra, with the the linear topology coming from the $\ZZ_{\geq 0}$-grading (i.e. the linear topology in which $\Log \ox s^N\CC[[s]]$, $N\in \ZZ_{\geq 0}$, are a base of the open neighbourhoods of $0$). 

The space $\Mh_0$ is a module over $\LpLog$.
Indeed, let us write $\JJ_{a,n} := \jota(\vfd(J_{a,n})) \in \Mfw$, that is
\be \JJ_{a,n} =   \sum_{(b,m)\in \Ag} R_{a,n}^{b,m}(\gam[0]) \bet_{b,m}[-1]\vac + \sum_{(b,c)\in \I} f_{ba}{}^c \SS^b_{c,n}[-1]\vac,\nn\ee
where we recall the definition \cref{Rdef} of the polynomials $R^{b,m}_{a,n}$ and the definition \cref{jotadef} of the injective map $\jota$. 
From \cref{jotalem} and the fact that  $\vfd:\g \to \Dg$ is a homomorphism of Lie algebras, it follows that
\begin{align}
    \JJ_{a,n} \vap 0 \JJ_{b,m} &\equiv f_{ab}{}^c \JJ_{c,n+m} \mod \Omega_\Og,\nn\\
    \JJ_{a,n} \vap 1 \JJ_{b,m} &\equiv 0  \mod \Og. \label{jopes}
\end{align}
The mode $\JJ_{a,n}\vap N$ restricts to a linear map $\Mh_0\to \Mh_0$ for all non-negative $N$, as in \cref{Midlem}(iii). 
This defines an action of $\LpLog$ because, in view the commutator formula \cref{com Am Bn} and \cref{jopes}, we have
\be [\JJ_{a,n}\vap N, \JJ_{b,m}\vap M] = f_{ab}{}^c \JJ_{c,n+m}\vap{N+M} \mod \L(\Mh_0), \qquad N,M\geq 0,
\nn\ee
(and equality modulo $\L(\Mh_0)$ is enough, cf. \cref{Midlem}(ii)).

As an $\LpLog$-module, $\Mh_0$ is smooth, which is to say that for all $v\in \Mh_0$ and all $a\in \g$, $a\fm N v =0$ for all sufficiently large $N$. In other words the action is continuous, when we endow $\Mh_0$ with the discrete topology. 

We have the $r$th exterior power $\extp^r \LpLog$ of the topological vector space $\LpLog$, with its natural topology.
The space of $r$-cochains of the Chevalley-Eilenberg complex of $\LpLog$ with coefficients in the module $\Mh_0$ is the space $C^r(\LpLog, \Mh_0) := \Homcont(\extp^r\LpLog,\Mh_0)$ of continuous linear maps
\be \lambda: \extp^r\LpLog \to \Mh_0.\nn\ee
Since $\Mh_0$ has the discrete topology, continuity means that for each such $\lambda$ there must be some $N$ such that $\lambda$ kills everything in grades $\geq N$. Thus
\begin{align} \Homcont\left(\extp^r\LpLog,\Mh_0\right) 
&= \bigoplus_{N\geq 0}\Hom\left( \left(\extp^r\LpLog\right)_N, \Mh_0\right).\label{Homs}
\end{align}

\subsubsection{Case of finite-dimensional $\g$}\label{sec: fd}
Let us digress to recall what happens when $\g$ is of finite dimension.
In that case the subspaces $\left(\extp^r \LpLog\right)_N$ appearing in \cref{Homs} are also all finite-dimensional, and so 
\begin{align}  
\bigoplus_{N\geq 0} \Hom\left(\left(\extp^r\LpLog\right)_N, \Mh_0\right) 
  &\cong \Mh_0 \ox \bigoplus_{N\geq 0} \left(\left(\extp^r \LpLog\right)_N\right)^*\nn\\
  &\cong \Mh_0 \ox \Wh_0^r,\nn
\end{align}
where in the second step we identify the restricted dual, $\bigoplus_{N\geq 0} \left(\left(\extp^r \LpLog\right)_N\right)^*$, with the space $\Wh_0^r$ spanned by states of the form $\cee^{a_1}[-N_1]\dots \cee^{a_r}[-N_r]\vac$, by means of the bilinear pairing given by
\begin{multline} \left(\cee^{a_1}[-N_1]\dots \cee^{a_r}[-N_r]\vac , J_{b_1}[K_1] \wedge \dots \wedge J_{b_r}[K_r] \right) \\\mapsto \begin{cases} \sign(\sigma) & \text{$a_i = b_{\sigma(i)}$ and $N_i = K_{\sigma(i)}$ for each $i$, for some $\sigma\in S_r$} \\ 0 & \text{otherwise} 
\end{cases} 
\nn
\end{multline}
(Here we picked a basis $J_a$ of $\g$, where $a$ runs over a finite set of indices.)
 
In this way, one has a linear isomorphism $C^\bl(\LpLog, \Mh_0) \cong \Mh_0\ox \Wh_0 $.

\subsubsection{Case of infinite-dimensional $\g$}\label{sec: tox}
Now we return to our case, in which $\g$ has countably infinite dimension. 
The tensor product $\Mh_0\ox \Wh_0$ now corresponds only to a subspace of $C^\bl(\LpLog,\Mh_0)$, and one which will not be large enough for our purposes. 
It is convenient simply to \emph{define}, for each $r$,
\be \Mh_0 \tox \Wh^r_0 := \bigoplus_{N\geq 0}\Hom\left( \left(\extp^r\LpLog\right)_N, \Mh_0\right) = C^r(\LpLog,\Mh_0) .\nn\ee

An element $\Phi\in \Mh_0\tox \Wh_0^r$ can be regarded as a possibly infinite sum of states of the form 
\be \cee^{a_1,n_1}[-N_1]\dots \cee^{a_r,n_r}[-N_r]v \in \Mh_0 \ox \Wh_0^r,\label{vform}\ee
with $\sum_{i=1}^r N_i\leq N$ for some bound $N$ depending on $\Phi$,\footnote{Note that we have not required these summands to have bounded depth overall, so the depths of the states $v\in \Mh_0$ could increase without bound as, say, $n_1$ increases. We shall only ever need elements of bounded depth, though.} and subject to the condition that for any $\lambda\in \extp^r \LpLog$, the pairing $(\Phi,\lambda)$, defined as follows, yields a well-defined (i.e. finite) element of $\Mh_0$.  
We first define the pairing between a state in $\Mh_0\ox \Wh_0^r$ of the form \cref{vform} and a vector in $\extp^r \LpLog$ of the form
\be J_{b_1,k_1}[K_1] \wedge \dots \wedge J_{b_r,k_r}[K_r] \label{lamform}\ee
by declaring that 
\begin{multline} \left(\cee^{a_1,n_1}[-N_1]\dots \cee^{a_r,n_r}[-N_r]m , J_{b_1,k_1}[K_1] \wedge \dots \wedge J_{b_r,k_r}[K_r] \right) \\\mapsto \begin{cases} \sign(\sigma) m & \text{$a_i = b_{\sigma(i)}$, $n_i= k_{\sigma(i)}$ and $N_i = K_{\sigma(i)}$ for each $i$, for some $\sigma\in S_r$} \\ 0 & \text{otherwise} 
\end{cases} 
\nn
\end{multline}
Then we extend by linearity in both slots.  

\begin{rem}\label{rem: compact}
Neither $\LpLog$ nor $\extp^r \LpLog$ are complete. For example, pick an $a\in \I$ and let $s_N = \sum_{n=1}^N J_{a,n}[n]$ for $N=1,2,3,\dots$. This sequence is Cauchy but not convergent (because the infinite sum $\sum_{n=1}^\8 J_{a,n}[n]$ does not belong to $\LpLog = \g\ox \CC[[t]]$). 

By not completing these spaces, we preserve a useful property of ``compact support'': for any $\lambda \in \extp^r \LpLog$ there exists some $k$ such that $\lambda$ can be written as a possibly infinite linear combination of terms of the form \cref{lamform} \emph{with $|k_i| < k$ for each $i=1,\dots, r$}.
\end{rem}

\subsection{The complex $C^\bl(\LpLog,\Mh_0)$ and the state $\Q$}\label{Qrem} Thus, we have the spaces of the Chevalley-Eilenberg complex $C^\bl(\LpLog,\Mh_0)$ of $\LpLog$ with coefficients in the module $\Mh_0$. 
The differential $d: C^r(\LpLog,\Mh_0) \to C^{r+1}(\LpLog,\Mh_0)$ is given by the usual formula,
\begin{align} (d f)(x_1,\dots,x_{r+1}) := \sum_{p=1}^{r+1} (-1)^{p+1} x_p \on f(x_1,\dots,x_{p-1},x_{p+1},\dots,x_{r+1}) \nn\\
+ \sum_{1\leq p< q \leq r+1} (-1)^{p+q} f([x_p,x_q],x_1,\dots,x_{p-1},x_{p+1},\dots,x_{q-1},x_{q+1},\dots,x_{r+1}) .\label{CEdiff}
\end{align}

Informally, we may introduce a state $\Q$ given by 
\be \Q := \sum_{(a,n) \in \Ag} \cee^{a,n}[0] \JJ_{a,n} 
- \half \sum_{\substack{a,b,c\in \I\\ n,m\in \ZZ}} f_{ab}{}^c \cee^{a,m}[0] \cee^{b,n}[0] \bee_{c,n+m}[-1] \vac. \nn\ee
This is by analogy with the definition of $\Q$ in the case where $\g$ has finite dimension.
In our case, $\Q$ does not belong to the tensor product $\Mfw\ox \Wh_0$, because of the infinite sums on $n,m$. It belongs to a suitable completion. One should not expect the vertex superalgebra structure to extend to that completion, just as $\wt\Mh$ was not a vertex algebra in \cref{sec: wtMh}. Since $\Q$ is the only such state we shall actually need, let us avoid a formal definition. Instead, we shall check that $\Q$ has the specific properties we need, as they arise.

Here is the first such property. The zero mode $\Q\vap 0$ unambiguously defines a linear map $\Mh_0\tox\Wh_0\to \Mh_0\tox\Wh_0$. (It involves only single contractions.) One easily sees that it coincides with the CE differential:
\be d\Phi = \Q \vap 0 \Phi .\nn\ee

\subsection{The local complex $\Cloc^\bl(L_\g,\L(\Mh_0))$}
Recall the definition of the local Lie algebra $\L(\vla)$ from \cref{sec: loc lie}. One sees from \cref{comf} that the formal zero modes generate a Lie subalgebra, which we shall denote by
\be \L_0(\vla)\subset \L(\vla).\label{L0def}\ee 
We can apply this in particular to the subspaces $\Mh_0 \ox \Wh_0^r$ of the supercommutative vertex superalgebra $\Mh_0\ox\Wh_0$, to obtain the vector superspaces (in fact, supercommutative Lie superalgebras) $\L_0(\Mh_0 \ox \Wh_0^r)$. We see that these spaces are given by
\begin{align} 
\L_0(\Mh_0 \ox \Wh_0^r)
&= \bigl.\left(\sum_{k\geq 0} T^k  ( \Mh_0 \ox \Wh_0^r) \right)\bigr/ (\im T + \CC\vac)\nn\\
&= \bigl.\left(  \Mh_0 \ox \Wh_0^r\right) \bigr/ (\im T + \CC\vac).\nn
\end{align}
(Note that we have to quotient by the subspace $\CC\vac= \ker T$ too, not just $\im T$, because the zero mode $\vac\fm 0 = \vac \ox t^0 = \del_t(\vac \ox t) = (T+ \del_t) (\vac \ox t) \equiv 0$ does vanish in $\L(\Mh_0 \ox \Wh_0^r)$.)

Now, when $\g$ has finite dimension, the spaces of the local complex are, by definition, precisely these $\L_0(\Mh_0 \ox \Wh_0^r)$. In our case, in which $\g$ has countably infinite dimension, we must use the larger spaces $\Mh_0 \tox \Wh_0^r$ from \cref{sec: tox}. By analogy with the above, let us define
\be \Cloc^r(\LLog,\L(\Mh_0)) := \bigl.\left(  \Mh_0 \tox \Wh_0^r\right) \bigr/ (\im T + \CC\vac)\label{Clocdef}\ee
for each $r$. Here we use the fact that the definition of the translation operator extends in a well-defined way to $\Mh_0 \tox \Wh_0^r$. 

Let us write $\int$ for the projection map $\int: C^r(\LpLog, \Mh_0) \to \Cloc^r(\LLog,\L(\Mh_0)) ;\quad \Phi \mapsto \int \Phi := \Phi[0]$. 
Let us attempt to define a differential 
\be d:  \Cloc^r(\LLog, \L(\Mh_0)) \to \Cloc^{r+1}(\LLog, \L(\Mh_0))\nn\ee
by setting 
\be d(\Phi[0]) := (d\Phi)[0] = (\Q\vap 0 \Phi)[0] .\nn\ee
The fact that this is a consistent definition is a consequence of the following lemma. 
\begin{lem}\label{lem: Cloccd}
We have the commutative diagram
\be\label{Cloccd}
\begin{tikzcd}
\Cloc^r(\LLog, \L(\Mh_0))\rar{d}&\Cloc^{r+1}(\LLog, \L(\Mh_0))\\
C^r(\LpLog, \Mh_0) \uar{\int}\rar{d} & C^{r+1}(\LpLog, \Mh_0) \uar{\int} \\
C^r(\LpLog, \Mh_0)\uar{T}\rar{d}& C^{r+1}(\LpLog, \Mh_0)\uar{T} .
\end{tikzcd}\ee
\end{lem}
\begin{proof}
What has to be checked is that the lower square is commutative. This is seen by direct calculation: as in \cref{Qrem}, there is a well-defined notion of the zero mode $\Q\vap 0$ of $\Q$ acting on $\Mh_0\tox\Wh_0$, and we check that 
$Td\Phi = T (\Q\vap 0 \Phi) = (T\Q) \vap 0 \Phi + \Q \vap 0 T\Phi = \Q \vap 0 T\Phi= d T\Phi$.
\end{proof}

Thus, we obtain a complex, $(\Cloc^\bl(\LLog,\L(\Mh_0)),d)$. This is the \emph{local complex}. 
As the notation suggests (and as we now check) it forms a subcomplex of the usual Chevalley-Eilenberg complex $C^\bl(\LLog, \L(\Mh_0))$ of $\LLog$ with coefficients in $\L(\Mh_0)$.

Consider first a state $\Psi \in \Mh_0 \ox \Wh_0^r \subset \Mh_0\tox \Wh_0^r$ of the form 
\be \Psi = \cee^{a_1,n_1}[-N_1]\dots \cee^{a_r,n_r}[-N_r] v. \nn\ee
We can apply the state-field map to this state, and in particular we can take the zero mode $\Psi\vap 0\in \End(\Mh\ox \Wh)$. This zero mode $\Psi\vap 0$ is a (generically infinite) sum of terms 
\be \cee^{a_1,n_1}[M_1] \dots \cee^{a_r,n_r}[M_r] v[M] \label{typmod}\ee
(Here $v[M] \in\L(\Mh_0)$ is the $M$th  mode of the state $v\in \Mh_0$.\footnote{Here and in what follows we are equivocating between formal modes $X\fm M\in \L(\Mh\ox\Wh)$ and modes $X\vap M\in \End(\Mh\ox\Wh)$ of states in $X\in \Mh\ox\Wh$. There is no loss in this because the vertex superalgebra $\Mh\ox\Wh$ of free fields has the property that the Lie algebra homomorphism $\L(\Mh\ox\Wh) \to \End\Mh\ox\Wh$ is injective.})

We have the exterior algebra $\extp^r \LLog$.
The pairing from \cref{sec: tox} goes over to this setting, and using it we can certainly interpret each term \cref{typmod} as an $r$-cochain in
\be C^r(\LLog, \L(\Mh_0) := \Homcont(\extp^r \LLog,\L(\Mh_0)).\nn\ee 
Moreover, the infinite sum $\Psi[0]$ is again a well-defined $r$-cochain, i.e. it is continuous, when  $\L(\Mh_0)$ gets its natural (i.e. $t$-adic) linear topology. 

These statements are exactly as in the case in which $\g$ has finite dimension. The only new aspect of the present case is that a general state $\Phi\in \Mh_0 \tox \Wh_0^r$ may itself be an infinite sum of such states $\Psi\in \Mh_0 \ox \Wh_0^r$, subject to the conditions we gave in \cref{sec: tox}. But, for any given $\mu \in \extp^r \LLog$, we can arrange that for only finitely many of these summands $\Psi$ does $\Psi[0]$ have nonzero pairing with $\mu$. (This is clear from the  notion of ``compact support'' from \cref{rem: compact}.) In this way, we can indeed interpret $\Phi[0]\in \Cloc^r(\LLog,\L(\Mh_0))$ as an element of $C^r(\LLog,\L(\Mh_0))$. 

Finally, one checks that the derivative on $\Cloc^\bl(\LLog,\L(\Mh_0))$ coincides with the usual Chevalley-Eilenberg derivative. 

\subsection{The cocycle $\upomega[0]$}
As the relevant example for us, consider the state $\upomega\in \Mh_0\tox \Wh_0^2$ given by 
\begin{align} 
\upomega 
&:=
- \cee^{a,n}[-1] \cee^{b,m}[0] 
                 \frac{\del R_{a,n}^{c,p}}{\del \gam^{d,q}[0]} 
                 \frac{\del R_{b,m}^{d,q}}{\del \gam^{c,p}[0]} \vac
- \cee^{a,n}[0] \cee^{b,m}[0] 
          \left[T,\frac{\del R_{a,n}^{c,p}}{\del \gam^{d,q}[0]}\right] 
                 \frac{\del R_{b,m}^{d,q}}{\del \gam^{c,p}[0]} \vac \label{upomegadef}
\end{align}
where the polynomials $R$ are as we defined them in \cref{Rdef}. 

(Here and below we use summation convention, for brevity.)

The formal zero mode $\upomega\fm 0$ is realized in $\End(\Mh_0\tox\Wh_0^2)$ as the following infinite sum
\begin{align} \upomega\vap 0 &= -\int  \cee^{a,n\, \prime}(x) \cee^{b,m}(x) 
                 Y\left(\frac{\del R_{a,n}^{c,p}}{\del \gam^{d,q}[0]} 
                 \frac{\del R_{b,m}^{d,q}}{\del \gam^{c,p}[0]} \vac,x\right)dx\nn\\
&\qquad- \int \cee^{a,n}(x) \cee^{b,m}(x) 
          Y\left(\left[T,\frac{\del R_{a,n}^{c,p}}{\del \gam^{d,q}[0]}\right] 
                 \frac{\del R_{b,m}^{d,q}}{\del \gam^{c,p}[0]} \vac,x\right)dx\nn\\
&= -\sum_{K,L\in \ZZ} K \cee^{a,n}[-K] \cee^{b,m}[-L] 
                 \left(\frac{\del R_{a,n}^{c,p}}{\del \gam^{d,q}[0]} 
                 \frac{\del R_{b,m}^{d,q}}{\del \gam^{c,p}[0]} \vac\right)[K+L-1]\nn\\
&\qquad- \sum_{K,L\in \ZZ} \cee^{a,n}[-K] \cee^{b,m}[-L] 
          \left(\left[T,\frac{\del R_{a,n}^{c,p}}{\del \gam^{d,q}[0]}\right] 
                 \frac{\del R_{b,m}^{d,q}}{\del \gam^{c,p}[0]} \vac\right)[K+L].\nn
\end{align}
At the same time, from \cref{omegadef} we have 
\begin{multline} \wt\vfd^*(\omega)(J_{a,n}\fm K,J_{b,m}\fm L) \\
 = -K \left(\frac{\del R_{a,n}^{c,p}}{\del \gam^{d,q}[0]} 
                 \frac{\del R_{b,m}^{d,q}}{\del \gam^{c,p}[0]} \vac\right)\fm{K+L-1}
 - \left(\left[T,\frac{\del R_{a,n}^{c,p}}{\del \gam^{d,q}[0]}\right] 
                 \frac{\del R_{b,m}^{d,q}}{\del \gam^{c,p}[0]} \vac\right)\fm{K+L}.\nn
\end{multline}
We see that $\upomega\fm 0$ is identified with our cocycle $\wt\vfd^*(\omega)$, and so our cocycle belongs to the local complex. 

\subsection{Submodules $\Mh_0(\np)$ and $\Mh_0(\nm)$}
To proceed, we need more information about the structure of $\Mh_0$ as an $\LpLog$-module. 
Recall that
\begin{align}
\Onp &= \CC[X^{a,n}]_{\ii a n\in \A},\nn\\
\Onm &= \CC[X^{a,n}]_{\ii a n\in \Am}\nn
\end{align}
and that our action of $\g$ on $\Og$ stabilizes both of these, as in \cref{stabprop}. 
It follows that if we now define 
\begin{align} 
\Mh_0(\np) &:=   \CC[\gam^{a,n}[-N]]_{\ii a n \in \A,N\geq 0}\vac,\nn\\
\Mh_0(\nm) &:=   \CC[\gam^{a,n}[-N]]_{\ii a n \in \Am,N\geq 0}\vac,\nn
\end{align}
then our action of $\LpLog$ on $\Mh_0$ (given by $J_{a,n}[N] v = \JJ_{a,n} \vap N v$, as in \cref{sec: CE}) stabilizes these subspaces (in fact, commutative vertex subalgebras) of $\Mh_0$. (Cf. \cref{Lpgstab}.) 

As modules over $\LpLog$, these subspaces turn out to be isomorphic to contragredient Verma modules, as we now describe.

\subsection{Contragredient Verma modules}



The \emph{contragredient Verma module} $M_\lambda^*$ over $\g$ of highest weight $\lambda\in \h^*$ is by definition the coinduced left $U(\g)$-module
\be M_\lambda^* = \Coind_{\b_-}^\g \CC v_\lambda := \Homres_{U(\b_-)}(U(\g),\CC v_\lambda), \nn\ee
where $\CC v_\lambda$ denotes the one-dimensional $U(\b_-)$-module defined by $\nm\on v_\lambda = 0$ and $h\on v_\lambda = \lambda(h) v_\lambda$ for $h\in \h$.
Here $\Homres$ means the following: we have the isomorphism of vector spaces
\begin{align} \Hom_{U(\b_-)}(U(\g),\CC v_\lambda) &\cong \Hom_{U(\b_-)}(U(\b_-)\ox U(\np),\CC v_\lambda)\nn\\ &\cong \Hom_{\CC}(U(\np),\CC)= U(\np)^*, \nn
\end{align}
and $\Homres$ means we allow only maps that, under this isomorphism, belong to the \emph{restricted dual}
$U(\np)^\vee := \bigoplus_{\alpha\in Q} (U(\np)_\alpha)^* \subset U(\np)^*$
of the $Q$-graded vector space $U(\np)$. 
The isomorphism above is also one  of left $U(\np)$ modules. So, as left $U(\np)$-modules,
\be M_\lambda^* \cong U(\np)^\vee. \label{mlu}\ee 

We also have the contragredient Verma modules ``in the opposite category $\scr O$'', i.e. the twists of the modules above by the Cartan involution $\cai$ of \cref{caidef}. 

Define $M_\lambda^{*,\cai}$ to be the coinduced left $U(\g)$-module
\be M_\lambda^{*,\cai} = \Coind_{\b_+}^\g \CC v^-_\lambda := \Homres_{U(\b_+)}(U(\g),\CC v^-_\lambda). \nn\ee
where $\CC v^-_\lambda$ denotes the one-dimensional $U(\b_+)$-module defined by $\np\on v^-_\lambda = 0$ and $h\on v^-_\lambda = \lambda(h) v^-_\lambda$ for $h\in \h$.

As vector spaces, and as modules over $U(\nm)$, we have
\be M_\lambda^{*,\cai} \cong U(\nm)^\vee. \nn\ee

These definitions go over to the half loop algebra $L_+ \g$ in an obvious way: $\CC v_\lambda$ becomes a module over $L_+ \b_-$ if we declare that $(\b_- \ox t\CC[[t]])\on v_\lambda = 0$ and then we get the left $U(L_+ \g)$-module $\Homres_{U(L_+\b_-)}(U(\LpLog),\CC v_\lambda)$, and likewise its twist by $\cai$. 

\begin{prop}\label{asgmods} $ $
\begin{enumerate}[(i)]
\item There are isomorphisms of $\g$-modules
\begin{align}
    \Onp &\cong \Homres_{U(\b_-)}(U(\g),\CC v_0),\nn\\
    \Onm &\cong \Homres_{U(\b_+)}(U(\g),\CC v_0) .\nn
\end{align}
\item There are isomorphisms of $L_+\g$-modules
\begin{align}
    \Mh_0(\np) &\cong \Homres_{U(L_+\b_-)}(U(L_+\g),\CC v_0) \nn\\
    \Mh_0(\nm) &\cong \Homres_{U(L_+\b_+)}(U(L_+\g),\CC v_0).\nn
\end{align}
\end{enumerate}
\end{prop}
\begin{proof}
The proof is the same as in the case of $\g$ of finite type in \cite[\S5.2.3,\S5.6.3]{Fre07}. For completeness let us go through the steps.

First we show that $\Onp \cong U(\np)^\vee$ as $\np$-modules.
To do that we consider the pairing $U(\np) \times \Onp \mapsto \CC; (x,P) \mapsto \la x,P\ra := \left.x\on P\right|_{0\in \np}$. It respects the $Q$-gradations of $U(\np)$ and $\Onp$, in the sense that $U(\np)_\alpha$ pairs as zero with $\Onp_{\beta}$ unless $\alpha+\beta = 0$. Consider the restriction of the pairing to $U(\np)_\alpha \times \Onp_{-\alpha}$ for some $\alpha\in Q_{>0}$. Recall our ordered basis $\B_+ = \{J_{a,n}\}_{ \ii a n\in \A}$ of $\np$ from \cref{Bplus} and \cref{Bpdef}. We have the PBW basis of $U(\np)_\alpha$ consisting of ordered monomials these basis elements, and we have the basis of $\Onp_{-\alpha}$ consisting of monomials in the $X^{a,n}$, $\ii a n\in \A$. Both these bases have the lexicographical ordering coming from the ordering of $\A$. The action of $J_{a,n}$, $(a,n)\in A$, on $\Onp$ is by a differential operator of the form
\be D_{a,n} + \sum_{\substack{(b,m) \in \A\\ \wgt(J_{b,m}) -\wgt(J_{a,n}) \in Q_{>0}}} P^{b,m}_{a,n}(X)D_{b,m} \label{Jact}\ee  
From this one sees that the matrix of the restricted pairing with respect to the two ordered bases above is diagonal with non-zero entries. Thus the restriction of the pairing to $U(\np)_\alpha \times \Onp_{-\alpha}$ is non-degenerate, for each $\alpha\in Q_{>0}$. This shows that $\Onp \cong U(\np)^\vee$ as a vector space. But the pairing is also manifestly $\np$-invariant: $\la x e_i, P\ra = \la x, e_i P\ra$. This shows that $\Onp \cong U(\np)^\vee$ as left modules over  $U(\np)$,  where the left $U(\np)$-module structure on $U(\np)^\vee$ is the canonical one, coming from the right  action of $U(\np)$ on itself by right multiplication.  

Thus, given \cref{mlu}, we have $\Onp \cong M_0^*$ as $\np$-modules. Now we show it is an isomorphism of $\g$-modules. 
The coinduced module $M_\lambda^*$ has the following universal property. Suppose $M$ is a $\g$-module and $N\subset M$ a $\b_-$-submodule of $M$ such that the quotient $M/N$ is isomorphic to $\CC v_\lambda$ as a $\b_-$-module. Then there is a homomorphism of $\g$-modules $M\to M_\lambda^*$ sending $v\mapsto v_\lambda^*$, where $v\in M$ is such that $v+N$ spans $M/N$, and where $v_\lambda^*\in M_\lambda^*$ is a non-zero vector of weight $\lambda$. In our case, as a $\b_-$-module, $\Onp$ has the submodule $N = \bigoplus_{\alpha\in Q_{>0}} \Onp_{-\alpha}$ (or equivalently, the ideal in $\Onp$ generated by $(X^{a,n})_{(a,n)\in\A}$). The quotient $\Onp/N$ is a $\b_-$-module of dimension one, spanned by the class $1+N$ of the vector $1\in \Onp$. This vector $1+N$ has weight zero and is annihilated by $\nm$ (since $\nm\on 1 \in N$). Hence there exists a homomorphism of $\g$-modules $\phi: \Onp \to M_0^*$ sending $1$ to a non-zero vector $v_0^*\in M_0^*$ of weight zero. Now, for any $P\in \Onp$ there exists $x\in U(\np)$ such that $x\on P = 1$: indeed, take the last nonzero monomial $m$ of $P$ with respect to the lexicographical ordering and consider the corresponding PBW basis element $m^*$ of $U(\np)$. We see that $m^*\on m$ is a nonzero multiple of $1$. Thus $x\on \phi(P) = \phi(x\on P) = \phi(1) = v_0^*$ is nonzero and hence $\phi(P)$ is also nonzero. That is, $\phi: \Onp \to M_0^*$ is injective. But we know $\Onp \cong M_0^*$ as an $\np$-module, as above, so in fact $\phi$ must be a bijection. This completes the proof that $\Onp \cong M_0^* \equiv \Homres_{U(\b_-)}(U(\g),\CC v_0)$ as $\g$-modules. 

The argument for $ \Onm \cong \Homres_{U(\b_+)}(U(\g),\CC v_0)$ is the same, just twisted by the Cartan involution $\cai$ so that $\Onp$ and $\Onm$, and $U(\np)$ and $U(\nm)$, are interchanged. (Compare \cref{vfddef}.)

For part (ii), the argument is again essentially the same, with the $Q$-gradation above replaced by the $Q\times \ZZ_{\geq 0}$ gradation. One shows first that $\Mh_0(\np) \cong U(L_+\np)^\vee$ as $L_+\np$-modules, and then uses that fact to show that the canonical homomorphism of $L_+\g$-modules $\Mh_0(\np)\to \Homres_{U(L_+\b_-)}(U(L_+\g),\CC v_0)$ is an isomorphism.
\end{proof}

Therefore there is an isomorphism of modules over $\g'/\CC\cent \cong L\oc\g:=\Log$,
\begin{align}
    \Onp &\cong \Homres_{U(\b'_-/\CC\cent)}(U(\g'/\CC\cent),\CC v_0) \nn
\end{align}
(where $\b'_- := \g' \cap \b_-$) and an isomorphism of modules over $L_+(\g'/\CC\cent) \cong \LpLog$,
\begin{align}
    \Mh_0(\np) &\cong \Homres_{U(L_+(\b'_-/\CC\cent))}(U(L_+(\g'/\CC\cent)),\CC v_0).\nn
\end{align}

Given an  $r$-cochain $\extp^r \LpLog \to \Mh_0(\np)$, we may restrict it to $\extp^r L_+\oc\h$ and then compose the resulting map $\extp^r L_+\oc\h \to \Mh_0(\np)$ with the canonical projection of $\LpLog$-modules $\Mh_0(\np) \to \CC v_0$. This defines a map of complexes
\be  \mu: C^\bl(\LpLog, \Mh_0(\np)) \to C^\bl(L_+\oc\h,\CC v_0). \nn\ee
\begin{lem}\label{lem: quasi-isom}
This map $\mu$ is a quasi-isomorphism, i.e. it induces an isomorphism of the cohomologies,
\be H^\bl(\LpLog,\Mh_0(\np)) \cong H^\bl(L_+\oc\h, \CC v_0).\nn\ee
\end{lem}
\begin{proof}
The proof, using the Serre-Hochschild spectral sequence (see \cite[\S1.5]{Fuks}), is the same as in \cite[Lemma 5.6.6]{Fre07}.
\end{proof}

The following is \cite[Lemma 5.6.7]{Fre07}. 
\begin{lem}\label{keylem} 
If the restriction of a cocycle $\gamma\in \Cloc^r(\LLog, \L(\Mh_0(\np)))$ to $\extp^rL\oc\h$ is zero, then $\gamma$ represents the zero cohomology class, $[\gamma] = [0]$, in $\Hloc^r(\LLog,\L(\Mh_0(\np)))$. 
\end{lem}
\begin{proof}
We can suppose $r\geq 1$. 

Consider a cocycle $\gamma \in \Cloc^r(\LLog, \L(\Mh_0(\np)))$. We have $\gamma=X[0]$ for some cochain $X \in C^r(\LpLog, \Mh_0(\np))$. The closure of $\gamma$ implies $dX$ is in the image of $T$: $dX = TY$, say, for some $Y\in C^{r+1}(\LpLog,\Mh_0(\np))$. We have $TdY = -dTY = -ddX = 0$ and since $T$ has kernel $0$ (for all $r\geq 1$), that implies $dY=0$.

Let $\ol\gamma$ denote the restriction of $\gamma$ to $\extp^r L \oc\h$. We have $\ol\gamma = \ol X[0]$, where $\ol X$ denotes the restriction of $X$ to $\extp^rL_+\oc\h$. If $\ol\gamma$ is zero then $\ol X$ is in the image of $T$. Therefore so too is $\mu(X)$. So we have $\mu(X) = Th$, say, for some $h\in C^r(L_+\oc\h, \CC v_0)$. Now, $\mu$ is a map of complexes, so $dTh = d\mu(X) = \mu(dX) = \mu(TY)$. It is clear that $T$ commutes with $\mu$. Thus $-Tdh = T\mu(Y)$ and hence, again since the kernel of $T$ is trivial, $dh = -\mu(Y)$. That is, $\mu(Y)$ is exact. But $\mu$ is a quasi-isomorphism as in \cref{lem: quasi-isom}. So $Y$ must also be exact: $Y = dB$, say, for some $B\in C^r(\LpLog, \Mh_0(\np))$. 

Let $X' = X+TB$. We see that $\gamma= X'[0]$ and $X'$ is a cocycle: $dX' = dX - TdB = dX - TY = 0$. 

(At this point, effectively we have shown we were at liberty to assume our $X$ -- now called $X'$ -- was not only a cochain, but a cocycle. We now repeat many steps from above, but armed with that extra fact.)

We have $\mu(X') = Th'$ (where $h'= h+\mu(B)$). So $-Tdh'= dTh' = d\mu(X') = \mu(dX') = 0$, and hence $dh'=0$. So $h'$ is a cocycle in $C^r(L_+\oc\h,\CC v_0) = \left(\extp^r L_+ \oc\h\right)^\vee$. Thus, again since $\mu$ is a quasi-isomorphism, we have $h'= \mu(B')$ for some cocycle $B'\in C^r(\LpLog,\Mh_0(\np))$. 
Finally, we see that $\mu(X') = T\mu(B') = \mu(TB')$. Hence the cocycles $X'$ and $TB'$ in $C^r(\LpLog,\Mh_0(\np))$ represent the same cohomology class in $H^r(\LpLog,\Mh_0(\np))$. Therefore the cocycles $X'[0]$ and $(TB')[0]$ in $\Cloc^r(\LpLog, \Mh_0(\np))$ represent the same cohomology class in $\Hloc^r(L_+ \g, \Mh_0(\np))$. (Indeed $X' - TB' = dC$ for some $C\in C^r(\LpLog,\Mh_0(\np))$ implies $X'[0] - (TB')[0] = (dC)[0] = d(C[0])$.) But $\gamma = X'[0]$ and $0 = (TB')[0]$, so we have shown $\gamma$ is cohomologous to zero, as required.  
\end{proof}

This is the key lemma. However, to use it, we have to get around one final obstacle: our cochain $\upomega[0]$ lives not in $\Cloc^2(\LLog, \L(\Mh_0(\np)))$ but only in the larger space $\Cloc^2(\LLog, \L(\Mh_0))$.

\subsection{Cohomology equivariant with respect to $\tau$}
Recall the involutive automorphism $\tau$ defined in \cref{sec: cai}. Let us also denote by $\tau$ the involutive automorphism of $\Wh\ox \Mfw$ defined by $\tau\vac = \vac$,
\begin{alignat}{2} 
\tau(\gam^{\ia \alpha n})[N] &= \gam^{\ia{-\alpha}{-n}}[N],\quad& 
\tau(\gam^{\ia i n})[N]      &= - \gam^{\ia i {-n}}[N],\nn\\
\tau(\bet_{\ia \alpha n})[N] &= \bet_{\ia{-\alpha}{-n}}[N],\quad& 
\tau(\bet_{\ia i n})[N]      &= - \bet_{\ia i {-n}}[N],\nn
\end{alignat}
\begin{alignat}{2} 
\tau(\cee^{\ia \alpha n})[N] &= \cee^{\ia{-\alpha}{-n}}[N],\quad& 
\tau(\cee^{\ia i n})[N]      &= - \cee^{\ia i {-n}}[N],\nn\\
\tau(\bee_{\ia \alpha n})[N] &= \bee_{\ia{-\alpha}{-n}}[N],\quad& 
\tau(\bee_{\ia i n})[N]      &= - \bee_{\ia i {-n}}[N],\nn
\end{alignat}
\begin{alignat}{2} 
\tau(\SS^\alpha_{\beta,n}[N]) &= \SS^{-\alpha}_{-\beta,-n}[N], \quad&
\tau(\SS^\alpha_{i,n}[N]) &= - \SS^{-\alpha}_{i,-n}[N], \nn\\
\tau(\SS^i_{\beta,n}[N]) &= - \SS^{i}_{-\beta,-n}[N], \quad&
\tau(\SS^i_{j,n}[N]) &= \SS^{i}_{j,-n}[N], \nn
\end{alignat}
for $\alpha,\beta\in \oc\Delta\setminus\{0\}$, $i,j\in \oc I$, $n\in \ZZ$, and $N\in \ZZ$, and $\tau(\Cocent[N]) = - \Cocent[N]$.

Let $\Cloc^{\tau,\bl}(\LLog,\L(\Mh_0)) = \L_0\left( \left(\Wh_0^\bl\ox \Mh_0\right)^\tau\right)$ denote the subspace consisting of zero modes of states $\Phi\in \Wh_0\ox\Mh_0$ such that $\tau\Phi = 0$. 


\begin{lem}
$\tau\Q = \Q$ and hence $\Cloc^{\tau,\bl}(\LLog,\L(\Mh_0))$ is a subcomplex of $\Cloc^{\bl}(\LLog,\L(\Mh_0))$.
\end{lem}
\begin{proof}
The term $\half f_{ab}{}^c \cee^{a,m}[0] \cee^{b,n}[0] \bee_{c,n+m}[-1] \vac$ in $\Q$ is $\tau$-invariant because $\cai$ is an automorphism of $\Log$. (More explicitly, this term is equal to 
\begin{multline}  
\sum_{\alpha\in \oc\Delta_+} \sum_{i\in \oc I} \sum_{n,m\in \ZZ}
 f_{\alpha,-\alpha}{}^i \cee^{\alpha,m}[0] \cee^{-\alpha,n}[0] \bee_{i,n+m}[-1] \vac\nn\\
 +  \sum_{\alpha\in \oc\Delta_+} \sum_{i\in \oc I} \sum_{n,m\in \ZZ}  
\left( f_{\alpha i}{}^\alpha \cee^{\alpha,m}[0] \cee^{i,n}[0] \bee_{\alpha,n+m}[-1] \vac
 +     f_{-\alpha, i}{}^{-\alpha} \cee^{-\alpha,m}[0] \cee^{i,n}[0] \bee_{-\alpha,n+m}[-1] \vac\right)
\nn\\
 +   \half \sum_{i,j,k\in \oc I} \sum_{n,m\in \ZZ} f_{ij}{}^k \cee^{i,m}[0] \cee^{j,n}[0] \bee_{k,n+m}[-1] \vac \nn
\end{multline}
and each line of this expression is $\tau$-invariant.) 
The other term in $\Q$, $\sum_{(a,n)\in \Ag}\cee^{a,n}[0] \JJ_{a,n}$, is $\tau$-invariant because, in view of \cref{vfdequi}, 
\be \tau \JJ_{a,n} = \tau \jota(\vfd(J_{a,n})) = \jota(\tau(\vfd(J_{a,n}))) = \jota(\vfd(\cai J_{a,n}))\nn\ee 
and thus
\be \tau \JJ_{\alpha,n} = \JJ_{-\alpha,-n}, \quad \tau \JJ_{i,n} = - \JJ_{i,-n} \nn\ee
for $\alpha\in \oc\Delta\setminus\{0\}$, $i\in\oc I$ and $n\in \ZZ$. 
\end{proof}

\begin{lem}
The element $\upomega \in \Wh_0^2 \ox \Mh_0$ obeys
\be \tau \upomega =0 .\nn\ee 
\end{lem}
\begin{proof}
Indeed, we defined $\upomega$ as in \cref{upomegadef} but one sees (in view of \cref{ABope}) that, equivalently, 
\be \upomega = \cee^{a,n}[0] \cee^{b,m}[0] 
  \left( \JJ_{a,n}\vap 0 \JJ_{b,m} 
        - f_{ab}{}^c \JJ_{c,n+m} \right)
+ \cee^{a,n}[-1] \cee^{b,m}[0] \JJ_{a,n}\vap 1 \JJ_{b,m}
\nn\ee
(summation convention). The fact that $\tau\upomega=0$ follows, using the statements above and the fact that $\tau$ is an automorphism for all the non-negative products. 
\end{proof}

Thus our cocycle $\upomega[0]$ belongs to the subcomplex $\Cloc^{\tau,\bl}(\LLog,\L(\Mh_0))$.
More is true. We have the subspace $\Mh_0(\np) + \Mh_0(\nm)$ of $\Mh_0$. It is closed (trivially) under all the non-negative products. Therefore $\L(\Mh_0)$ has the Lie subalgebra $\L(\Mh_0(\np) + \Mh_0(\nm))$.  

\begin{lem}
The cocycle $\upomega[0]$ has coefficients in this subalgebra, i.e 
\be \upomega[0] \in \Cloc^{\tau,2}(\LLog,\L(\Mh_0(\np)+ \Mh_0(\nm))).\nn\ee
\end{lem}
\begin{proof}
We use summation convention. By definition, \cref{upomegadef},
\be \upomega =
- \cee^{a,n}[-1] \cee^{b,m}[0] 
                 \frac{\del R_{a,n}^{c,p}}{\del \gam^{d,q}[0]} 
                 \frac{\del R_{b,m}^{d,q}}{\del \gam^{c,p}[0]} \vac
- \cee^{a,n}[0] \cee^{b,m}[0] 
          \left[T,\frac{\del R_{a,n}^{c,p}}{\del \gam^{d,q}[0]}\right] 
                 \frac{\del R_{b,m}^{d,q}}{\del \gam^{c,p}[0]} \vac \nn\ee
where $R_{a,n}^{b,m}(X)$ are the polynomials from \cref{Rdef}. 
On recalling \cref{Rpdef} and \cref{Jdif}, we see that for every $\ii a n\in \Ag$, we have that
\be R_{a,n}^{b,m}(X) = A_{a,n}^{b,m}(X) + B_{a,n}^{b,m}(X) \nn\ee
where
\be  A_{a,n}^{b,m}(X) \in \CC[X^{c,p}]_{\ii c p \in \pm \A}, \quad 
     B_{a,n}^{b,m}(X) \in\bigoplus_{\ii c p \in \mp \A} \CC X^{c,p},\quad\text{for all $\ii bm \in \pm \A$.} \nn\ee
Now we shall argue that
\begin{align} 
\upomega &=
- \cee^{a,n}[-1] \cee^{b,m}[0] 
           \left(      \frac{\del A_{a,n}^{c,p}}{\del \gam^{d,q}[0]} 
                       \frac{\del A_{b,m}^{d,q}}{\del \gam^{c,p}[0]} \vac
                  +    \frac{\del B_{a,n}^{c,p}}{\del \gam^{d,q}[0]} 
                       \frac{\del B_{b,m}^{d,q}}{\del \gam^{c,p}[0]} \vac \right)\nn\\ 
& \qquad \qquad - \cee^{a,n}[0] \cee^{b,m}[0] 
          \left[T,\frac{\del A_{a,n}^{c,p}}{\del \gam^{d,q}[0]}\right] 
                 \frac{\del A_{b,m}^{d,q}}{\del \gam^{c,p}[0]} \vac. \label{ABeq}
\end{align}
Indeed, we see all $A$-$B$ cross terms are zero just by inspecting the index contractions. The remaining term is 
$ - \cee^{a,n}[0] \cee^{b,m}[0] 
          \left[T,\frac{\del B_{a,n}^{c,p}}{\del \gam^{d,q}[0]}\right] 
                 \frac{\del B_{b,m}^{d,q}}{\del \gam^{c,p}[0]} \vac$ but this is zero by the linearity of the $B_{a,n}^{b,m}$ and the fact that $[T,1]=0$.
So we have the equality \cref{ABeq}. The term $\frac{\del B_{a,n}^{c,p}}{\del \gam^{d,q}[0]} 
                       \frac{\del B_{b,m}^{d,q}}{\del \gam^{c,p}[0]} \vac$ is proportional to $\vac$, again by the linearity of $B_{a,n}^{b,m}$. And by definition of the $A_{a,n}^{b,m}$, the $A$-$A$ terms all belong to $\Wh_0^2\ox(\Mh_0(\np) + \Mh_0(\nm))$. So we have established that
\be \upomega \in \Wh_0^2\ox(\Mh_0(\np) + \Mh_0(\nm)) \nn\ee
and hence the result. 
\end{proof} 

\begin{lem}
There is an isomorphism of complexes
\be \Cloc^{\tau,\bl}(\LLog,\L(\Mh_0(\np)+\Mh_0(\nm))) \cong \Cloc^\bl(\LLog,\L(\Mh_0(\np)))\nn\ee
\end{lem}
\begin{proof}
There is certainly a linear isomorphism
\begin{align} \Wh_0 \ox \Mh_0(\np) &\isom \Wh_0 \ox \left( \Mh_0(\np) + \Mh_0(\nm)\right)\nn\\
    v &\mapsto (1+\tau) v,\nn
\end{align}
and hence a linear isomorphism $\L_0(\Wh_0 \ox \Mh_0(\np)) \cong \L_0(\Wh_0 \ox \left( \Mh_0(\np) + \Mh_0(\nm)\right))$. 
We have $(1+\tau)(\Q\vap 0 v) = \Q\vap 0 v + (\tau\Q) \vap 0 \tau v = \Q\vap 0 v + \Q\vap 0 \tau v = \Q \vap0 (1+\tau) v$, so this isomorphism commutes with the differential. 
\end{proof}

Finally we can complete the proof of \cref{mainthm}. 

We have shown that our cocycle $\wt\vfd^*(\omega) = \upomega[0]$ belongs to $\Cloc^{\tau,2}(\LLog,\L(\Mh_0(\np)+\Mh_0(\nm)))$ and therefore corresponds to a cocycle $\zeta \in \Cloc^2(\LLog,\L(\Mh_0(\np)))$. By the key lemma, \cref{keylem}, such a cocycle is cohomologous to zero if its restriction to $L\oc\h$ vanishes. The restriction of $\zeta$ to $L\oc\h$ is zero if the restriction of $\wt\vfd^*(\omega)$ to $L\oc\h$ is zero. And the restriction of $\wt\vfd^*(\omega)$ to $L\oc\h$ is indeed zero because, for all $i\in \oc I$, $\vfd(J_{i,0}) = \J_{i,0}$ (and hence $R_{i,0}^{b,n}(X) = 0$).

Thus there exists a 1-cochain $\xi \in \Cloc^1(\LLog,\L(\Mh_0(\np)))$ such that $\zeta = d\xi$. It may be written in the form
\be \xi = \Xi\fm 0,\quad\text{with}\quad \Xi := \cee^{a,n}\fm 0 \sum_{(b,m) \in \A} Q_{a,n;b,m}(\gam\fm 0) \gam^{b,m}\fm{-1}\vac\nn\ee
for some polynomials $Q_{a,n;b,m}(X) \in \Onp$, $(b,m)\in \A$, such that $\Xi$ has $Q$-grade $0$. 

In this way we obtain a 1-cochain, $(1+\tau)\xi \in \Cloc^{\tau,1}(\LLog,\L(\Mh_0(\np) + \Mh_0(\nm))$, such that $\wt\vfd^*(\omega) = d(1+\tau)\xi$, as required. 

\section{Proof of \cref{bthm}}\label{sec: proofb}
In \cref{sec: inft} we studied the infinitesimal right action of $\gc$ on the right coset space $U_0(\CC_\eps):= \exxp{\eps\b_-} \!\Big\backslash\! \exxp{\f_\eps}$. In the same way we may consider the right action of $\gc$ on the right coset space $\exxp{\eps\n_-} \!\Big\backslash\! \exxp{\f_\eps}$. We get the following analog of \cref{Plem} and \cref{vfdef}.


\begin{lem}
Let $A\in \gc$. 
Then there exist polynomials $\left\{P_A^{\ia b m}(X)\in \Onp\right\}_{{\ii b m}\in \A}$ and $\{p^i_A(X) \in \Onp\}_{i\in I}$ (depending linearly on $A$) such that
\begin{align} &\Exp{\eps \nm} \left( \Exp{\sum_{i=1}^{\dim\h} y^i b_i} \prodr_{{\ii b m}\in\A} \Exp{x^{\ia b m} J_{\ia b m}}\right) \Exp{\eps A} \nn\\
&\qquad\qquad= \Exp{\eps \nm} \Exp{\sum_{i=1}^{\dim\h} \left( y^i + \eps p^i_A(x)\right) b_i}
 \prodr_{{\ii b m}\in\A} \Exp{\left(x^{\ia b m} + \eps P_A^{\ia b m}(x)\right)J_{\ia b m}}\nn
\end{align}
for every element $\Exp{\sum_{i=1}^{\dim\h} y^i b_i}\prodr_{{\ii b m}\in\A} \Exp{x^{\ia b m} J_{\ia b m}}$ of the group $B:= H\ltimes U$. 

Hence, the linear map 
\be \gc \to \Derc\Onp \ltimes \bigoplus_{i=1}^{\dim\h} \Onp \del_{Y^i} \nn\ee 
given by
\be A \mapsto \sum_{i=1}^{\dim\h} p^i_A(X) \del_{Y^i} + \sum_{{\ii b m}\in \A} P_A^{\ia b m}(X) D_{\ia b m} \nn\ee
is a homomorphism of Lie algebras.
It respects the $Q$-gradation (where we assign $Q$-grade zero to the generators $\del_{Y^i}$).
 \qed
\end{lem}

Explicitly, this homomorphism sends
\begin{align}
e_i \mapsto \vf(e_i), \nn\qquad
h  \mapsto \vf(h) + \sum_{j=1}^{\dim\h} \la b^j,h\ra \del_{Y^j},\nn\qquad
f_i  \mapsto \vf(f_i) + \sum_{j=1}^{\dim\h} \la b^j,\check\alpha_i\ra X^{e_{i}} \del_{Y^j}, \nn
\end{align}
where $X^{e_i} := X^i$ are as in \cref{Xidef}. Let $X^{f_i} := \tau X^{e_i}$. 
It follows that there is a homomorphism of Lie algebras (here, recall \cref{vfddef})
\be \vfd' : \g \to \Dg \ltimes \bigoplus_{i=1}^{\dim\h} \Og \del_{Y^i} \nn\ee
defined by 
\begin{align}
e_i &\mapsto \vfd(e_i)+ \sum_{j=1}^{\dim\h} \la b^j,\check\alpha_i\ra X^{f_{i}} \del_{Y^j}, \nn\\ 
h  &\mapsto \vfd(h) + \sum_{j=1}^{\dim\h} \la b^j,h\ra \del_{Y^j},\nn\\
f_i  &\mapsto \vfd(f_i) + \sum_{j=1}^{\dim\h} \la b^j,\check\alpha_i\ra X^{e_{i}} \del_{Y^j}, \nn
\end{align}
for $i\in I$ and $h\in \h$.
This yields a homomorphism of the loop algebras, just as in \cref{vfdhatdef},
\be L\g \to L\left(\Dg \ltimes \bigoplus_{i=1}^{\dim\h} \Og \del_{Y^i} \right). \nn\ee

Now we note that 
\be (\Mfw\ox\pi_0)[0] =_\CC \Og\quad\text{and}\quad(\Mfw\ox\pi_0)[1] =_\CC \Omega_\O \oplus \jota\left(\Dg \ltimes\bigoplus_{i=1}^{\dim\h} \Og \del_{Y^i} \right)\nn\ee
where we continue to identify $\Og$ and $\Omega_\Og$ with subspaces of $\Mfw = \Mfw\ox \CC\vac$ as before, and we extend the definition of the injective linear map $\jota$, \cref{jotadef}, by setting
\be \jota\left( \sum_{i=1}^{\dim\h} p^i(X) \del_{Y^i}\right) = \sum_{i=1}^{\dim\h} p^i(\gam\fm0) b_i\fm{-1}\vac.\nn\ee

As in \cref{LDextlem} and \cref{jotalem}, we then find that there is an isomorphism of Lie algebras
\be L\left(\Dg \ltimes \bigoplus_{i=1}^{\dim\h} \Og \del_{Y^i} \right)
\cong \L((\Mfw\ox \pi_0)\df1)\big/ \L(\Ogp) .\nn\ee
We get an exact sequence of Lie algebras
\be 0 \to \L(\Ogp)  \to \L((\Mfw\ox\pi_0)\df1) \to L\left(\Dg\ltimes \bigoplus_{i=1}^{\dim\h} \Og \del_{Y^i} \right) \to 0\ee 
As in the proof of \cref{jotalem}, the cocycle defining this extension is given by double contraction terms in the OPE. Following \cite{Fre07}, the key observation is then that there are no possible double contractions between the new terms we have added (which belong to the subspace $\O \ox \pi_0 \subset \Mfw \ox \pi_0$) and the existing terms (which each have at most one factor of $\bet$ or $\SS$). It follows that the statement of \cref{mainthm} still holds (with the same lifting map $\vphi$) when $\vfd$ is replaced by the map $\vfd'$ above.

\appendix

\section{The coefficients $c_i$}\label{sec: cis}
In this section, for $i\in I$, let us write $X^i := X^{\ia a n}$ for the unique $\ii a n$ such that $J_{\ia a n} = e_i$, i.e.
\be X^i := \begin{cases} X^{\alpha_i,0}, & i \in \oc I = I \setminus\{0\}\\
                         X^{-\delta+\alpha_0,1},  & i =0  \end{cases}\label{Xidef}\ee
and similarly $X^{[i,j]} := X^{\ia a n}$ for the unique $\ii a n$ such that $J^{\ia a n} \propto [e_i,e_j]$. Define $D_i$, $D_{[i,j]}$ likewise. Recall the homomorphism $\vf: \gc \to \Derc\Onp$ from \cref{vfdef}.

\begin{lem}\label{quadterms}
Let $i\in I$. The terms in $\vf(f_i)$ of the form  $X^i X^{\ia a n} D_{\ia a n}$ for some $\ii a n\in \A$ are  
\be  - X^i X^i D_i + \sum_{j\prec i} a_{ij} X^i X^{[i,j]} D_{[i,j]} 
                             - \sum_{j\prec i} a_{ij} X^i X^j D_j. \nn\ee
\end{lem}
\begin{proof}
This follows from a direct calculation, of the sort in the proof of \cref{Plem} and \cref{linthm}. Let us give the outline. We have 
\be \vf(f_i) = \sum_{\ii a n\in \A} P_{f_i}^{\ia a n}(X) D_{\ia a n}.\label{Pfi}\ee 
By inspection, one sees that if $P_{f_i}^{\ia a n}(X)$ is to have both $X^{\ia a n}$ and $X^i$ as factors, then it must be that $[f_i, J_{\ia a n}]$ is proportional to a basis vector that precedes $e_i$ in our basis. (We have to pick up the dependence on $X^{\ia a n}$ as $\exxp{\eps f_i}$ moves leftwards through the product, and then pick up the dependence on $X^i$ as some term is pushed through $\exxp{x^i e_i}$.)
This is a strong constraint: we must have either 
\begin{enumerate}
\item $J_{\ia a n} = e_j$ for some $j\in I$, or
\item $J_{\ia a n} \propto [e_i,e_j]$ for some $j\in I$ such that $e_j \prec e_i$.
\end{enumerate}
Let us compute the coefficients of the resulting terms. Consider case (2): suppose $e_j \prec e_i$ and $J_{\ia an}= c [e_i,e_j]$ for some nonzero $c\in \CC$, so that $[J_{\ia a n}, f_i] = c [[e_i,e_j],f_i] = c [\check \alpha_i,e_j] = ca_{ij} e_j$.
We have
\begin{align} &\exxp { x^{\ia a n} J_{\ia a n}  } \exxp{\eps f_i} \nn\\
&= \exxp{\eps f_i}\Bigl( \exxp{\eps x^{\ia a n} [J_{\ia a n} , f_i] } 
                   \exxp{\eps \tfrac 1 2 [J_{\ia a n},[J_{\ia a n},f_i]]}\dots \Bigr) 
                   \exxp{x^{\ia a n} J_{\ia a n}} \nn\\
&= \exxp{\eps f_i} \exxp{\eps x^{\ia a n} ca_{ij} e_j } 
                   \exxp{x^{\ia a n} J_{\ia a n}}\dots .\nn
\end{align}
Here $ \exxp{\eps x^{\ia a n} ca_{ij} e_j } $ still needs to be pushed left through the factor $\exxp{x^i e_i}$ which appears further left in the product. We get
\begin{align} \exxp{x^i e_i} \exxp{\eps x^{\ia a n} c a_{ij} e_j } 
&=  \exxp{\eps x^{\ia a n} ca_{ij} e_j } \exxp{\eps x^i x^{\ia a n} a_{ij} c [e_i,e_j]}   \exxp{x^i e_i}\dots.\nn
\end{align}
We obtain the terms $\sum_{j\prec i} a_{ij} X^i X^{[i,j]} D_{[i,j]}$ in $\vf(f_i)$. 

Now consider terms $J_{\ia a n} = e_j$ for some $j$. When $i\prec j$ we just get
\be \exxp {x^j e_j} \exxp{\eps f_i} = \exxp{\eps f_i} \exxp {x^j e_j} \nn\ee
since $[e_j,f_i] = 0$. Eventually we reach the factor $ \exxp {x^i e_i}$. We continue to shuffle terms, getting
\begin{align} \exxp {x^i e_i} \exxp{\eps f_i} &= \exxp{\eps f_i} \exxp{\eps x^i \check\alpha_i } \exxp{\eps \half (-2 x^i x^i e_i)} \exxp{x^i e_i} \nn\\
&= \exxp{\eps f_i} \exxp{\eps x^i \check\alpha_i } \exxp{\left(x^i - \eps x^ix^i\right) e_i}  \nn
\end{align}
and then finally we have to move $\exxp{\eps f_i} \exxp{\eps x^i \check\alpha_i }$ further left through factors  $\exxp {x^j e_j}$ with $j\prec i$:
\begin{align}  \exxp{x^j e_j}  \exxp{\eps f_i} \exxp{\eps x^i \check\alpha_i } 
   &=  \exxp{\eps f_i} \exxp{\eps x^i \check\alpha_i } \exxp{-\eps x^jx^i a_{ij} e_j  } \exxp{x^j e_j}\nn\\
   &=  \exxp{\eps f_i} \exxp{\eps x^i \check\alpha_i } \exxp{\left( x^j -\eps x^jx^i a_{ij}\right) e_j  }\nn 
\end{align}
From these last two expressions we read off the terms $-X^iX^i D_i$ and $-\sum_{j\prec i} a_{ij} X^i X^j D_j$ in $\vf(f_i)$. 
\end{proof}

For all $h\in \h$, we have
\be \vf(h) 
= -\sum_{\ii a n\in \A }
 \la \wgt\ii a n ,h \ra X^{\ia a n} D_{\ia a n}\label{hX}\ee
and hence $\vf(h)(X^i) = - \la \alpha_i, h\ra X^i$ 
(as it certainly should, on $Q$-grading grounds).

\begin{prop}
For all $h\in \h$ and for each $i\in I$, we have 
\begin{align}
 \lpg{\vf(h)}{\vf(e_i)} &= 0,\nn\\
 \lpg{\vf(h)}{\vf(f_i)} &= - c_i\vf(h)(X^i),\nn  
\label{vpr}
\end{align}
with $c_i$ as in \cref{CScor}.
\end{prop}
\begin{proof}
Recall the Wick lemma from \cref{sec: wick}. First products, like $\lpg{\vf(h)}{\vf(f_i)}$, involve a double contraction. In view of \cref{Pfi} and \cref{hX}, and then making use of the lemma above, we find
\begin{align}
  \lpg{\vf(h)}{\vf(f_i)} 
&= \sum_{\ii a n\in \A} \la \wgt\ii a n ,h \ra \left(D_{\ia a n} P_{f_i}^{\ia a n}(X)\right)\nn\\
&= - 2\la \alpha_i, h\ra X^i + \sum_{j\prec i} a_{ij} \la \alpha_i + \alpha_j, h\ra X^i
                            -  \sum_{j\prec i} a_{ij} \la \alpha_j, h\ra X^i\nn\\
&= \la \alpha_i, h\ra \left( -2  + \sum_{j\prec i} a_{ij}\right) X^i = - c_i\vf(h)(X^i) .  
\nn\end{align}
\end{proof}

Therefore
\be \jota(\vf(h)) \,\vap 1\, \left(\jota(\vf(f_i)) + c_i \gam^{e_i}[-1]\vac\right) =0 \nn\ee
for each $i\in I$. This shows that $\vphi(f_i) = c_i \gam^{e_i}[-1]\vac$ and, hence, $\vphi(e_i) = c_i\gam^{f_i}[-1]\vac$.

\printbibliography

\end{document}